\documentclass[11pt,reqno]{amsart}
\usepackage[left=0.9in,top=0.9in,right=0.9in,bottom=0.9in]{geometry}
\usepackage{bbm}

\usepackage{amsfonts,amsmath,amsthm,amssymb,latexsym,mathrsfs,stmaryrd}
\usepackage{hyperref}
\usepackage{mathtools}
\usepackage{graphicx}
\usepackage{subcaption}

\usepackage{tikz}\usetikzlibrary{cd,fit,matrix,arrows,decorations.pathmorphing}
\tikzset{commutative diagrams/.cd}

\usepackage{cleveref}

\numberwithin{equation}{section}

\newtheorem{theorem}{Theorem}[section]
\newtheorem{corollary}[theorem]{Corollary}
\newtheorem{lemma}[theorem]{Lemma}
\newtheorem{proposition}[theorem]{Proposition}

\theoremstyle{definition}
\newtheorem{definition}[theorem]{Definition}
\newtheorem{definition-theorem}[theorem]{Definition-Theorem}
\newtheorem{definition-lemma}[theorem]{Definition-Lemma}
\newtheorem{example}[theorem]{Example}

\newtheorem{notation}[theorem]{Notation}

\newtheorem{question}[theorem]{Question}

\newtheorem{remark}[theorem]{Remark}

\theoremstyle{remark}

\newtheorem*{remark*}{Remark}

\usepackage{enumitem}
\setenumerate[1]{label=\textup{(\alph*)}}
\setenumerate[2]{label=\textup{(\roman*)}}

\newcommand\Z{{\mathbb Z}}
\newcommand\N{{\mathbb N}}
\newcommand\Q{{\mathbb Q}}
\newcommand\C{{\mathbb C}}
\newcommand\A{{\mathbb A}}

\newcommand\p{{\mathfrak{p}}}
\newcommand\F{{\mathbb F}}
\newcommand\Fp{{\mathbb F}_p}
\newcommand\Fq{{\mathbb F}_q}

\newcommand\Zp{{{\mathbb Z}_p}}
\def\O{\mathcal O}

\DeclareMathOperator{\im}{im}
\DeclareMathOperator{\coker}{coker}

\DeclareMathOperator{\Hom}{Hom}
\DeclareMathOperator{\End}{End}
\DeclareMathOperator{\Aut}{Aut}
\DeclareMathOperator{\tr}{tr}
\newcommand{\red}{{\mathrm{red}}}
\DeclareMathOperator{\Frac}{Frac}
\DeclareMathOperator{\Spec}{Spec}

\DeclareMathOperator{\GL}{GL}

\DeclareMathOperator{\Mat}{Mat}
\DeclareMathOperator{\Sym}{Sym}
\newcommand\resp{{\textit{resp.~}}}
\newcommand\subeq{\subseteq}

\newcommand\supeq{\supseteq}

\newcommand\bbar{\overline} %\bar is too narrow; avoid
\newcommand\tl{\widetilde} 
\newcommand\ttilde{\widetilde} %\tilde is too narrow; avoid
\newcommand\hhat{\widehat} %\hat is too narrow; avoid

\newcommand\incl{\hookrightarrow}
\DeclarePairedDelimiter{\abs}{\lvert}{\rvert}
\DeclarePairedDelimiter{\norm}{\lVert}{\rVert}
\DeclarePairedDelimiter{\set}{\{}{\}}

\DeclarePairedDelimiter{\parens}{\lparen}{\rparen}
\DeclarePairedDelimiter{\bracks}{\lbrack}{\rbrack}
\DeclarePairedDelimiter\floor{\lfloor}{\rfloor}
\DeclareMathOperator{\Hilb}{Hilb}

\newcommand{\Zhat}{\hhat{Z}}
\newcommand{\zetahat}{\hhat{\zeta}}
\newcommand{\nuhat}{\hhat{\nu}}
\DeclareMathOperator{\Coh}{Coh}
\DeclareMathOperator{\Quot}{Quot}
\DeclareMathOperator{\rk}{rk}
\newcommand{\m}{\mathfrak{m}}
\DeclareMathOperator{\Span}{span}
\renewcommand{\L}{\mathbb{L}}

\newcommand{\cotimes}{\widehat{\otimes}}
\newcommand{\Ohat}{\hhat{\O}}

\DeclareMathOperator{\Surj}{Surj}

\DeclareMathOperator{\Gr}{Gr}

\DeclareMathOperator{\Cont}{Cont}
\newcommand{\KVar}[1]{K_0(\mathrm{Var}_{#1})}

\newcommand{\Nhat}{\widehat{N}}

\newcommand{\KStck}[1]{{K_0(\mathrm{Stck}_{#1})}}

\newcommand\fc{{\mathfrak c}}
\usepackage{accents}
\newcommand{\ut}{\undertilde}

\newcommand{\calE}{{\mathcal{E}}}
\newcommand{\calF}{{\mathcal{F}}}

\newcommand{\one}{\mathbf{1}}

\newcommand{\qbinom}[2]{\genfrac{[}{]}{0pt}{}{#1}{#2}}

\newcommand{\Pd}{\mathrm{Pd}}

\newcommand{\mat}[1]{\begin{matrix}#1\end{matrix}}
\newcommand{\bmat}[1]{\begin{bmatrix}#1\end{bmatrix}}

\newcommand{\AG}{\mathbf{AG}}
\newcommand{\Br}{\mathbf{Br}}

\begin{document}
\title[Coh and Quot zeta functions]{Motivic Coh and Quot zeta functions of singular curves}
\author{Yifeng Huang}
\address{Department of Mathematics, University of Southern California}
\email{yifeng.huang@usc.edu}
\thanks{}

\author{Ruofan Jiang}
\address{Department of Mathematics, University of California--Berkeley}
\email{ruofanjiang@berkeley.edu}
\thanks{}


\date{}

\date{\today}

\begin{abstract}
    We present a general and effective algebraic framework for enumerating finite-length quotients of a torsion-free sheaf of arbitrary rank (the Quot zeta function) and finite-length coherent sheaves (the Coh zeta function) over reduced singular curves. We prove that Quot zeta functions are motivically rational, using a novel parametrization and the geometry of affine Grassmannians, and that they satisfy an arbitrary-rank reflection principle, via harmonic analysis. We show that the a normalized high-rank limit of Quot zeta functions converges to the Coh zeta function. As a first application, we compute explicit formulas for these zeta functions for all $y^2 = x^n$ singularities, revealing a surprising and previously unknown connection to Rogers--Ramanujan type $q$-series. Further applications to affine Springer fibers and commuting varieties are also discussed.
\end{abstract}

\maketitle

\setcounter{tocdepth}{1}
\tableofcontents

\section{Introduction}

Let $k$ be an algebraically closed field, and let $\KVar{k}$ (\resp $\KStck{k}$) be the Grothendieck ring of $k$-varieties (\resp $k$-stacks), with $\L:=[\A^1_k]$ denoting the Lefschetz motive. In this paper, we study two types of motivic generating functions that enumerate geometric objects on a $k$-variety\footnote{A $k$-variety is a quasi-projective $k$-scheme, not necessarily reduced or irreducible.} $X$. The first object of study is $\Coh_n(X)$, the stack of coherent sheaves on $X$ that have $0$-dimensional support and length $n$. Its generating series is the \textbf{(motivic) Coh zeta function}
\begin{equation}
    \Zhat_X(t) := \sum_{n\geq 0} [\Coh_n(X)]\, t^n \in \KStck{k}[[t]].
\end{equation}
The second is $\Quot_n(\calE)$, the Quot scheme parametrizing $0$-dimensional length-$n$ quotients of a given coherent sheaf $\calE$. Its generating series is the \textbf{(motivic) Quot zeta function}:
\begin{equation}
    Z^X_\calE(t) = Z_\calE(t) := \sum_{n\geq 0} [\Quot_{n}(\calE)]\, t^n \in \KVar{k}[[t]].
\end{equation}
We use the terms rank $d$ Quot scheme (of points) and rank $d$ Quot zeta function to refer to the objects corresponding to $\calE=\O_X^d:=\O_X^{\oplus d}$. The rank $1$ Quot scheme of points is the Hilbert scheme of points: $\Quot_n(\O_X)=\Hilb_n(X)$.

These objects are of classical interest. The moduli spaces are directly connected to the study of quiver varieties, degree 0 Donaldson--Thomas theory, varieties of commuting matrices and moduli spaces of modules \cite{bbs2013motivic,cazzaniga2023higher,huang2023mutually,hos2023matrix,bbvx2023sato,jelisiejewsivic,moschettiricolfi2018}, and their motivic generating functions serve as geometric counterparts to several arithmetic zeta functions, such as Hey's order zeta function \cite{hey1927}, Solomon's lattice zeta function \cite{solomon1977zeta}, and Cohen and Lenstra's zeta function \cite{cohenlenstra1984heuristics}.

For smooth varieties of dimension one or two, these theories are well-understood. The table below summarizes the key results.

\begin{center}
\begin{tabular}{l|l|l}
\textbf{$X$} & \textbf{Object} & \textbf{Key references and remarks} \\
\hline
Smooth curve & $Z_{\O_X}(t)$ & \cite{kapranov2000}, $\Hilb_n(X)=\Sym^n(X)$ is smooth \\
& $Z_{\O_X^{\oplus d}}(t)$ & \cite{bfp2020motive,bifet1989}, rank $d$ Quot schemes are smooth \\
& $\Zhat_X(t)$ & classical, $[\Mat_n]=\L^{n^2}$ \\
\hline
Smooth surface & $Z_{\O_X}(t)$ & \cite{ellingsrudstromme1987homology,goettsche2001motive}, $\Hilb_n(X)$ is smooth\\
& $Z_{\O_X^{\oplus d}}(t)$ & \cite{mozgovoy2019}, rank $d>1$ Quot schemes are not smooth \\
& $\Zhat_X(t)$ & \cite{feitfine1960pairs,bryanmorrison2015motivic}, commuting variety\\
\end{tabular}
\end{center}

For smooth varieties of dimension 3 or above, computing these generating functions is an extremely difficult problem; however, see \cite{gmmr2024motive} in the opposite regime where the ``number of points'' $n$ is fixed and the dimension grows. For smooth 3-folds, the theory is well-developed in the sense of virtual motives \cite{bbs2013motivic,cazzaniga2023higher}. 

Our focus, however, is on the case of reduced singular curves, where the situation is more complex and reveals new phenomena. Even in the rank 1 case, the theory for singular curves is an extensive and active area of research. The Hilbert schemes of points on a planar singular curve over $\C$ is deeply connected to the topology of its associated link, leading to proven and conjectured relations with knot theory, geometric representation theory, and algebraic combinatorics, marked by the celebrated Oblomkov--Rasmussen--Shende conjecture (see \cite{gks2021link, ors2018homfly} and references therein).

In contrast, the higher-rank and unframed counterparts for singular curves remain largely uncharted lands. The study of the \textbf{unframed} version, $\Zhat_X(t)$, was pioneered by the first author in \cite{huang2023mutually}. Through the example of the simple node $xy=0$, that work revealed the surprising appearance of Rogers--Ramanujan type $q$-series, a phenomenon not seen in the rank 1 story. A subsequent work of the authors, \cite{huangjiang2023punctual}, established a crucial connection, demonstrating that the high-rank theory provides a systematic way to compute its unframed counterpart. Its initial computations for the simple cusp $y^2=x^3$
(in ranks $d\leq 3$) led to a conjectured formula in all ranks (and thus the unframed theory), confirming the persistence of the $q$-series phenomenon. This conjecture is proven in the present paper as a special case of our main computational results. However, the methods used in the prior work were difficult to extend, leaving the nature of the $q$-series connection vague.

The main contribution of this paper is a new and robust algebraic recipe to compute the motivic Quot zeta function for torsion-free bundles on reduced singular curves, overcoming the technical obstacles of previous methods. This framework is general, and as first applications, we establish several key results. First, we prove the rationality of the Quot zeta function over any reduced curve (\Cref{thm:rationality-intro}). Second, our method is explicit enough to yield the first complete formulas for the full family of $y^2=x^n$ singularities, for all ranks $d$ (\Cref{thm:torus-intro}). This proves the conjecture formulated in \cite{huangjiang2023punctual} (as the $n=3$ case) and reveals a rich combinatorial structure. Independent of these explicit computations, we prove a general reflection principle for the zeta functions of Gorenstein curve singularities on the point-count level (\Cref{thm:reflection-intro}). Finally, we present a simple and transparent limit formula (\Cref{thm:A}) that clarifies the fundamental connection between the high-rank and unframed theories found in \cite{huangjiang2023punctual}.

\subsection{Statements of results}
We now state these results more precisely. Before presenting our main computational theorems, we first revisit the connection between the Quot and Coh zeta functions. While this relationship was established in a more complex form in \cite{huangjiang2023punctual}, we present a new, more transparent formulation as a simple limit. This fundamental result establishes the high-rank theory as a direct computational path to the unframed theory.

\begin{theorem}[$\subeq$ \Cref{thm:A-effective}]\label{thm:A}
    Let $k$ be an algebraically closed field, and $X$ be any variety over $k$. The motivic Coh zeta function is the rank $d\to\infty$ limit of the rank $d$ Quot zeta function with a suitable rescaling:
    \begin{equation}\label{eq:thmA}
        \Zhat_X(t)=\lim_{d\to \infty} Z^X_{\O_X^d}(\L^{-d}t) \in \KStck{k}[[t]],
    \end{equation}
    where the convergence is coefficient-wise and in the sense of the dimension filtration on $\KStck{k}$ \cite{behrenddhillon2007}.
\end{theorem}

Since both $\Coh_n(X)$ and $\Quot_n^X(\calE)$ parametrize sheaves with zero-dimensional support, the motivic theories $\Zhat_X(t)$ and $Z_\calE^X(t)$ are of local nature; see \cite{fantechiricolfi2024structural,fantechiricolfi2024motivic} and a precise statement (\Cref{thm:local-to-global}). This justifies formulating our main results in terms of the punctual setting. Let $R$ be any finitely generated algebra over $k$, or a completion thereof; so $X=\Spec R$ is an affine $k$-variety or a \textbf{formal $k$-variety} (i.e., a germ of a $k$-variety). Let $\Coh_n(R)$ be the stack of $R$-modules that are $n$-dimensional (as a $k$-vector space), and
\begin{equation}
    \Zhat_R(t):=\sum_{n\geq 0} [\Coh_n(R)]\, t^n \in \KStck{k}[[t]].
\end{equation}
For a finitely generated $R$-module $E$, let $\Quot_n(E)$ be the Quot scheme parametrizing $n$-codimensional $R$-submodules of $E$, and
\begin{equation}
    Z^R_E(t)=Z_E(t):=\sum_{n\geq 0} [\Quot_n(E)]\, t^n \in \KVar{k}[[t]].
\end{equation}
If $R$ is complete, we also call $\Zhat_R(t)$ and $Z^R_E(t)$ the \textbf{punctual} Coh and Quot zeta functions.

Our new recipe for computing these series leads to the following general rationality result. The denominator is expressed using the $q$-Pochhammer symbol, defined by
\begin{equation}
    (a;q)_n:=(1-a)(1-aq)\cdots (1-aq^{n-1}) \text{ for }n\in \Z_{\geq 0}\cup \set{\infty}.
\end{equation}

\begin{theorem}[$\subeq$ \Cref{thm:rationality-motivic}]\label{thm:rationality-intro}
    Let $R$ be the germ of a reduced curve over an algebraically closed field $k$, with branching number $b(R)$. For any $d \ge 0$, the motivic Quot zeta function $Z_{R^d}^R(t)$ is rational. More precisely, the \textbf{normalized} rank $d$ Quot zeta function
    \begin{equation}
        N^R_{R^d}(t):=(t;\L)_d^{b(R)} Z_{R^d}^R(t)
    \end{equation}
    is a polynomial in $t$ with coefficients in $\KVar{k}$.
\end{theorem}

The role that the branching number plays here is best understood via the geometric, global version of this statement.
\begin{corollary}[$\subeq$ \Cref{thm:rationality-intro,thm:local-to-global}]
    Let $k$ be an algebraically closed field of characteristic zero and $X$ be a reduced curve over $k$. Let $\tl X\to X$ be its resolution of singularities. Then the quotient power series
    \begin{equation}
        N^X_{\O_X^d}(t):=\frac{Z^X_{\O_X^d}(t)}{Z^{\tl X}_{\O_{\tl X}^d}(t)}
    \end{equation}
    is in fact a polynomial in $\KVar{k}[t]$.
\end{corollary}

Our method is constructive, allowing for the explicit computation of these polynomials for specific families of singularities. We focus on the planar singularities corresponding to the $(2,n)$ torus knots and links. For $m \ge 1$, we define the germ of the \textbf{cusp} singularity ($(2,2m+1)$-torus knot) by
\begin{equation}
    R_{2,2m+1} := k[[X,Y]]/(Y^2-X^{2m+1}),
\end{equation}
and the \textbf{nodal} singularity ($(2,2m)$-torus link) by
\begin{equation}
    R_{2,2m} := k[[X,Y]]/(Y(Y-X^m)).
\end{equation}
When the characteristic of $k$ is not $2$, $R_{2,2m}$ is isomorphic to $k[[X,Y]]/(Y^2-X^{2m})$.

\begin{theorem}[$=$ \Cref{thm:cusp_simplification,thm:node-full}, and Remark~\ref{rmk:motivic-torus}]\label{thm:torus-intro}
    For $m\ge 1, d\geq 0$, and $R=R_{2,2m+1}$ or $R_{2,2m}$, the normalized Quot zeta function $N_{R^d}^R(t)$ is given by an explicit polynomial in $\Z[\L,t]$. 
\end{theorem}

Finally, independent of the explicit computations above, we prove a general symmetry that these polynomials must satisfy on the level of point counts. A curve is \textbf{Gorenstein} if its dualizing sheaf is a line bundle. The Gorenstein-ness of a curve depends only on its singularities, but not on its global geometry. Any planar curve singularity is Gorenstein. The \textbf{Serre invariant} of a curve singularity $R$ is $\dim_k \tl R/R$, where $R\incl \tl R$ is the normalization.

\begin{theorem}[$\subeq$ \Cref{thm:reflection-local}]\label{thm:reflection-intro}
    Let $R$ be a Gorenstein curve singularity over $\Fq$ with Serre invariant $\delta$. Then the normalized rank $d$ Quot zeta function satisfies the reflection principle
    \begin{equation}
        N^R_{R^d}(t) \overset{q}{=} \L^{d^2 \delta} t^{2d\delta} N^R_{R^d}(\L^{-d} t^{-1}),
    \end{equation}
    where $\overset{q}{=}$ means equality up to counting $\Fq$-points.
\end{theorem}

\subsection{Relation to $q$-series}
The explicit polynomials described in \Cref{thm:torus-intro} are not merely computational artifacts; on the contrary, they reveal brand new connections to classical $q$-series. The precise nature of this connection was not immediately clear when the present paper was first posted; the full picture, which we present here with the benefit of hindsight, required a pivotal simplification by a subsequent work of Chern \cite{shane2024multiple} and a recent observation by the first author \cite{huang2025coh}.

The central picture is that the normalized Quot zeta functions for the $y^2=x^n$ singularities are themselves $t$-deformations of celebrated multi-sums of the Rogers--Ramanujan type. To state this precisely, we define the following finitizations of the Andrews--Gordon and Bressoud multi-sums:
\begin{align}
    \AG_d(q,t;2m+3)&:=(q;q)_d \sum_{n_1,\dots,n_m} \frac{q^{\sum n_i^2}t^{2\sum n_i}}{(q;q)_{d-n_1}(q;q)_{n_1-n_2}\cdots (q;q)_{n_m}},\\
    \Br_d(q,t;2m+2)&:=(q;q)_d \sum_{n_1,\dots,n_m} \frac{q^{\sum n_i^2} t^{2\sum n_i}}{(q;q)_{d-n_1}(q;q)_{n_1-n_2}\cdots (q;q)_{n_m}(-tq;q)_{n_m}}.
\end{align}
The sums naturally truncate by the convention $1/(q;q)_n:=0$ for $n < 0$. To frame the motivic results in this language, we use the \textbf{rank $d$ finitized Coh zeta function},
\begin{equation}
    \Zhat_{R,d}(t):=Z_{R^d}^R(\L^{-d} t) \in \KStck{k}[[t]],
\end{equation}
a rescaling of the rank $d$ Quot zeta function whose $d\to \infty$ limit recovers $\Zhat_R(t)$ by \Cref{thm:A}. The following theorem summarizes the connection.

\begin{theorem}[{$\subeq$ Theorems~\ref{thm:cusp_simplification}, \ref{thm:node-full}, and \cite[Theorem~1.7]{shane2024multiple}}]
The rank $d$ finitized Coh zeta functions of $R_{2,2m+1}$ and $R_{2,2m}$ are given by:
\begin{align}
    \Zhat_{R_{2,2m+1},d}(t) &= \frac{1}{(t\L^{-1};\L^{-1})_d} \AG_d(\L^{-1},t;2m+3), \\
    \Zhat_{R_{2,2m},d}(t) &= \frac{1}{(t\L^{-1};\L^{-1})_d} \Br_d(\L^{-1},-t;2m+2).
\end{align}
\end{theorem}

The contribution of the present paper to this picture is twofold. The formula for the cusp case is proven completely herein (\Cref{thm:cusp_simplification}). For the nodal case, this paper provides the initial, un-simplified formula (\Cref{thm:node-full}), whose equivalence to the Bressoud sum above requires the additional $q$-series identity from \cite{shane2024multiple}. As crucial initial evidence, however, we provide a self-contained proof of the $t=1$ specialization (\Cref{thm:node_special_value}). This result amounts to a new Rogers--Ramanujan type identity whose proof relies purely on our geometric framework, without using any external $q$-series techniques.

Finally, we remark on the combinatorial nature of these series. Their $d\to \infty$ limits, which correspond to the Coh zeta function, have Rogers--Ramanujan type infinite product expressions at $t=\pm 1$:
\begin{align}
    \AG_\infty(q,\pm 1;2m+3) &= \frac{(q^{m+1},q^{m+2},q^{2m+3};q^{2m+3})_\infty}{(q;q)_\infty},\\
    \Br_\infty(q,1;2m+2) &= \frac{(q^{m+1},q^{m+1},q^{2m+2};q^{2m+2})_\infty}{(q;q)_\infty},\\
    \Br_\infty(q,-1;2m+2) &= \frac{1}{(q;q)_\infty},
\end{align}
where $(a_1,\dots,a_r;q)_n:=(a_1;q)_n\cdots (a_r;q)_n$. The \textbf{modulus} $2m+3$ or $2m+2$ appearing in the infinite product motivates our notation of these series. This raises a natural, albeit speculative, question:
\begin{question}
    For the germ of a general planar singularity $R$ over $k=\C$, if $\Zhat_R(1)$ converges and is a power series in $\L^{-1}$, is it a Rogers--Ramanujan type infinite product? If so, what determines its modulus?
\end{question}
Ongoing work of the authors and Alexei Oblomkov suggests that the modulus for the $(a,b)$-torus knot singularity should be $a+b$, extending the pattern found here. This correspondence appears to be new in the context of connections between knot theory and $q$-series; for example, in \cite[Eq.~(3-37)]{hikami2003volume}, the $(2,2m+1)$-torus knot corresponds to modulus $2m+1$ instead of $2m+3$.

While the $t$-deformations themselves remain mysterious, the reflection principle (\Cref{thm:reflection-intro}) imposes a strong constraint on their structure, implying functional equations of $q,t$-series
\begin{align}
    \AG_d(q,t;2m+3)&=t^{2dm} q^{d^2m}\AG_d(q,t^{-1}q^{-d};2m+3),\\
    \Br_d(q,t;2m+2)&=t^{2dm-d-1} q^{d^2m-\binom{d}{2}} \frac{1+t}{1+tq^d} \Br_d(q,t^{-1}q^{-d};2m+2).
\end{align}
The one for the Bressoud sum is nontrivial and was proven directly in \cite[Theorem 1.8]{shane2024multiple}. The above discussion leads to a far-reaching goal for this research program:
\begin{question}
    What is the nature of the $q,t$-series that arise as Coh and Quot zeta functions of planar singular curves? Do they belong to a natural space of functions, and how do they relate to the theory of modular forms and theta functions?
\end{question}

\subsection{Relation to other works}
Our results connect to several other active areas of research beyond what has been discussed.

A central theme in the study of singular curves is to compare the moduli spaces on the curve to those on its normalization. A recent, influential conjecture in this direction is the Hilb-vs-Quot conjecture \cite{kivinentrinh2023}, which, for the rank 1 case, proposes a precise relationship between the Hilbert scheme of the singularity $\Hilb_n(R)$ and the Quot scheme of its normalization $\Quot_n(\tl R)$; more precisely, $N^R_R(t)=N^R_{\tl R}(t)|_{\L \mapsto \L t}$ for those $R$ such that $\Hilb_n(R)$ and $\Quot_n^R(\tl R)$ admit affine pavings for all $n$. Our result for the cusp singularity (\Cref{thm:cusp_simplification}) reveals a new, strikingly different phenomenon that holds for all ranks $d$: the zeta functions for $R^d$ and $\tl R^d$ are related by the simple substitution $N^R_{R^d}(t)=N^R_{\tl R^d}(t^2)$. This rule and the rank-1 rule happen to coincide for the $(2,2m+1)$ torus knot in rank 1, but they do not appear to have a common generalization. The situation for the nodal singularity is even more mysterious, as there is no obvious rule relating the zeta functions for $R^d$ and $\tl R^d$, perhaps suggesting the need for an extra variable. Our work establishes initial general connections between the zeta functions of $R^d$ and $\tl R^d$, see \Cref{cor:central-motivic} and \Cref{prop:s-zero-specialization}. In view of the method of the present work and \cite{huang2025coh}, understanding the precise relation between the Quot zeta functions for $R^d$ and $\tl R^d$ is not only of intrinsic interest, but is also instrumental for the computation, as $Z_{\tl R^d}^R(t)$ is often more computationally approachable. 

The study of Hilbert and Quot schemes over planar singular curves is also intimately related to orbital integrals and affine Springer fibers \cite{yun2013orbital,kivinentsai2022}. A recent breakthrough of Kivinen and Tsai \cite{kivinentsai2022} on affine Springer fibers for tamely ramified regular semisimple elements implies that the Hilbert zeta function of a planar curve singularity has polynomial point-counts over finite fields. However, the Quot zeta function of rank $d>1$ corresponds to affine Springer fibers of elements of the form $\gamma^{\oplus d}$ (a block diagonal matrix with $d$ identical diagonal blocks), which are highly non-regular. This demonstrates that our framework is in the opposite regime of, and thus not covered by, the work of \cite{kivinentsai2022}. It remains an open problem whether the Quot zeta function of any planar singularity has polynomial point-counts. 

For arbitrary singular curve germs $R$ (not necessarily planar or Gorenstein), not many quantitive results were known. While the motivic rationality of $Z_{\tl R}^R(t)$ can be viewed as an incarnation of the existence of the (local) compactified Jacobian as a scheme, the rationality of $Z_R^R(t)$ in this generality was not known until recently proven in \cite{brv2020motivic}; the subtlety partly owes to the containment condition. Our \Cref{thm:rationality-intro} can be viewed as a higher-rank generalization of \cite{brv2020motivic}. In view of the techniques in both works, the main difficulty we overcome in higher rank stems from the fact that $\Quot_n^{\tl R}(\tl R^d)$ is not a finite set of points, unless $d=1$. Therefore, the extension map $L\mapsto \tl RL$ that lands in $\Quot_\bullet^{\tl R}(\tl R^d)$ no longer stratifies the source into a finite disjoint union of fibers.

Finally, our results have implications in several other fields. First, our computation of $\Zhat_R(t)$ completely solves the problem of counting commuting matrix pairs $(A,B)$ over finite fields satisfying the equation $A^2=B^{2m+1}$ or $A(A-B^m)=0$ by giving a generating function of the flavor of Feit--Fine \cite{feitfine1960pairs}. This is also equivalent to the problem of counting finite-length modules over the corresponding coordinate rings, with each isomorphism class weighted inversely by the size of the automorphism group. Second, the limit formula in \Cref{thm:A} (or rather, \Cref{thm:A-effective}) can be viewed as a new example of the motivic stabilization phenomenon in the spirit of Vakil and Wood \cite{vakilwood2015discriminants}. Third, our computational results confirm that the main conjecture of the first author \cite{huang2023mutually}, which states that the point-count version of $\Nhat_X(t):=\Zhat_X(t)/\Zhat_{\tl X}(t)$ for a reduced curve $X$ has an infinite radius of convergence in $t$, holds for the $y^2=x^n$ singularity. Finally, our results show that for the motivic degree zero Donaldson--Thomas theory of singular curves, the higher-rank theory does not factorize in terms of the rank 1 theory (unless we specialize to the Euler characteristic; see \Cref{prop:chi-power}), but is in fact more subtle and interesting.

\subsection{Outline of methods and plan of the paper}
The main technical contribution of this paper is a new method for the explicit computation of motivic and arithmetic \emph{lattice zeta functions} for \emph{local orders}, particularly in higher ranks. Our work thereby advances the intrinsic study of lattice zeta functions, a field of study that is deeply rooted in the classical Dedekind zeta function and order zeta function \cite{hey1927}, pioneered by Solomon \cite{solomon1977zeta}, and vastly enriched by Bushnell, Reiner, and Yun \cite{bushnellreiner1980zeta,bushnellreiner1981functional,bushnellreiner1981L,yun2013orbital}. While the rationality of the arithmetic lattice zeta functions was known via the harmonic analysis approach of Bushnell and Reiner \cite{bushnellreiner1980zeta}, that method has not, to our knowledge, previously yielded explicit multi-sum formulas for these series in rank $d>1$. Our approach, based on a novel parametrization of lattices, overcomes this obstacle and is naturally suited for a motivic upgrade. The previous best explicit results for the families of singularities studied here were due to Saikia \cite{saikia1988}, but only in rank 1.

The paper is organized as follows. In \Cref{sec:fundamental}, we review the structure of the relevant moduli spaces. In \Cref{sec:quot-to-coh}, we establish a simple but fundamental limit formula (\Cref{thm:A}) that expresses the unframed Coh zeta function as a high-rank limit of the Quot zeta function. The proof is based on a robust observation about the abundance of surjections from high-rank modules, which, while elementary, provides a transparent connection between these two theories. While geometric connections between finite-framing-rank theory and unframed theory have been a well-explored theme in quiver varieties (see for example \cite{jelisiejewsivic}), our precise motivic relation appears to be new.

The core of our new computational method is developed in \Cref{sec:lattice,sec:rationality,sec:motivic}. After developing our precise setup for local orders and lattices in \Cref{sec:lattice}, we prove in \Cref{sec:rationality} the rationality of the arithmetic lattice zeta function (\Cref{thm:rationality-arithmetic}) by introducing a new stratification of sublattices via an intermediate ``boundary lattice". More precisely, the key observation (\Cref{lem:padding}) and the use of the ``extension fiber'', defined by $E_R(\tl L):=\set{L\subeq_R \tl L: \tl R\cdot L=\tl L}$ for a lattice $\tl L$ over the normalization $\tl R$, reduce the problem of counting all sublattices to counting lattices within a bounded colength. In \Cref{sec:motivic}, we show that this parametrization-based approach is well-suited for a motivic upgrade, allowing us to prove the rationality of the motivic Quot zeta function (\Cref{thm:rationality-motivic}) using the geometry of affine Grassmannians. Crucially, these rationality theorems are formulated for arbitrary torsion-free modules, not just for $R^d$ as required in \Cref{thm:rationality-intro}. This level of generality is essential for our methods, as the most computationally accessible case is often the lattice $\tl R^d$, rather than $R^d$ itself. These ideas are further pursued and formalized in a much more general context in Appendix~\ref{sec:AGreduced}, where we introduce the notion of affine Grassmannians for arbitrary reduced curve germs and study ``constructible morphisms'' $\underline{\pi_*}, \underline{\vec{\pi}^*}$ and $ \underline{\vec{\pi}^!}$ that encode lattice operations that arise from the six functor formalism. 

In \Cref{sec:func-eq}, we prove a reflection principle (\Cref{thm:reflection-local}) for the arithmetic zeta function of the dualizing module over \emph{any} reduced curve singularities. The methods here are independent of the rest of the paper and follow the classical approach of harmonic analysis and Tate's thesis, extending the rank-1 results of \cite{yun2013orbital} to the higher-rank setting.

Finally, in \Cref{sec:prelim,sec:torus,sec:combo}, we apply this machinery. After collecting the necessary combinatorial preliminaries on $q$-series and Hall polynomials (\Cref{thm:hall_skew}), we derive the first explicit formulas for the Quot zeta functions of the $y^2=x^n$ singularities (\Cref{prop:cusp-tl,prop:node-tl,prop:cusp-full}, and \Cref{thm:node-full}). We then analyze these formulas to prove our main combinatorial theorems, including the simplification for the cusp case (\Cref{thm:cusp_simplification}) and the special value for the nodal case (\Cref{thm:node_special_value}), revealing their connections to Rogers--Ramanujan type identities.

\subsection*{Acknowledgements}
We thank Dima Arinkin for some important ideas towards \Cref{sec:rationality,sec:motivic}, and S.~Ole Warnaar for proving \Cref{lem:cusp_squaring}. The authors thank Minh-Tam Trinh for feedback on earlier drafts. We thank Jim Bryan, Barbara Fantechi, Asvin G, Lothar G\"{o}ttsche, Eugene Gorsky, Nathan Kaplan, Mikhail Mazin, Leonardo Mihalcea, Alexei Oblomkov, Ken Ono, and Dragos Oprea for fruitful discussions. The first author was partially supported by the AMS-Simons Travel Grant. The first author acknowledges the use of the AI model Gemini 2.5 Pro in improving the exposition and organization of this paper. All mathematical ideas and results presented herein are the authors' own. 

\section*{Notation and terminology}
For the reader's convenience, we collect the main notation and terminology used throughout the paper. We clarify a convention used in our constructions: while the base scheme $X$ or $\Spec R$ may be non-reduced (notably in \Cref{sec:quot-to-coh}), all moduli spaces we consider are constructed as locally closed subschemes of projective spaces. We do not distinguish these moduli spaces from their underlying reduced structures, as this treatment is sufficient for determining their motives.

\begin{tabular}{@{}p{0.13\textwidth}p{0.82\textwidth}}
$(a;q)_n,\;\qbinom{n}{k}_q$ & $q$-Pochhammer symbol and $q$-binomial coefficient. See \Cref{subsec:hypergeom}.\\
$g^{\lambda}_{\mu\nu}, \; g^\lambda_\mu$ & (Skew) Hall polynomial. See \Cref{def:hall-poly} and \Cref{thm:hall_skew}.
\end{tabular}

\begin{tabular}{@{}p{0.13\textwidth}p{0.82\textwidth}}
$k$ & An algebraically closed field of arbitrary characteristic, or a finite field. \\
$d$ & A non-negative integer representing the rank.\\
$R$ & A local order, typically the germ of a reduced curve. See \Cref{def:local-order}.\\
$R_{2,n}$ & The specific planar singularity for $y^2=x^n$. See \Cref{sec:torus}. \\
$\tl R$ & The normalization of $R$. See \Cref{def:tot-frac}. \\
$\tl R_i, K_i$ & A branch of $\tl R, K$. See \Cref{def:tot-frac}.\\
$e_i,\pi_i,\kappa_i$ & The idempotent, the uniformizer, and the residue field of $\tl R_i$. \\
$K$ & The total ring of fractions of $R$. See \Cref{def:tot-frac}. \\
$b(R)$ & The branching number of $R$. See \Cref{sec:lattice}. \\
$\fc$ & The conductor of $R$. See \Cref{def:conductor}.\\
$L,M,N$, etc. & $R$-lattice: A finitely generated $R$-submodule of $K^d$ whose $K$-span is $K^d$. See \Cref{subsec:lattice}.\\
$\tl L, \tl M, \tl N$, etc. & $\tl R$-lattices.\\
$E$ & torsion-free $R$-module: An abstract $R$-module isomorphic to an $R$-lattice. See \Cref{subsec:lattice}.\\
$E^d, \O_X^d$ & The direct sum $E^{\oplus d}, \O_X^{\oplus d}$.\\
$[L_1:L_2]_k$ & The $k$-colength between two lattices. See \Cref{lem:colength}. \\
$(L_1:L_2)$ & The relative index between two lattices. See \Cref{lem:colength}.\\
$[[\tl L_1:\tl L_2]]$ & The vector of branch-wise colengths between $\tl R$-lattices. See \eqref{eq:colength-vector}\\
$\boldsymbol{\rk}(Q)$ & The vector of branch-wise ranks of an $\tl R$-module $Q$. See \eqref{eq:rank-vector}.\\
$\abs{\boldsymbol{n}}$ & The sum of coordinates of any tuple of integers. \\
$\Omega_R$ & The dualizing module of $R$. See \Cref{sec:func-eq}. \\
$\delta$ & The additive Serre invariant, $[\tl R:R]_k$. \\
$\Delta$ & The multiplicative Serre invariant, $\abs{\tl R/R}$. \\
\end{tabular}

\begin{tabular}{@{}p{0.13\textwidth}p{0.82\textwidth}}
$\ut M$ & A fixed $\tl R$-lattice contained in $M$.\\
$E_R(\tl L)$ & The extension fiber $\set{L\subeq_R \tl L:\tl R L=\tl L}$. \\
$B_R(M;\ut M)$ & The boundary locus $\set{L\subeq_R M:\tl R L\supeq \ut M}$.\\
$L_b$ & Boundary lattice $L_b:=\tl R L\cap M$, where $L\in B_R(M;\ut M)$. See \Cref{subsec:boundary}.
\end{tabular}

\begin{tabular}{@{}p{0.13\textwidth}p{0.82\textwidth}}
$Z_E(t)$ & The motivic Quot zeta function, cf.~\eqref{def:motivic-lattice-zeta}. \\
$\zeta_E(s)$ & The arithmetic lattice zeta function. See \eqref{def:lattice-zeta}.\\
$\Zhat_R(t)$ & The motivic Coh zeta function of a ring $R$. See \eqref{eq:motivic-coh}.\\
$\zetahat_R(s)$ & The arithmetic Coh zeta function of a ring $R$. See \eqref{eq:arithmetic-coh}.\\
$\Zhat_{R,d}(t)$ & The rank $d$ finitized motivic Coh zeta function, $Z_{R^d}(\L^{-d}t)$. \\
$\zetahat_{R,d}(s)$ & The rank $d$ finitized arithmetic Coh zeta function, $\zeta_{R^d}(s+d)$. \\
$N_E(t)$ & The normalized motivic Quot zeta function, $Z_E(t)/Z_{\tl R^d}(t)$. \\
$\nu_E(s)$ & The normalized arithmetic lattice zeta function, $\zeta_E(s)/\zeta_{\tl R^d}(s)$. \\
$\Nhat_{R,d}(t)$ & The normalized finitized motivic Coh zeta function, $\Zhat_{R,d}(t)/\Zhat_{\tl R,d}(t)$. \\
$\nuhat_{R,d}(s)$ & The normalized finitized arithmetic Coh zeta function, $\zetahat_{R,d}(s)/\zetahat_{\tl R,d}(s)$. \\
\end{tabular}

\section{Quot schemes and commuting varieties}\label{sec:fundamental}  In this section, we start with the definitions of Quot schemes and commuting varieties for classical varieties (i.e., finite type $k$-schemes), and then march into the realm of formal varieties. Throughout the section, ket $k$ be an arbitrary field.

\subsection{Classical varieties}\label{subsec:classical} Let $X$ be a finite type scheme over $k$ and let $E$ be a coherent $\mathcal{O}_X$-module. Define the \textbf{Quot scheme} $\Quot_n^X(E)$ (in short, $\Quot_n(E)$) as a functor $\mathbf{Sch}_k^{\mathrm{op}}\rightarrow \mathbf{Set}$ that sends $S$ to $\mathcal{O}_{X_S}$-submodules of $\pi^* E$ whose quotients are locally free of rank $n$ over $S$; here and in the next paragraph, $X_S:=S\times_k X$ and $\pi:X_S\to X$ is the structural morphism. It is classically known that $\Quot_n(E)$ is representable by a projective scheme over $k$.

For $n\geq 0$, let $C_n(X)$ be the \textbf{commuting variety} over $X$ whose $S$-points are coherent sheaves $\calF$ on $X_S$ together with an isomorphism of $\O_S$-modules $\iota:\O_S^n\to \pi_* \calF$; see also \cite[Appendix]{huangjiang2023punctual} and \cite[Theorem~3.5]{fantechiricolfi2024structural}. On the level of $k$-points, $C_n(X)$ parametrizes points $\calF\in \Coh_n(X)$ together with an ordered $k$-basis of $H^0(X;\calF)$. More formally,
\begin{equation}
    C_n(X)(k)=\set{(\calF,\iota):\calF\in \Coh_n(X), \; \iota:k^n\overset{\simeq}{\to} H^0(X;\calF)}.
\end{equation}
The group $\GL_n$ acts on $C_n(X)$ via its natural action on $k^n$. If $X=\Spec R$ is affine, then $C_n(R)=C_n(X)$ equivalently parametrizes $k$-algebra homomorphisms $\rho: R \to \Mat_n(k)$, and the action of $\GL_n$ is given by conjugation on $\Mat_n$. 

Let $\Coh_n(X)$ be the stack of length $n$ coherent sheaves on $X$, and it can be realized as a global quotient $\Coh_n(X)=[C_n(X)/\GL_n]$; see \cite[Theorem~3.5]{fantechiricolfi2024structural}. Since $C_n(X)$ is a finite type scheme, $\Coh_n(X)$ is a finite type Artin stack. Moreover, since $\GL_n$ is a special group, we have $[\Coh_n(X)]=[C_n(X)]/[\GL_n]\in \KStck{k}$ by \cite{behrenddhillon2007}. %Define $C_n(X)$. Write the definition here, i.e., those in 3.1, and mention that these objects are well studied and finite type.  Define $\Coh_n(X)$ as unframed version of $C_n(X)$. 

\subsection{Formal varieties}\label{sub:indscheme} Let $k$ be a field, and $R$ be a completion of a finitely generated $k$-algebra with maximal ideal $\m$, and let $E$ be a coherent $R$-module. In this section, we define $\Quot^R_n(E), C_n(R)$ and $\Coh_n(R)$, essentially following the treatment of \cite[\S 2.2.1]{fantechiricolfi2024motivic}. More precisely, we will define them as ind-objects via taking limits over truncations, which is natural since $R$ is not of finite type over $k$. If $X$ is a finite type scheme and $p$ is a closed point of $X$ such that $\Ohat_{X,p}=R$, then the \textbf{punctual} (or \textbf{nilpotent}) commuting variety of $X$ supported at $p$, as well as the commuting variety of the formal variety $(X,p)$, is defined as $C_n(X,p):=C_n(R)$; similarly for $\Quot^{X,p}_n(E)$ and $\Coh_n(X,p)$.

Recall that an ind-scheme $Y$ is a functor, which can be written as a union of subfunctors $Y=\bigcup_i Y_i$, such that each $Y_i$ is represented by a scheme, and $Y_i\hookrightarrow Y_{i+1}$ is a closed immersion. If $\mathcal{P}$ is a property of schemes, e.g., projective, smooth, etc. Then an ind-scheme is said to be ind-$\mathcal{P}$, if every $Y_i$ has $\mathcal{P}$. 

\begin{definition}\label{def:indindind}
  Let $n \ge 0$. For any $r\geq 1$, consider $\Quot^{R/\m^r}_n(E/\m^r E)$ in the classical sense of \Cref{subsec:classical}, as $R/\m^r$ is of finite type over $k$. It has a natural closed embedding into $\Quot^{R/\m^{r+1}}_n(E/\m^{r+1} E)$. We then define $\Quot_n(E) :=  \bigcup_r\Quot^{R/\m^r}_n(E/\m^r E)$. Similarly, we define $C_n(R)$ as $\bigcup_r C_n(R/\m^r)$, and $\Coh_n(R)=\bigcup_r \Coh_n(R/\m^r)=[C_n(R)/\GL_n]$.
\end{definition}

By definition, $\Quot_n(E)$ and $C_n(R)$ are ind-finite type in-schemes, while $\Coh_n(R)$ is an ind-finite type ind-stack. Moreover, $\Quot_n(E)$ is ind-projective. However, their structures are in general very wild, even for the punctual Hilbert scheme of one point on an affine line! Consider $\Quot_1(E)$, with $R=k[[T]]=E$ and $k=\C$. While $\Quot_1(E)$ has only one $k$-point corresponding to $(T)E$, the scheme $\Quot_1(E)$ has complicated $S$-points when $S$ is non-reduced: $\Quot_1(E)(S)$ parametrizes Laurent tails of nilpotent elements in $S$ \cite[\S 2.4, Remark 2.1]{kivinentrinh2023}. 

However, the reduced structure of $\Quot_n(E)$ is manageable, for example, $\Quot_1(E)_{\red}=\Spec k$. The key observation is the following.
\begin{lemma}[Artinian reduction]\label{lem:artinian-reduction-geom}
Let $F$ be a field, $A$ be a Noetherian $F$-algebra, and let $M$ be a finitely generated $A$-module. If $M' \subeq M$ is an $A$-submodule of $F$-codimension $n$, then $M'$ must contain the submodule $J(A)^n M$, where $J(A)$ is the Jacobson radical of $A$.
\end{lemma}
\begin{proof}
Consider the decreasing chain $J(A)^i (M/M')$, $i\geq 0$ of submodules of $M/M'$. By Nakayama's lemma, it must be strictly decreasing until reaching zero. Since $F$ is a field and $\dim_F M/M'=n$, the length of any strictly decreasing chain in $M/M'$ cannot exceed $n$. As a result, $J(A)^n(M/M')=0$, namely, $J(A)^n M \subeq M'$.
\end{proof}

\begin{lemma}
    The closed embedding $\Quot^{R/\m^r}_n(E/\m^r E)\incl \Quot^{R/\m^{r+1}}_n(E/\m^{r+1}E)$ induces an isomorphism on the reduced structure if $r\geq n$. 
\end{lemma}
\begin{proof}
   %Since both sides are finite type $k$-schemes, 
   It suffices to verify that the closed embedding induces a bijection on field-valued points.

    Fix $r\geq n$ and a field $F$ containing $k$.  Let $A=F\otimes_k R/\m^{r+1}$. Then $\m A$ is contained in the nilradical of $A$, and thus contained in the Jacobson radical of $A$.
    
    Let $M=F\otimes_k E/\m^{r+1} E$. For any $1\leq s\leq r+1$, the $F$-points of $\Quot^{R/\m^{s}}_n(E/\m^{s} E)$ are $A$-submodules $M'$ of $M$ that (1) have $F$-codimension $n$ and (2) contain $(\m A)^s M$. However, for any $s\geq n$, it follows from \Cref{lem:artinian-reduction-geom} that the  condition (2) is implied by (1). This shows that the closed embedding induces a bijection on $F$-points.
\end{proof}

The above lemma, and the analogues for $C_n$ and $\Coh_n$, imply that
\begin{equation}
    \Quot_n(E)_\red = \Quot_n^{R/\m^n}(E/\m^n)_\red, \;C_n(R)_\red = C_n(R/\m^n)_\red, \text{ and } \Coh_n(R)_\red = \Coh_n(R/\m^n)_\red.
\end{equation}

\begin{remark}
    Let $R=k[[T]]$ and $S$ be any $k$-algebra, then $S$-points of $C_n(R)$ are nilpotent $n\times n$ matrices over $S$, while $S$-points of $C_n(R/\m^n)$ are $n\times n$ matrices over $S$ whose $n$-th power vanishes. These sets are different if $S$ is not reduced, even if $n=1$. 
\end{remark}

\begin{remark}
    In \cite[\S 2.2.1]{fantechiricolfi2024motivic}, it is stated that $\Coh_n(R/\m^r)\to \Coh_n(R/\m^{r+1})$ is an isomorphism of stacks if $r\geq n$. This is not true unless we pass to the reduced structure. Let $k=\C, R=k[[T]], n=r=1$. Let $S=k[\varepsilon]/(\varepsilon^2)$. Then an $S$-object in $\Coh_1(R/\m^2)=\Coh_1(k[T]/T^2)$ is given by a (locally) free $S$-module of rank $1$ together with an endomorphism $x$ such that $x^2=0$, while an $S$-object in $\Coh_1(R/\m)=\Coh_1(k[T]/T)$ is given by a free $S$-module of rank $1$ together with an endomorphism $x$ such that $x=0$. However, these notions are different because $\varepsilon^2=0$. 
\end{remark}

Since motives in the Grothendieck ring of varieties (stacks) do not distinguish non-reduced structures, we identify these moduli spaces with their reduced structure by abuse of notation in the rest of the paper.

\section{Motivic Quot-to-Coh and Proof of Theorem \ref{thm:A}}\label{sec:quot-to-coh}
In this section, we prove \Cref{thm:A} by establishing an effective version of the limit formula. The core of the argument is a comparison between the stack $\Coh_n(X)$ and the commuting variety $C_n(X)$, and between the Quot scheme and a framed version of it, in the spirit of Nakajima's treatment of $\Hilb_n(\A^2)$ \cite{nakajima1999hilbert}. In this section, let $k$ be an algebraically closed field and $X$ be a classical or formal $k$-scheme. We will work with the reduced structure and the set of $k$-points throughout.  

\subsection{Framing data}
Consider the variety $V_{n,d}:=C_n(X)\times (\A^n)^d$, parametrizing points $\rho\in C_n(X)$ together with an order $d$-tuple of vectors $v_1,\dots,v_d\in k^n$, called the \textbf{framing vectors}. Each point $\mathbf{f}=(\rho,v_1,\dots,v_d)$ of $V_{n,d}$ is called a \textbf{framing datum}. Recall that each $\rho$ represents a pair $(\calF,\iota)$, so each framing datum gives rise to a morphism of coherent $\O_X$-modules $\Phi_{\mathbf{f}}:\O_X^d\to \calF$, whose $i$-th component is the morphism $\O_X\to \calF$ corresponding to the global section $\iota(v_i)\in H^0(X;\calF)$. A framing datum $\mathbf{f}$ is \textbf{stable} if $\Phi_{\mathbf{f}}$ is a surjection of $\O_X$-modules. Let $U_{n,d}\subeq V_{n,d}$ be the \textbf{stable locus} consisting of stable framing data; $U_{n,d}$ is open in $V_{n,d}$.

Let $\GL_n$ act on $V_{n,d}=C_n(X)\times (\A^n)^d$ via its action on $C_n(X)$ and its standard action on each $\A^n$. Note that this action restricts to a free $\GL_n$-action on $U_{n,d}$: if $g\in \GL_n(k)$ fixes $\mathbf{f}=((\calF,\iota),v_1,\dots,v_d)\in J_{n,d}$, then there is a sheaf automorphism $u\in \Aut(\calF)$ that fits both of the following diagrams:
\[
\begin{minipage}{0.25\textwidth}
\begin{tikzcd}[column sep=large, row sep=tiny]
    & \calF \\
    {\O_X^d} \\
    & \calF
    \arrow["u", from=1-2, to=3-2]
    \arrow["{\Phi_{\mathbf{f}}}", from=2-1, to=1-2]
    \arrow["{\Phi_{\mathbf{f}}}"', from=2-1, to=3-2]
\end{tikzcd}
\end{minipage}
\begin{minipage}{0.25\textwidth}
\begin{tikzcd}[column sep=large, row sep=large]
    H^0(X;\calF) \arrow[d, "H^0(u)"'] & k^n \arrow[l, "\iota"', "\simeq"] \arrow[d, "g"] \\
    H^0(X;\calF) & k^n \arrow[l, "\iota"',"\simeq"] 
\end{tikzcd}
\end{minipage}
\]
Since $\Phi_{\mathbf{f}}$ is surjective by assumption, $u$ is the identity morphism, which forces $g$ to be the identity matrix. The freeness of the action implies that the natural map $U_{n,d}\to \Quot_n(R^d)$ that sends $\mathbf{f}$ to the kernel of $\Phi_{\mathbf{f}}$ realizes $U_{n,d}$ as a $\GL_n$-torsor over $\Quot_n(R^d)$. This implies the motivic identity
\begin{equation}
    [\Quot_n(R^d)]=\frac{[U_{n,d}]}{[\GL_n]}\in \KStck{k}.
\end{equation}

The proof of \Cref{thm:A} relies on showing that for large $d$, the stable locus $U_{n,d}$ is a very good approximation of the full space $V_{n,d}$. This is quantified by the following dimension bound on the unstable locus.

\begin{lemma}\label{lem:unstable-locus}
    For $d,n\geq 1$, the unstable locus $Z_{n,d}:=V_{n,d}\setminus U_{n,d}$ has
    \begin{equation}
        \dim Z_{n,d} \leq \dim C_n(R)+ (d+1)(n-1).
    \end{equation}
\end{lemma}
\begin{proof}
    Let $W_{n,d} \subeq (\A^n)^d$ be the determinantal variety consisting of $d$-tuples of vectors that do not span $k^n$. Its dimension is at most $(d+1)(n-1)$. The dimension bound follows from the simple inclusion $Z_{n,d} \subeq C_n(R) \times W_{n,d}$: if $d$ global sections span $H^0(X;\calF)$ for a zero-dimensional sheaf $\calF$, then the induced morphism $\Phi: \O_X^d \to \calF$ is surjective.
\end{proof}

\subsection{An effective bound}
We now prove an effective version of \Cref{thm:A}, from which the theorem follows as an immediate corollary. We first recall the dimension filtration on the Grothendieck ring of stacks.

\begin{definition}[\cite{behrenddhillon2007}]
    For $m\in \Z$, let $F_m\KVar{k}[\L^{-1}]$ be the abelian subgroup of $\KVar{k}[\L^{-1}]$ generated by elements the form $[Y]/\L^n$ where $Y$ is a variety and $\dim Y - n \le m$. This defines a filtration on $\KVar{k}[\L^{-1}]$ satisfying $F_m \cdot F_n\subeq F_{m+n}$, and a sequence $(a_d)_d$ converges to $0$ if $a_d \in F_{m_d}$ with $m_d \to -\infty$ as $d\to\infty$. We note that $[\GL_n]^{-1} \in F_{-n^2}$ and $\L^{-1} \in F_{-1}$. By definition, $\KStck{k}$ is the completion of $\KVar{k}[\L^{-1}]$, with dimension filtration $F_m$ defined as the closure of $F_m\KVar{k}[[\L^{-1}]]$. 
\end{definition}

\begin{theorem}[Effective \Cref{thm:A}]\label{thm:A-effective}
    For any $d,n\geq 1$, we have
    \begin{equation}
        [\Coh_n(X)]-\frac{[\Quot_n(\O_X^d)]}{\L^{nd}} \in F_{-d+\dim C_n(X)-n^2+n-1}.
    \end{equation}
    In particular, 
    \begin{equation}
        [\Coh_n(X)]=\lim_{d\to \infty} \frac{[\Quot_n(\O_X^d)]}{\L^{nd}}.
    \end{equation}
\end{theorem}
\begin{proof}
   In $\KStck{k}$, we have
    \begin{align}
        [Z_{n,d}] &= [V_{n,d}]-[U_{n,d}] \\
        &= [C_n(X)]\L^{nd}-[\GL_n][\Quot_n(\O_X^d)] \\
        &= [\GL_n]\left([\Coh_n(X)]\L^{nd}-[\Quot_n(\O_X^d)]\right).
    \end{align}
    Multiplying by $\L^{-nd}/[\GL_n]$ on both sides and applying the dimension bound from \Cref{lem:unstable-locus}, we have
    \begin{align}
        [\Coh_n(X)]-\frac{[\Quot_n(\O_X^d)]}{\L^{nd}} &= [Z_{n,d}]\L^{-nd} [\GL_n]^{-1} \\
        &\in F_{\dim C_n(X)+(d+1)(n-1)-nd-n^2}=F_{-d+\dim C_n(X)-n^2+n-1},
    \end{align}
    as claimed. The convergence follows because the dimension of the error term tends to $-\infty$ as $d\to\infty$.
\end{proof}

\begin{proof}[Proof of Theorem \ref{thm:A}]
    The convergence of the $t^n$-coefficients of the series is established by \Cref{thm:A-effective}.
\end{proof}

\subsection{Side remarks: examples}
We note how \Cref{thm:A} connects existing results on $\Coh_n(X)$ and $\Quot_n(\O_X^d)$. Let $Z_X(t):=\sum_{n\geq 0} [\Sym^n(X)]\,t^n\in \KVar{k}[[t]]$ denote the motivic zeta function of a variety \cite{kapranov2000}.

\begin{example}
    If $C$ is a smooth curve, it is known that (\cite{bifet1989,bfp2020motive})
    \begin{equation}
        Z_{\O_C^d}(t)=\prod_{i=0}^{d-1} Z_C(\L^i t).
    \end{equation}
    By \Cref{thm:A}, we recover the classical formula for the Coh zeta function:
    \begin{equation}
        \Zhat_C(t)=\lim_{d\to \infty} Z_{\O_C^d}(\L^{-d} t)=\lim_{d\to \infty} \prod_{i=0}^{d-1} Z_C(\L^{i-d} t)=\prod_{j=1}^\infty Z_C(\L^{-j} t).
    \end{equation}
\end{example}

\begin{example}
    If $S$ is a smooth surface, a result of Mozgovoy \cite{mozgovoy2019} gives
    \begin{equation}
        Z_{\O_S^d}(t)=\prod_{i=0}^{d-1}\prod_{j=0}^\infty Z_S(\L^{i+dj}t^{1+j}).
    \end{equation}
    Applying \Cref{thm:A} and taking the limit as $d\to\infty$, we recover the celebrated theorem of Feit--Fine, proven motivically by Bryan--Morrison \cite{feitfine1960pairs,bryanmorrison2015motivic}:
    \begin{equation}
        \Zhat_S(t)=\prod_{i=1}^\infty \prod_{j=1}^\infty Z_S(\L^{-i}t^j).
    \end{equation}
\end{example}

\section{Torsion-free bundles over reduced curves, via lattices}\label{sec:lattice}

In this section, we develop the core algebraic framework for our study. We translate the geometric problem of counting quotients of torsion-free bundles on a curve into the algebraic language of counting sublattices of modules over local rings.

\subsection{The local setup}
We begin by defining the local rings that capture the geometry of curve singularities. Let $R$ be a complete reduced Noetherian local ring of Krull dimension 1 with maximal ideal $\m$ and residue field $\kappa$. The set of minimal prime ideals of $R$ is finite; we denote it by $\set{\p_1,\dots,\p_{b(R)}}$, where $b(R)$ is the \textbf{branching number} of $R$.

\begin{definition}\label{def:tot-frac}
    For a ring $R$ as above, let $K_i$ be the field of fractions of the integral domain $R/\p_i$. The \textbf{total ring of fractions} of $R$ is $K:=\prod_i K_i$. The \textbf{normalization} of $R$ is the ring $\tl R:=\prod_i \tl R_i$, where $\tl R_i$ is the integral closure of $R/\p_i$ in $K_i$.
\end{definition}

We focus on rings that satisfy a key finiteness condition relating them to their normalizations.
\begin{definition}\label{def:local-order}
    A ring $R$ as above is a \textbf{local order} if its normalization $\tl R$ is a finitely generated $R$-module.
\end{definition}

\begin{example}\label{eg:local-orders}
The following settings are the exclusive focus of this paper:
    \begin{enumerate}
        \item A local order is \textbf{arithmetic} if its residue field $\kappa$ is a finite field.
        \item A local order is \textbf{geometric} if its residue field $\kappa$ is an algebraically closed field, which we denote by $k$. In this case, $R$ is a $k$-algebra, and we say $R/k$ is a geometric local order.
    \end{enumerate}
By applying \cite[Chapter II]{serrelocalfields} to $\tl R_i$, these abstract definitions correspond to concrete situations: an arithmetic local order arises as a finite extension of $\Zp$ or $\Fp[[T]]$, and a geometric local order arises as the germ of a reduced curve over $k$.
\end{example}

\subsection{Some basic structures}
Let $R$ be a local order. Since each $\tl R_i$ is a complete, integrally closed, one-dimensional local domain, it is a complete DVR. Let $\pi_i$ be a uniformizer for $\tl R_i$ with residue field $\kappa_i$. For any $\tl R$-module $M$, there is a unique decomposition $M=\prod_i e_i M$, where $e_i$ is the idempotent for the $i$-th branch.

\begin{definition}\label{def:conductor}
    The \textbf{conductor} of $R$ is the ideal $\fc:=\set{f\in K: f\tl R\subeq R}$. It is the largest ideal of $\tl R$ contained in $R$.
\end{definition}

\begin{lemma}\label{lem:conductor-finite}
    The conductor of $R$ is of finite colength in $\tl R$.
\end{lemma}
\begin{proof}
    Since $\fc$ is an ideal of $\tl R$, it suffices to show that for each branch $i$, the component $e_i\fc$ is non-zero. Let $v_1,\dots,v_h\in \tl R$ be a set of generators for $\tl R$ as an $R$-module. For each $i$ and $j$, the element $v_j e_i$ is in $K$, so there exists a non-zero-divisor $g_{ij} \in R$ such that $g_{ij} v_j e_i \in R$. Let $g_i = \prod_j g_{ij}$. Then $g_i e_i$ is a non-zero element of $\tl R$ satisfying $g_i e_i v_j \in R$ for all $j$. It follows that $g_i e_i \tl R \subeq R$, so $g_i e_i \in \fc$, as required.
\end{proof}

\subsection{Torsion-free modules and lattices}\label{subsec:lattice}
We now define the central objects of study in this algebraic framework.
\begin{definition}
    A finitely generated $R$-module $E$ is \textbf{torsion-free of rank $d$} if the natural map $E\to E\otimes_R K$ is injective and $E\otimes_R K\simeq K^d$. An \textbf{$R$-lattice} in a free $K$-module $V$ is a finitely generated $R$-submodule of $V$ whose $K$-span is $V$.
\end{definition}

An $R$-lattice of rank $d$ is itself a torsion-free module of rank $d$, and conversely, any torsion-free module of rank $d$ is isomorphic to an $R$-lattice. While many of our results depend only on the isomorphism class of a module, the proofs often require working with a specific lattice embedding. The notions of torsion-free $\tl R$-modules and $\tl R$-lattices are similarly defined.

\begin{remark}\label{rmk:lattice-bounded}
    We collect some finiteness properties that we will use throughout the paper. For any $R$-lattice $L$, its \textbf{extension} $\tl R L$ is an $\tl R$-lattice, and is the smallest $\tl R$-lattice containing $\L$. Every $\tl R$-lattice of rank $d$ is isomorphic to $\tl R^d$, and is sandwiched between two $\tl R$-lattices of the form
    \begin{equation}\label{eq:lattice-bounded}
        \pi^N \tl R^d \subeq \tl L\subeq \pi^{-N} \tl R^d\text{ for }N\gg 0,
    \end{equation}
    where $\pi=\sum_{i=1}^{b(R)} \pi_i e_i$. The analogous boundedness statement \eqref{eq:lattice-bounded} for $R$-lattices also holds because every $R$-lattice $L$ is sandwiched between two $\tl R$-lattices: $\fc \tl R L \subeq L \subeq \tl R L$. This boundedness easily implies the following facts. A submodule $L'$ of an $R$-lattice $L$ is itself a lattice if and only if $L/L'$ is of finite length. If $L_1,L_2$ are $R$-lattices in the same $K^d$, then so are $L_1+L_2$ and $L_1\cap L_2$.
\end{remark}

We use a flexibly rescaled notion of colength to measure the relative size of lattices. 
\begin{definition-lemma}\label{lem:colength}
    For a local order $(R,\m,\kappa)$ and a subfield $k\subeq \kappa$ with $[\kappa:k]<\infty$, there exists a unique function assigning an integer $[L_1:L_2]_k$, the \textbf{$k$-colength}, to any pair of $R$-lattices, such that:
    \begin{enumerate}
        \item If $L_1\supeq L_2$,  then $[L_1:L_2]_k = \ell_k(L_1/L_2):=[\kappa:k]\cdot \ell(L_1/L_2)$, where $\ell$ is the standard module length in terms of the Jordan--H\"older filtration.
        \item The function is additive in chains: $[L_1:L_3]_k=[L_1:L_2]_k+[L_2:L_3]_k$.
    \end{enumerate}
    If $k$ is a finite field, the \textbf{relative index} is $(L_1:L_2) := \abs{k}^{[L_1:L_2]_k}$, which is independent of the choice of $k$. If $L_1\supeq L_2$, then $(L_1:L_2)=\abs{L_1/L_2}$.
\end{definition-lemma}
\begin{proof}
    The uniqueness is clear from the defining properties. For existence, take
    \begin{equation}
        [L_1:L_2]_k=\ell_k(L_1+L_2/L_2)-\ell_k(L_1+L_2/L_1).\qedhere
    \end{equation} 
\end{proof}

\subsection{Lattice zeta functions}
We now define the central objects of the paper: the generating functions that enumerate sublattices. A key principle that ensures these are well-defined is the ``Artinian reduction" (cf.~\Cref{lem:artinian-reduction-geom}), which shows that for any fixed colength, the problem of counting sublattices can be reduced to a problem of counting submodules within a finite-length (Artinian) module. Below, we give a version of \Cref{lem:artinian-reduction-geom} that applies to any local order.

\begin{lemma}[Artinian reduction, cf.~\Cref{lem:artinian-reduction-geom}]
Let $R$ be a local order with maximal ideal $\m$, and let $E$ be a finitely generated torsion-free $R$-module. If $E' \subeq E$ is a submodule of $k$-colength $n$, then $E'$ must contain the submodule $\m^n E$. \label{lem:artinian-reduction}
\end{lemma}
\begin{proof}
The quotient $E/E'$ is an $R$-module of $k$-length $n$, thus of length $n/[\kappa:k]$. By the same argument as in \Cref{lem:artinian-reduction-geom}, $\m^{n/[\kappa:k]}(E/E')=0$. As $n\geq n/[\kappa:k]$, this implies $\m^n E \subeq E'$.
\end{proof}

\begin{definition}
Let $E$ be a torsion-free module of rank $d$ over a local order $R$.
\begin{enumerate}
    \item If $R$ is an arithmetic local order, its \textbf{arithmetic lattice zeta function} is
    \begin{equation}\label{def:lattice-zeta}
        \zeta_E(s)=\zeta_E^R(s):=\sum_{E'\subeq_R E} \abs{E/E'}^{-s},
    \end{equation}
    summing over all finite-colength submodules $E'$. It is a well-defined Dirichlet series by \Cref{lem:artinian-reduction}.
    \item If $R/k$ is a geometric local order, its \textbf{motivic lattice zeta function} is
    \begin{equation}\label{def:motivic-lattice-zeta}
        Z^R_E(t):=\sum_{n=0}^\infty [\Quot_n(E)]\, t^n \in \KVar{k}[[t]].
    \end{equation}
\end{enumerate}
\end{definition}

\begin{remark}\label{rmk:quot-k-scheme-struct}
    Recall that $\Quot_n(E)=\Quot_n^{R/\m^r}(E/\m^r E)$ for any $r\geq m$, up to reduced structure. This naturally embeds $\Quot_n(E)$ into the classical Grassmannian $\Quot_n^k(E/\m^r E)$ as a closed subset. 
\end{remark}

\begin{remark}\label{rmk:z-arithmetic}
    The two notions of zeta functions are related by point-counts. For an arithmetic order $(R,\m,\kappa)$, let $k$ be a subfield of $\kappa$ with $q=\abs{k}$. The point-count generating series
    \begin{equation}
        Z^{R/k}_E(t):=\sum_{E'\subeq_R E} t^{[E:E']_k}
    \end{equation}
    is related to the arithmetic zeta function by the change of variables $\zeta^R_E(s)=Z^{R/k}_E(q^{-s})$. While $Z^{R/k}_E(t)$ depends on the choice of $k$, $\zeta^R_E(s)$ does not.
\end{remark}

\subsection{From local to global}
The computation of global Quot zeta functions on a curve $X$ reduces to the local ones on its singularities.
\begin{theorem}[\cite{fantechiricolfi2024motivic}]\label{thm:local-to-global}
    Let $k$ be an algebraically closed field of characteristic zero, and $X/k$ be a reduced curve with singular points $p_1,\dots,p_h$. For a torsion-free bundle $\calE$ on $X$ of rank $d$, the global Quot zeta function factorizes as
    \begin{equation}
        Z_\calE^X(t)=Z_{\calE|_{X^{\mathrm{sm}}}}^{X^\mathrm{sm}}(t) \cdot \prod_{j} Z_{\calE|_{p_j}}^{\Ohat_{X,p_j}}(t).
    \end{equation}
\end{theorem}
Since the smooth factor is well-understood (see \eqref{eq:bfp}), the problem reduces to computing the punctual zeta functions for the local rings $R_j = \Ohat_{X,p_j}$. A similar decomposition holds for the arithmetic lattice zeta function using a standard Euler product argument \cite{solomon1977zeta}.

\begin{remark}
    \Cref{thm:local-to-global} is the one-dimenensional case of \cite[Thm.~5.3]{fantechiricolfi2024motivic} with one slight generalization: the original statement assumes $\calE$ is a locally free sheaf, which corresponds to the special case where each $E_j:=\calE|_{p_j}$ is a free module over $R_j:=\Ohat_{X,p_j}$, while in our setting, we allow $E_j$ to be a torsion-free $R_j$-module not isomorphic to $(R_j)^d$. However, the proof of \cite[Theorem~5.3]{fantechiricolfi2024motivic} generalizes because the stratification in \cite[Corollary~3.12]{fantechiricolfi2024motivic} still applies.
\end{remark}

For future reference, we record the zeta functions in the smooth curve case. The global formula is due to \cite{bfp2020motive}, and the local and arithmetic versions are due to \cite{bifet1989} and \cite{solomon1977zeta}, respectively:
\begin{align}
    Z_{\calE}^{X^\mathrm{sm}}(t) &= \prod_{i=0}^{d-1} Z_{X^{\mathrm{sm}}}(\L^i t), \quad \calE \text{ a vector bundle of rank } d, \label{eq:bfp} \\
    Z_{\tl R^d}^{\tl R}(t) &= \frac{1}{(t;\L)_d^{b(R)}}, \\
    \zeta_{\tl R^d}^{\tl R}(s) &= \frac{1}{\prod_{i=1}^{b(R)} (\abs{\kappa_i}^{-s};\abs{\kappa_i})_d}.\label{eq:solomon}
\end{align}

\section{Rationality: arithmetic}\label{sec:rationality}

In this section, we establish the key structural lemmas that underpin the proof of our main rationality theorems. We then apply this framework to prove the rationality of the arithmetic lattice zeta function (\Cref{thm:rationality-arithmetic}). Throughout, $(R,\m,\kappa)$ is a local order. We fix a subfield $k\subeq \kappa$ with $[\kappa:k]<\infty$, and abbreviate the $k$-colength $[L_1:L_2]_k$ as $[L_1:L_2]$.

\subsection{Some commutative algebra}
We first collect two standard results that will be used throughout our constructions.

\begin{lemma}\label{lem:sur_homogeneity}
    Let $S$ be a direct product of finitely many Noetherian local rings, and $M$ be a finitely generated module over $S$. Then $\GL_n(S)$ acts transitively on the set of surjections $\Surj(S^n,M)$.
\end{lemma}
\begin{proof}
    We may assume $S$ is a local ring. The conclusion then follows from the uniqueness of the minimal free resolution, see \cite[Theorem~20.2]{eisenbudcommutative}. 
\end{proof}

\begin{lemma}\label{lem:extension_problem}
    Let $V_1$ be a direct summand of a module $V$ over a ring $A$. Any decomposition $V=V_1\oplus V_2$ induces a bijection
    \begin{equation}
        \set{W\subeq_A V: W+V_1=V} \simeq \set{(W',\varphi): W'\subeq_A V_1, \varphi\in \Hom_A(V_2,V_1/W')},
    \end{equation}
    where a submodule $W$ corresponds to the pair $(W\cap V_1, V_2 \to V/W \simeq V_1/(W\cap V_1))$ using the second isomorphism theorem, and conversely, a pair $(W',\varphi)$ corresponds to the kernel of $1\oplus \varphi:V_1\oplus V_2\to V_1/W'$. Moreover, this bijection preserves the quotient module, i.e., $V/W \simeq_A V_1/W'$.
\end{lemma}
\begin{proof}
    This is a straightforward check.
\end{proof}

\subsection{Extension fiber}
Our main strategy is to compare the problem of counting $R$-lattices to the well-understood problem of counting $\tl R$-lattices. The connection is made via the \textbf{extension fiber}, defined for an $\tl R$-lattice $\tl L$ as
\begin{equation}
    E_R(\tl L):=\set{L\subeq_R \tl L: \tl RL=\tl L},\quad E_R(\tl L;n):=\set{L\in E_R(\tl L):[\tl L:L]=n}.
\end{equation}

A crucial observation is that any lattice in an extension fiber is bounded by the conductor.
\begin{lemma}\label{lem:bounded}
    For all $L\in E_R(\tl L)$, we have $\fc \tl L\subeq L\subeq \tl L$.
\end{lemma}
\begin{proof}
    This follows from $L=RL\supeq \fc L=\fc \tl R L=\fc \tl L$.
\end{proof}
Consequently, the colength $[\tl L:L]$ is bounded by $[\tl L:\fc\tl L]=d[\tl R:\fc]$. This observation is sufficient to prove the rationality for the simplest case.

\begin{theorem}\label{thm:rationality-tl-arithmetic}
    If $R$ is an arithmetic local order and $\tl E$ is a torsion-free module over $\tl R$ of rank $d$, its normalized zeta function
    \begin{equation}\label{eq:rationality-tl-arithmetic}
        \nu^R_{\tl E}(s):=\frac{\zeta^R_{\tl E}(s)} {\zeta^{\tl R}_{\tl R^d}(s)}=\sum_{N\in E_R(\tl R^d)} \abs{\tl R^d/N}^{-s}
    \end{equation}
    is a Dirichlet polynomial.
\end{theorem}
\begin{proof}
    We may assume $\tl E$ is a $\tl R$-lattice. Because the extension of any $R$-sublattice of $\tl E$ is contained in $\tl E$, we can write 
    \begin{equation}
        \zeta^R_{\tl E}(s) = \sum_{\tl L \subeq_{\tl R} \tl E} (\tl E:\tl L)^{-s} \sum_{L \in E_R(\tl L)} (\tl L:L)^{-s}
    \end{equation}
    Since all $\tl R$-lattices of rank $d$ are isomorphic to $\tl R^d$, we may replace $\tl L$ in the inner sum by $\tl R^d$, so the formula \eqref{eq:rationality-tl-arithmetic} follows. Finally, the right-hand side of \eqref{eq:rationality-tl-arithmetic} is a Dirichlet polynomial by \Cref{lem:bounded}.
\end{proof}

\subsection{Restricted extension fiber and the padding lemma}
When counting sublattices of an $R$-lattice $M$ (which is not necessarily an $\tl R$-lattice), the condition $L\subeq M$ no longer implies $\tl L \subeq M$. To handle this, we introduce the \textbf{restricted extension fiber}
\begin{equation}
    E_R(\tl L;M):=\set{L\in E_R(\tl L):L\subeq M}, \quad E_R(\tl L;M;n):=\set{L\in E_R(\tl L;n):L\subeq M}.
\end{equation}
The key to controlling the lattices $\tl L$ that we need to consider is the following crucial ``padding" lemma. It shows that the structure of $E_R(\tl L;M)$ depends only on the restricted extension fiber of the "padded" lattice $\tl L + \ut M$ for a fixed $\tl R$-lattice $\ut M \subeq M$.

\begin{lemma}[Padding lemma]\label{lem:padding}
    Let $\ut{M}$ be a fixed $\tl R$-lattice contained in an $R$-lattice $M$. For any two $\tl R$-lattices $\tl N, \tl L$ such that $\tl N+\ut M=\tl L$:
    \begin{enumerate}
        \item There exists an isomorphism $h:\tl L\to \tl N$ such that $(h-\mathrm{id})(\tl L)\subeq \ut M$ and $(h^{-1}-\mathrm{id})(\tl N)\subeq \ut M$.
        \item Any such $h$ induces a canonical bijection $E_R(\tl L;M;n)\to E_R(\tl N;M;n)$ for any colength $n$ by sending $L$ to $h\cdot L$.
    \end{enumerate}
\end{lemma}
\begin{proof}
    For (a), consider the diagram of natural maps:
    \[
    \begin{tikzcd}
    \tl N \arrow[dr, two heads] & & \tl L \arrow[ll, "h", dashed, "\simeq"'] \arrow[dl, two heads] \\
    & \tl L/\ut M &
    \end{tikzcd}
    \]
    The map $\tl N \to \tl L/\ut M$ is surjective because $\tl L = \tl N + \ut M$. Since $\tl L$ and $\tl N$ are both isomorphic to $\tl R^d$, \Cref{lem:sur_homogeneity} guarantees the existence of an isomorphism $h:\tl L \to \tl N$ making the diagram commute. This commutativity implies that for any $x \in \tl L$, $hx-x \in \ut M$, as required; similarly for $h^{-1}$.

    For (b), the map $L \mapsto hL$ is a bijection from $E_R(\tl L;n)$ to $E_R(\tl N;n)$. We need only check it preserves the condition of being contained in $M$. For any $L \in E_R(\tl L;M)$ and any $x \in L$, we have $x \in M$. From part (a), $hx - x \in \ut M \subeq M$. Thus, $hx \in M$, which shows $hL \subeq M$. The same argument applied to $h^{-1}$ shows the inverse map also preserves the condition. 
\end{proof}

\subsection{Padding fiber}
The padding lemma necessitates studying the set of all $\tl R$-lattices $\tl N$ that can be ``padded" to a given $\tl L$, a set we call the \textbf{padding fiber}:
\begin{equation}
    \Pd(\tl L;\ut M):=\set{\tl N\subeq_{\tl R} \tl L:\tl N+\ut M=\tl L}.
\end{equation}
The goal of this subsection is to compute a suitable generating function for the padding fiber. To state our results, we first need to define some branch-wise invariants. 
\begin{definition}
For any finitely generated $\tl R$-module $Q$, its \textbf{rank vector} is
\begin{equation}\label{eq:rank-vector}
    \boldsymbol{\rk}(Q):=(\rk_i(Q))_{i=1}^{b(R)}, \quad \text{where}\quad \rk_i(Q):=\dim_{\kappa_i} (Q\otimes_{\tl R} \kappa_i).
\end{equation}
For two $\tl R$-lattices $\tl L_1, \tl L_2$, their \textbf{$k$-colength vector} is
\begin{equation}\label{eq:colength-vector}
    [[\tl L_1:\tl L_2]]:=([\tl L_1\otimes_{\tl R} \tl R_i:\tl L_2\otimes_{\tl R} \tl R_i])_{i=1}^{b(R)},
\end{equation}
and similarly, in the arithmetic setting, the \textbf{index vector} is
\begin{equation}\label{eq:index-vector}
    ((\tl L_1:\tl L_2)):=\parens[\bigg]{(\tl L_1\otimes_{\tl R} \tl R_i:\tl L_2\otimes_{\tl R} \tl R_i)}_{i=1}^{b(R)},
\end{equation}
\end{definition}
% For a system of variables $\boldsymbol{t}=(t_1,\dots,t_{b(R)})$ and an integer vector $\boldsymbol{n}=(n_1,\dots,n_{b(R)})$, we use the notation $\boldsymbol{t}^{\boldsymbol{n}} := t_1^{n_1}\cdots t_{b(R)}^{n_{b(R)}}$. 
For a system of variables $\boldsymbol{s}=(s_1,\dots,s_{b(R)})$ and a positive integer vector $\boldsymbol{q}=(q_1,\dots,q_{b(R)})$, we use the notation $\boldsymbol{q}^{-\boldsymbol{s}} := q_1^{-s_1}\cdots q_{b(R)}^{-s_{b(R)}}$, viewed as a multivariate Dirichlet monomial.

The structure of the padding fiber is described by the following lemma, which provides an explicit parametrization.

\begin{lemma}\label{lem:padding_fiber}
    Let $\ut M \subeq \tl L$ be two $\tl R$-lattices, and let $\boldsymbol{r}=\boldsymbol{\rk}(\tl L/\ut M)$. Then 
    \begin{enumerate}
        \item There exists a (non-canonical) direct sum decomposition of $\tl R$-module $\tl L = L^{\mathrm{in}}\oplus L^{\mathrm{out}}$ such that $L^{\mathrm{in}}\subeq \ut M$, $\boldsymbol{\rk}(L^{\mathrm{out}})=\boldsymbol{r}$, and the padding fiber is precisely the set of sublattices $\tl N$ that satisfy $\tl N+L^{\mathrm{in}}=\tl L$. 
        \item Any decomposition as above canonically induces a bijection
        \begin{equation}
        \Pd(\tl L;\ut M) \simeq \prod_{i=1}^{b(R)} \set*{(N_i',\varphi_i): N_i'\subeq_{\tl R_i} \tl R_i^{d-r_i}, \varphi_i\in \Hom_{\tl R_i}(\tl R_i^{r_i}, \tl R_i^{d-r_i}/N_i')}.
        \end{equation}
        Moreover, under this bijection, the colength vector for $\tl N\in \Pd(\tl L;\ut M)$ is given by $[[\tl L:\tl N]]=([[\tl R_i^{d-r_i}:N_i']])_{i=1}^{b(R)}$.
    \end{enumerate}
\end{lemma}
\begin{proof}
We may assume $\tl R$ has a single branch, so $\boldsymbol{r}=r$ is an integer.
\begin{enumerate}
    \item To construct the decomposition, we consider the quotient map $\rho:\tl L\to \tl L/\ut{M}$ and fix a surjection $\alpha:\tl R^r\to \tl L/\ut{M}$. Since $\tl L$ is free (thus projective), there exists a map $\beta:\tl L\to \tl R^r$ such that $\rho=\alpha\circ\beta$. Since $\alpha$ is an isomorphism modulo $\pi \tl R$, $\beta$ is surjective modulo $\pi \tl R$, and thus surjective by Nakayama's lemma. We define $L^{\mathrm{in}}:=\ker(\beta)$, which is contained in $\ut M$ by the diagram. Since $\tl R^r$ is free, the surjection $\beta$ splits, providing a submodule $L^{\mathrm{out}} \simeq \tl R^r$ such that $\tl L=L^{\mathrm{in}}\oplus L^{\mathrm{out}}$.
    \[ 
    \begin{tikzcd}[column sep=large, row sep=large]
    L^\mathrm{in} \arrow[r, hook] &
    \tl L \arrow[dr, "\rho"', two heads] \arrow[r, "\beta", two heads] &
    \tl R^r \arrow[d, "\alpha", two heads] \arrow[l, bend right=45, "\gamma"'] \\
    & & \tl L/\ut{M}
    \end{tikzcd}
    \]
    We now show that for this decomposition, $\tl N+\ut M=\tl L$ if and only if $\tl N+L^{\mathrm{in}}=\tl L$. Consider the following diagram consisting of natural maps:
    \[\begin{tikzcd}[column sep=large, row sep=large]
    & \tl N \arrow[dl, "{\alpha_2}"'] \arrow[dr, "{\alpha_1}"] & \\
    \tl L/L^\mathrm{in} \arrow[rr, "{\delta}"] & & \tl L/\ut{M}
    \end{tikzcd}\]
    Since $\rk(\tl L/L^{\mathrm{in}})=\tl L/\ut M$, $\delta$ is an isomorphism modulo $\pi \tl R$. By Nakayama's lemma, the surjectivities of $\alpha_1$ and $\alpha_2$ are equivalent, finishing the proof.
    \item The condition $\tl N+L^{\mathrm{in}}=\tl L$ is precisely the setup for \Cref{lem:extension_problem}, which gives the desired bijection. The statement about the colength vector follows from the second isomorphism theorem. \qedhere
\end{enumerate}
\end{proof}

This classification allows us to compute the generating function for the padding fiber directly.

\begin{proposition}\label{prop:padding-fiber-gen-arithmetic}
    Let $R$ be an arithmetic local order. Fix $\tl R$-lattices $\tl L\supeq \ut M$ such that $\boldsymbol{\rk}(\tl L/\ut M)=\boldsymbol{r}$. Then we have an identity of Dirichlet series in $s_1,\dots,s_{b(R)}$: 
    \begin{equation}\label{eq:padding-fiber-gen-arithmetic}
        \sum_{\tl N\in \Pd(\tl L;\ut M)} ((\tl L:\tl N))^{-\boldsymbol{s}} = \prod_{i=1}^{b(R)} \frac{1}{(q_i^{r_i-s_i};q_i)_{d-r_i}},
    \end{equation}
    where $q_i=\abs{\kappa_i}$. In particular, taking $s_i=s$ for all $i$, we get
    \begin{equation}
        \sum_{\tl N\in \Pd(\tl L;\ut M)} (\tl L:\tl N)^{-s} = \prod_{i=1}^{b(R)} \frac{1}{(q_i^{r_i-s};q_i)_{d-r_i}}.
    \end{equation}
\end{proposition}
\begin{proof}
    Using the bijection from \Cref{lem:padding_fiber}, we can compute the sum by iterating over all choices of $N'_i$ and $\varphi_i$ for each branch, and then take the product over all branches. For a branch $i$, fix identifications $L^{\mathrm{in}}\otimes_{\tl R} \tl R_i\simeq \tl R_i^{d-r_i}$ and $L^{\mathrm{out}}\otimes_{\tl R} \tl R_i\simeq \tl R_i^{r_i}$. For a fixed sublattice $N'_i \subeq_{\tl R_i} \tl R_i^{d-r_i}$, the number of homomorphisms $\varphi_i:\tl R_i^{r_i}\to \tl R_i^{d-r_i}/N_i'$ is $(\tl R_i^{d-r_i}:N'_i)^{r_i}$. The generating function for the $i$-th branch is therefore
    \begin{equation*}
        \sum_{N'_i \subeq_{\tl R_i} \tl R_i^{d-r_i}} (\tl R_i^{d-r_i}:N'_i)^{r_i} (\tl R_i^{d-r_i}:N'_i)^{-s_i}  = \zeta^{\tl R_i}_{\tl R_i^{d-r_i}}(s_i-r_i).
    \end{equation*}
    By Solomon's formula \eqref{eq:solomon}, this is equal to $1/(q_i^{r_i-s_i};q_i)_{d-r_i}$. Taking the product over all branches gives the result.
\end{proof}

\subsection{Rational formulas}
We now combine the preceding structural results to give a rational formula for the arithmetic lattice zeta function. The key is to stratify the set of all sublattices according to the rank vector of their extension.

\begin{definition}\label{def:boundary-locus}
Let $M$ be an $R$-lattice and $\ut M$ be an $\tl R$-lattice contained in $M$. The \textbf{boundary locus} is the set of $R$-lattices
\begin{equation}
    B_R(M;\ut M) := \set*{
        L\subeq_R M: \tl RL\supeq \ut M
    },
\end{equation}
which we stratify by the rank vector of the extension:
\begin{equation}
    B_R(M;\ut M;\boldsymbol{r}) := \set*{
        L\in B_R(M;\ut M): \boldsymbol{\rk}(\tl RL/\ut M)=\boldsymbol{r}
    }.
\end{equation}
We can equivalently think of $B_R(M;\ut M)$ as the set of pairs $(\tl L,L)$ where $L\in E_R(\tl L;M)$ and $\tl L \supeq \ut M$.
\end{definition}
\begin{remark}\label{rmk:boundary-locus-bounded}
    The boundary locus is ``bounded": for all $L\in B_R(M;\ut M)$, by \Cref{lem:bounded}, we have $\fc \ut M\subeq L\subeq \tl RM$.
\end{remark}

With this definition, we can state and prove the main theorem of this section.

\begin{theorem}\label{thm:rationality-arithmetic}
    Let $R$ be an arithmetic local order, and $M$ be an $R$-lattice in $K^d$. Let $\ut M$ be any $\tl R$-lattice contained in $M$. The normalized lattice zeta function is given by the finite sum
    \begin{equation}\label{eq:rational-formula-arithmetic}
        \nu_M^R(s) :=\frac{\zeta_M^R(s)}{\zeta_{\tl R^d}^{\tl R}(s)} = \sum_{\boldsymbol{r}} \left( \prod_{i=1}^{b(R)} (\abs{\kappa_i}^{-s};\abs{\kappa_i})_{r_i} \right) \left( \sum_{L\in B_R(M;\ut M;\boldsymbol{r})} (M:L)^{-s} \right),
    \end{equation}
    which is a Dirichlet polynomial in $s$.
\end{theorem}
\begin{proof}
    The set of all $R$-sublattices of $M$, denoted $\Quot(M)$, can be stratified by the rank vector $\boldsymbol{r} = \boldsymbol{\rk}((\tl R N + \ut M)/\ut M)$. Let $\Quot(M;\boldsymbol{r})$ be the set of sublattices in a given stratum.

    The proof proceeds by establishing a bijection between the lattices in $\Quot(M;\boldsymbol{r})$ and a more structured set of combinatorial data. For a given $N \in \Quot(M;\boldsymbol{r})$, we define $\tl N = \tl R N$ and $\tl L = \tl N + \ut M$. The padding lemma (\Cref{lem:padding}) provides a bijection between $E_R(\tl N;M)$ and $E_R(\tl L;M)$. Let $L\in E_R(\tl L;M)$ be the lattice corresponding to $N$ under this bijection. This shows that the set $\Quot(M;\boldsymbol{r})$ is in bijection with the set of triples $(\tl L,L,\tl N)$, where $(\tl L,L)\in B_R(M;\ut M;\boldsymbol{r})$ and $\tl N\in \Pd(\tl L;\ut M)$. Under this bijection, the colengths are related by
    \begin{equation*}
        [M:N]=[M:\tl L]+[\tl L:\tl N]+[\tl N:N]=[M:L]+[\tl L:\tl N].
    \end{equation*}
    
    This allows us to re-write the sum over all lattices $N$ in the stratum $\Quot(M;\boldsymbol{r})$ as a sum over the boundary locus $B_R(M;\ut M;\boldsymbol{r})$ multiplied by the generating function for the padding fiber (\Cref{prop:padding-fiber-gen-arithmetic}). This yields:
    \begin{equation*}
        \sum_{N \in \Quot(M;\boldsymbol{r})} (M:N)^{-s} = \left(\prod_{i=1}^{b(R)} \frac{1}{(\abs{\kappa_i}^{r_i-s};\abs{\kappa_i})_{d-r_i}}\right) \left(\sum_{L \in B_R(M;\ut M;\boldsymbol{r})} (M:L)^{-s}\right).
    \end{equation*}
    Multiplying by the normalization factor $\zeta_{\tl R^d}^{\tl R}(s)^{-1} = \prod (\abs{\kappa_i}^{-s};\abs{\kappa_i})_d$ and using the elementary identity $(a;q)_d / (aq^{r};q)_{d-r} = (a;q)_r$ gives the formula for the contribution of the $\boldsymbol{r}$-stratum to $\nu_M^R(s)$. Summing over all rank vectors $\boldsymbol{r}\in \set{0,\dots,d}^{b(R)}$ gives the final expression. The result is a Dirichlet polynomial by \Cref{rmk:boundary-locus-bounded}.
\end{proof}

As an immediate corollary of this formula, we can evaluate the normalized zeta function at $s=0$.

\begin{corollary}\label{cor:central-arithmetic}
    Let $R$ be an arithmetic local order, and let $M$ be an $R$-lattice in $K^d$. Then
    \begin{equation}\label{eq:central-arithmetic}
       \nu^R_M(0)=\abs{E_R(\tl R^d)}.
    \end{equation}
    In particular, it is independent of $M$.
\end{corollary}
\begin{proof}
    We evaluate the formula in \Cref{thm:rationality-arithmetic} at $s=0$ with an arbitrary choice of $\ut M$. The $q$-Pochhammer symbol $(1;q)_{r}$ is zero unless $r=0$. Therefore, the only term in the sum over $\boldsymbol{r}$ that survives is the one corresponding to the zero vector, $\boldsymbol{r}=\boldsymbol{0}$. The boundary locus $B_R(M;\ut M;\boldsymbol{0})$ simply consists of lattices $L\in E_R(\ut M)$. The sum thus becomes $\abs{E_R(\ut M)}=\abs{E_R(\tl R^d)}$.
\end{proof}

\begin{remark}
    The $d=1$ case of \Cref{cor:central-arithmetic} was proven by Yun \cite[Theorem 2.5(2)]{yun2013orbital}.
\end{remark}

\section{Rationality: motivic}\label{sec:motivic}
We now upgrade the results of the previous section to the motivic setting, proving the rationality of the motivic lattice zeta function for geometric local orders. For a geometric local order $R/k$, we define a further stratification of $B_R(M;\ut M;\boldsymbol{r})$ as follows.

\begin{definition}
Let $M$ be an $R$-lattice and $\ut M$ be an $\tl R$-lattice contained in $M$. For a rank vector $\boldsymbol{r}$ and a colength $n$, define 
\begin{equation}
    B_R(M;\ut M;\boldsymbol{r};n) := \set*{
        L\in \Quot_n(M): \tl RL\supeq \ut M, \; \boldsymbol{\rk}(\tl RL/\ut M)=\boldsymbol{r}
    }.
\end{equation}
We can think of $B_R(M;\ut M;\boldsymbol{r};n)$ as the set of pairs $(\tl L,L)$ by setting $\tl L=\tl RL$.
\end{definition}

The Artinian reduction principle (\Cref{rmk:quot-k-scheme-struct}) realizes $B_R(M;\ut M;\boldsymbol{r};n)$ as a locally closed subset in a projective space. Specifically, the stratum of $B_R(M;\ut M;\boldsymbol{r};n)$ corresponding to a fixed $i=[\tl RM:\tl L]$ and $j=[\tl RM:L]$ lives in the product $\Quot_i^{\tl R}(\tl RM)\times \Quot_j^{R}(\tl M)$, each factor being a closed subset in the Grassmannian parametrizing $k$-subspaces of $\tl RM/\m^j \tl RM$. The conditions that cut out $B_R(M;\ut M;\boldsymbol{r};n)$ can be expressed as a sequence of closed and open conditions on a product of Grassmannians. Specifically, the conditions $\tl L\supeq \ut M$, $L\subeq \tl L$, and $L\subeq M$ are closed conditions, and subject to these, the condition $\boldsymbol{\rk}(\tl L/\ut M)=\boldsymbol{r}$ is a determinantal locally closed condition and the condition $\tl R L = \tl L$ is an open condition. Thus, $B_R(M;\ut M;\boldsymbol{r};n)$ can be equipped with the structure of a locally closed subscheme of a product of projective spaces. Since we only care about its motive, we do not distinguish this subscheme with its reduced structure. In the rest of this section, we will describe a scheme morphism through its underlying map of $k$-points. The validity of any such morphism we will describe this way can be easily verified using an explicit embedding of the source and the target into a product of Grassmannians (and further to a product of projective spaces via the Pl\"ucker embedding). 

\begin{theorem}\label{thm:rationality-motivic}
    Let $R$ be a geometric local order over an algebraically closed field $k$, and $M$ be an $R$-lattice in $K^d$. Let $\ut M$ be any $\tl R$-lattice contained in $M$. Then the normalized motivic lattice zeta function of $M$ is given by
    \begin{equation}\label{eq:rational-formula-motivic}
        N_M^R(t):=\frac{Z_M^R(t)}{Z_{\tl R^d}^{\tl R}(t)} = \sum_{\boldsymbol{r}} \left( \prod_{i=1}^{b(R)} (t;\L)_{r_i}\right) \left( \sum_n [B_R(M;\ut M;\boldsymbol{r};n)] \, t^n \right),
    \end{equation}
    where the right-hand side is a polynomial in $t$. 
\end{theorem}

While the proof of \Cref{thm:rationality-arithmetic} was based on explicit parametrization—a method suitable for a motivic upgrade—some parts of the parametrization, namely the choice of $h$ in \Cref{lem:padding} and the decomposition in \Cref{lem:padding_fiber}, were non-canonical. We will handle this issue using the Schubert cell decomposition of the affine Grassmannian.

Before we proceed to the proof, we give two corollaries of \Cref{thm:rationality-motivic}.

\begin{corollary}\label{prop:rtilde-motivic}
    If $R/k$ is a geometric local order, then for any $\tl R$-lattice $\tl L$ of rank $d$, we have
    \begin{equation}
        N^{R}_{\tl L}(t) = \sum_{n=0}^{\infty} \bracks{E_{R}(\tl R^d;n)}\, t^n,
    \end{equation}
    a polynomial in $t$.
\end{corollary}
\begin{proof}
    Take $M=\ut M=\tl L$ in \Cref{thm:rationality-motivic}, then the only nonvanishing summand is $\boldsymbol{r}=\boldsymbol{0}$. 
\end{proof}

\begin{corollary}\label{cor:central-motivic}
    For any $R$-lattice $M$ of rank $d$,
    \begin{equation}
        N^R_M(1)=[E_R(\tl R^d)]=N^R_{\tl R^d}(1).
    \end{equation}
\end{corollary}
\begin{proof}
    This follows from the same argument as in \Cref{cor:central-arithmetic}, except we take the motivic \Cref{thm:rationality-motivic} as input. 
\end{proof}

\subsection{Affine Grassmannian and its stratification}\label{subsec:schubert}
Let $G=\GL_d(K)$ and $P=\GL_d(\tl R)$, and define the \textbf{affine Grassmannian} $X=\Gr_{\tl R}(K^d):=G/P$. It parametrizes $\tl R$-lattices in $K^d$; an element $x P$ corresponds to the lattice $x\tl R^d$. The affine Grassmannian is an ind-scheme; its truncation $\Gr_{\tl R}(\tl M,\tl L):=\set{\tl N:\tl L\subeq \tl N\subeq_{\tl R}\tl M}$ parametrizing $\tl R$-lattices between $\tl L$ and $\tl M$ is a projective scheme (cf.~\Cref{lem:quot-k-scheme-struct}). 

We describe the standard stratification of $X$ into Schubert cells, adapted into an explicit language from our work \cite{huangjiang2022a}; for a standard reference, see \cite{Zhu_AffineGrass}. For now, assume $\tl R$ has one branch, namely, $\tl R=k[[T]]$. A \textbf{(single-branch) signature} is a $d$-vector $\mu=(\mu_1,\dots,\mu_d)\in \Z^d$. 
A signature is \textbf{dominant} if $\mu_1\geq \dots \geq \mu_d$, and \textbf{nonnegative} (written $\mu\geq 0$) if $\mu_i\geq 0$ for all $i$. The \textbf{(opposite) Iwahori subgroup} $I$ is the subgroup of $P$ consisting of matrices that are lower triangular modulo $T$. The diagonal matrix with signature $\mu$ is
\begin{equation}
    D_\mu = \mathrm{diag}(T^{\mu_1},\dots,T^{\mu_d})\in G.
\end{equation}
The \textbf{Schubert cell} of signature $\mu$ is the stratum $X_\mu^\circ$ in the Bruhat decomposition:
\begin{equation}
    X=\bigsqcup_{\mu\in \Z^d} X_\mu^\circ := \bigsqcup_{\mu\in \Z^d} I D_\mu P/P.
\end{equation}
The stratum $X_\mu^\circ$ is an affine space over $k$. To give an explicit parametrization, define the \textbf{Iwahori cell} of signature $\mu$ as the set of matrices
\begin{equation}
    I_\mu := \set*{[\iota_{ij}]_{i,j=1}^d: \mat{ \iota_{ii}=1 \\
     \iota_{ij}\in \Span_k \set{T^a: 0\leq a < \mu_i-\mu_j}, \text{ if $i>j$}\\
     \iota_{ij}\in \Span_k \set{T^a: 0< a < \mu_i-\mu_j}, \text{ if $i<j$}
     }}\subeq I. 
\end{equation}
It is manifestly an affine space; its dimension, denoted by $\ell(\mu)$, can be computed by counting the number of free coefficients in the matrices in $I_\mu$. A formula for $\ell(\mu)$ given in \cite[Eq.~(2.15)]{huangjiang2022a} is
\begin{equation}
    \ell(\mu)=\dim I_\mu = \dim X_\mu^\circ = \sum_{i,j=1}^d \max\set*{0, \mu_j - \mu_i + \floor*{\frac{j- i}{d}}}.
\end{equation}
We note that the set $I_\mu D_\mu$ is precisely the set of matrices in the ``hlex normal form'' described in \cite[Section 6]{huangjiang2022a}: 
\begin{equation}\label{eq:normal-form}
    I_\mu D_\mu = \set*{[x_{ij}]_{i,j=1}^d: \mat{ x_{ii}=T^{\mu_i} \\
     x_{ij}\in \Span_k \set{T^a: \mu_j\leq a < \mu_i}, \text{ if $i>j$}\\
     x_{ij}\in \Span_k \set{T^a: \mu_j< a < \mu_i}, \text{ if $i<j$}
     }}. 
\end{equation}
We have an isomorphism of schemes
\begin{equation}\label{eq:affineparameterization}
    I_\mu \to X_\mu^{\circ}, \quad \iota\mapsto \iota D_\mu \tl R^d,
\end{equation}
providing the affine parametrization.

If $\tl R$ has multiple branches, i.e., $\tl R=\prod_{i=1}^{b(R)} \tl R_i$ with $\tl R_i=k[[T_i]]$, then we have the product structure $G=\prod_{i=1}^{b(R)} G^{(i)}$, where $G^{(i)}=\GL_d(K_i)$ is the corresponding group for the $i$-th branch. Similarly, $P=\prod_{i=1}^{b(R)} P^{(i)}$, $X=\prod_{i=1}^{b(R)} X^{(i)}$, and $I=\prod_{i=1}^{b(R)} I^{(i)}$, where $P^{(i)}, X^{(i)}$ and $I^{(i)}$ are the corresponding objects for the $i$-th branch. A \textbf{signature} for $X$ is a $b(R)$-tuple of $d$-vectors
\begin{equation}
    \boldsymbol{\mu}=(\mu^{(i)})_{i=1}^{b(R)}\in (\Z^d)^{b(R)}, \quad \mu^{(i)}=(\mu^{(i)}_j)_{j=1}^d \in \Z^d.
\end{equation}
As usual, a signature $\boldsymbol{\mu}=(\mu^{(i)})_i$ is dominant (\resp nonnegative) if every $\mu^{(i)}$ is; the diagonal matrix with signature $\boldsymbol{\mu}$ is $D_{\boldsymbol{\mu}}=\sum_{i=1}^{b(R)} \mathrm{diag}(T_i^{\mu^{(i)}_1},\dots, T_i^{\mu^{(i)}_d})$; the Schubert cell and the Iwahori cell of signature of $\boldsymbol{\mu}$ are $X_{\boldsymbol{\mu}}=\prod_{i=1}^{b(R)} X_{\mu^{(i)}}$ and $I_{\boldsymbol{\mu}}=\prod_{i=1}^{b(R)} I_{\mu^{(i)}}$. The dimension of both $X_{\boldsymbol{\mu}}$ and $I_{\boldsymbol{\mu}}$ is $\ell(\boldsymbol{\mu}):=\sum_{i=1}^{b(R)} \ell(\mu^{(i)})$. In addition, the \textbf{size} of a signature $\boldsymbol{\mu}$ is defined as
\begin{equation}
    \abs{\boldsymbol{\mu}}:=\sum_{i=1}^{b(R)} \abs{\mu^{(i)}}:=\sum_{i=1}^{b(R)} \sum_{j=1}^d \mu^{(i)}_j. 
\end{equation}
If $\tl L\in X_{\boldsymbol{\mu}}^\circ$, then $[\tl R^d:\tl L]=\abs{\boldsymbol{\mu}}$.

For $\tl L\in X$, we say the signature of $\tl L$ is the unique $\boldsymbol{\mu}_{\tl L}\in \Z^d$ such that $\tl L\in X_{{\boldsymbol{\mu}}_{\tl L}}^\circ$. The \textbf{Iwahori element} of $\tl L$, denoted by $\iota_{\tl L}$, is the unique element in $I_{\boldsymbol{\mu}}$ such that $\iota_{\tl L} D_{{\boldsymbol{\mu}}_{\tl L}} \tl R^d = \tl L$. The \textbf{normal form} of $\tl L$ is the matrix $x_{\tl L}:= \iota_{\tl L} D_{{\boldsymbol{\mu}}_{\tl L}}$; we have $\tl L=x_{\tl L}\tl R^d$. If $\tl L$ has signature $\boldsymbol{\mu}$, then $\tl L\subeq \tl R^d$ if and only if $\boldsymbol{\mu}\geq 0$, and $\tl L\supeq \tl R^d$ if and only if $-\boldsymbol{\mu}\geq 0$, where the \textbf{negation} $-\boldsymbol{\mu}$ is defined by entrywise negation. For $\tl L\subeq \tl R^d$, the type vector (see \Cref{subsec:dvr}) of $\tl R^d/\tl L$ is given by 
\begin{equation}
    \boldsymbol{type}(\tl R^d/\tl L)=\boldsymbol{\mu_{\tl L}}_{\geq},
\end{equation}
where $\boldsymbol{\mu}_\geq$ denotes the dominant signature associated with $\boldsymbol{\mu}$ defined by sorting the entries of each $\mu^{(i)}$.

\subsection{Padding fiber} Recall the padding fiber $\Pd(\tl L;\ut M)$ parametrizng $\tl R$-lattices $\tl N$ such that $\tl N+\ut M=\tl L$. The following lemma gives an affine cell decomposition of $\Pd(\tl L;\ut M)$, and gives a canonical candidate for the map $h$ in \Cref{lem:padding}. For two nonnegative signatures $\boldsymbol{\mu}, \boldsymbol{\nu}$, we say they are \textbf{orthogonal} to each other, denoted by $\boldsymbol{\mu}\perp \boldsymbol{\nu}$ if
\begin{equation}
    \min\set{\mu^{(i)}_j,\nu^{(i)}_j} = 0 \text{ for all }i,j.
\end{equation}

\begin{lemma}
    \label{lem:padding_fiber_schubert}
    Let $\boldsymbol{\mu}\in \Z^{d\times b(R)}$ be a nonnegative dominant signature, and let $\tl L=\tl R^d$ and $\ut M=D_{\boldsymbol{\mu}}\tl R^d$. Then
    \begin{enumerate}
        \item The padding fiber is a union of Schubert cells:
        \begin{equation}
            \set{\tl N: \tl N+\ut M=\tl L} = \bigsqcup_{\boldsymbol{\nu}\perp \boldsymbol{\mu}} X_{\boldsymbol{\nu}}^\circ.
        \end{equation}
        \item If $\boldsymbol{\nu}\perp \boldsymbol{\mu}$ and $\tl N\in X_{\boldsymbol{\nu}}^\circ$, then
        \begin{equation}
            (x_{\tl N}-\mathrm{id})\tl L\subeq \ut M.
        \end{equation}
    \end{enumerate}
\end{lemma}
\begin{proof}
    It suffices to prove the uni-branch case, so we write $\mu=\boldsymbol{\mu}$ and $\nu=\boldsymbol{\nu}$. Since $\mu$ is dominant, there is $0\leq r\leq d$ such that $\mu_1,\dots,\mu_r>0$ and $\mu_{r+1},\dots,\mu_d=0$. 

    \begin{enumerate}
        \item Let $\nu$ be any nonnegative signature and $\tl N\in X_\nu^\circ$. The lattice $\tl N+\ut M$ is the $\tl R$-column span of the matrix $B=\bmat{x_{\tl N} & D_\mu}$. By Nakayama's lemma, $\tl N+\ut M=\tl R^d$ if and only if $B$ is full rank modulo $T$. By inspecting \eqref{eq:normal-form}, $B$ modulo $T$ is of the form
        \begin{equation}
            \bmat{x_{\tl N} & D_\mu} \bmod T = \bmat{L_{r\times r} & 0 & 0_{r\times r} & 0 \\
            * & *_{(d-r)\times (d-r)} & 0 & 1_{(d-r)\times (d-r)}},
        \end{equation}
        where $L_{r\times r}$ is lower triangular. Moreover, the diagonal of $L_{r\times r}$ is all one if $\nu_1=\dots=\nu_r=0$, namely, $\nu\perp \mu$, while the diagonal of $L_{r\times r}$ contains zero if $\nu \not\perp\mu$. As a result, $B$ is full rank modulo $T$ if and only if $\nu \perp \mu$, proving the claim.
        \item Assume $\nu\perp \mu$ and $\tl N\in X_\nu^\circ$. By inspecting \eqref{eq:normal-form}, the top $r$ rows of $x_{\tl N}$ is the matrix $\bmat{1_{r\times r} & 0_{r\times (d-r)}}$, so the top $r$ rows of the matrix $x_{\tl N}-\mathrm{id}$ is zero. As a result, the $\tl R$-column span of $x_{\tl N}-\mathrm{id}$ is contained in $0^r\oplus \tl R^{d-r}$, which is contained in $\ut M$. \qedhere
    \end{enumerate}
\end{proof}

We now prove the motivic analogue of \Cref{prop:padding-fiber-gen-arithmetic}. 

\begin{corollary}\label{cor:padding-fiber-gen}
    Fix $\tl R$-lattices $\tl L\supeq \ut M$ such that $\boldsymbol{\rk}(\tl L/\ut M)=\boldsymbol{r}$. Then we have an identity in $\KVar{k}[[\boldsymbol{t}]]=\KVar{k}[[t_1,\dots,t_{b(R)}]]$:
    \begin{equation}\label{eq:padding-fiber-gen}
        \sum_{\boldsymbol{n}\in \N^{b(R)}} \bracks*{\set*{\tl N: \mat{\tl N+\ut M=\tl L\\ [[\tl L:\tl N]]=\boldsymbol{n}}}}\, \boldsymbol{t}^{\boldsymbol{n}} = \prod_{i=1}^{b(R)} \frac{1}{(\L^{r_i}t_i;\L)_{d-r_i}}.
    \end{equation}
    In particular, taking $t_i=t$, we get
    \begin{equation}
        \sum_{n} \bracks*{\set*{\tl N: \mat{\tl N+\ut M=\tl L\\ [\tl L:\tl N]=n}}}\, t^n = \prod_{i=1}^{b(R)} \frac{1}{(\L^{r_i}t;\L)_{d-r_i}}.
    \end{equation}
\end{corollary}
\begin{proof}
    By passing to isomorphic copies, we may assume $\tl L,\ut M$ are in the setting of \Cref{lem:padding_fiber_schubert}. Thus, each moduli space on the left-hand side of \eqref{eq:padding-fiber-gen} is a union of Schubert cells, so its motive is a polynomial of $\L$. Therefore, the point-counting formula (\Cref{prop:padding-fiber-gen-arithmetic}) directly upgrades to the motivic formula.
\end{proof}

\begin{remark}
    Instead of resorting to the point-count version, we can directly deal with the combinatorics of Schubert cells to find the exact formula, using the ``spiral shifing operators'' from \cite{huangjiang2022a}. Let $b(R)=1$ and assume the setting in the proof of \Cref{lem:padding_fiber_schubert}, so that $\mu_1,\dots,\mu_r>0$, $\mu_{r+1},\dots,\mu_d=0$. Note that $r=\rk(\tl L/\ut M)$, and $\nu\perp \mu$ if and only if $\nu_1=\dots=\nu_r=0$. Then \Cref{cor:padding-fiber-gen} reduces to the combinatorial identity
    \begin{equation}
        \sum_{\substack{\nu\in \N^d\\ \nu_1=\dots=\nu_r=0}} t^{\abs{\nu}} q^{\ell(\nu)} = \frac{1}{(tq^r;q)_{d-r}}.
    \end{equation}
    The left-hand side sums over the orbit of the zero signature under the spiral shifting operators $g_{r+1},\dots,g_d$ as in \cite{huangjiang2022a}. The ``content'' of the zero signature is $1$, and the content of $g_j$ is $tq^{j-1}$. By \cite[Theorem~5.4]{huangjiang2022a}, the left-hand side sum has the product formula
    \begin{equation}
        \Cont(0)\cdot \frac{1}{(1-\Cont(g_{r+1}))\cdots (1-\Cont(g_d))} = \frac{1}{(1-tq^r)\cdots (1-tq^{d-1})} = \frac{1}{(tq^r;q)_{d-r}},
    \end{equation}
    as required.
\end{remark}

\subsection{Proof of \Cref{thm:rationality-motivic}} Let an $R$-lattice $M$ and an $\tl R$-lattice $\ut M\subeq M$ be given. By passing to isomorphic copies, we may assume $\ut M=\tl R^d$. We start with defining a combinatorial stratification of $\Quot^R(M)$, whose purpose will be clear later. For every nonnegative signature $\boldsymbol{\mu}\in \N^{d\times b(R)}$, fix any permutation matrix\footnote{A permutation matrix in $\GL_d(\tl R)$ is a matrix whose $i$-th branch is a permutation matrix in $\GL_d(\tl R_i)$.} $\sigma_{\boldsymbol{\mu}}$ such that $\sigma_{\boldsymbol{\mu}} D_{\boldsymbol{\mu}} \sigma_{\boldsymbol{\mu}}^{-1}= D_{{\boldsymbol{\mu}}_\geq}$, so that $\sigma_{\boldsymbol{\mu}} D_{\boldsymbol{\mu}} \tl R^d= D_{{\boldsymbol{\mu}}_\geq} \tl R^d$. Define the index set for our stratification:
\begin{equation}
    Y_n:=\set*{\tau=(\boldsymbol{\lambda},\boldsymbol{\mu},\boldsymbol{\nu},n):\mat{\boldsymbol{\lambda},\boldsymbol{\mu},\boldsymbol{\nu}\in \N^{d\times b(R)}\\ \boldsymbol{\lambda}_{\geq}=\boldsymbol{\mu}_{\geq}\\ \boldsymbol{\nu}\perp \boldsymbol{\mu}_\geq}},
\end{equation}
and $Y=\bigsqcup_{n\in \N} Y_n$. Given $\tau\in Y_n$, we define the stratum $\Quot^R(M;\tau)\subeq \Quot^R_n(M)$ to consist of $R$-lattices $N\subeq M$ that are of type $\tau$, whose meaning is detailed below.
\begin{definition}\label{def:stratum-Y}
    For $N\subeq_R M$, let $\tl N=\tl RN$ and $\tl L=\ut M+\tl N$. We say the lattice $N$ is of \textbf{type} $\tau$ if the followings hold:
    \begin{itemize}
        \item $\tl L\in X_{-\boldsymbol{\lambda}}^{\circ}$.
        \item Let $\ut M':=D_{-\boldsymbol{\lambda}}^{-1}\iota_{\tl L}^{-1} \tl R^d$. Then $\ut M'\in X_{\boldsymbol{\mu}}^\circ$.
        \item Let $\tl N'':=\sigma_{\boldsymbol{\mu}} \iota_{\ut M'}^{-1} D_{-\boldsymbol{\lambda}}^{-1} \iota_{\tl L}^{-1} \tl N$. Then $\tl N''\in X_{\boldsymbol{\nu}}^\circ$. 
        \item $[\tl N:N]=n$.
    \end{itemize}
\end{definition}
The definition of $\ut M'$ and $\tl N''$ can be easily extracted from the following diagram:
\begin{equation}\label{diag:iso}
\begin{tikzcd}[column sep=large, row sep=large]
    \tl N \arrow[r, "\simeq"] \arrow[d, hook] &\tl N' \arrow[r, "\simeq"] \arrow[d, hook] &\tl N''\arrow[d, hook]\\
    % First row
    \tl L \arrow[r, "{D_{-\boldsymbol{\lambda}}^{-1}\iota_{\tl L}^{-1}}", "{\simeq}"'] & 
    \tl R^d \arrow[r, "{\sigma_{\boldsymbol{\mu}} \iota_{\ut M'}^{-1}}", "{\simeq}"'] & 
    \tl R^d \\
    % Second row
    \tl R^d \arrow[r, "\simeq"] \arrow[u, hook] & 
    \ut M'  \arrow[r, "\simeq"] \arrow[u, hook]& 
    D_{\boldsymbol{\mu}_\geq} \tl R^d \arrow[u, hook]
\end{tikzcd}
\end{equation}

It is clear that $\Quot^R(M;\tau)$ is a locally closed subset of $\Quot^R_n(M)$, and $\Quot^R_n(M)=\bigsqcup_{\tau \in Y_n} \Quot^R(M;\tau)$. For the latter assertion, we remark that the property $\boldsymbol{\lambda}_\geq = \boldsymbol{\mu}_\geq$ is guaranteed by the $\tl R$-module isomorphism $\tl L/\tl R^d \simeq \tl R^d / \ut M'$. 

For $N\in \Quot^R(M)$, define also
\begin{itemize}
    \item $M'':=\sigma_{\boldsymbol{\mu}} \iota_{\ut M'}^{-1} D_{-\boldsymbol{\lambda}}^{-1} \iota_{\tl L}^{-1} M.$
    \item $N''=\sigma_{\boldsymbol{\mu}} \iota_{\ut M'}^{-1} D_{-\boldsymbol{\lambda}}^{-1} \iota_{\tl L}^{-1} N.$
\end{itemize}
Due to the explicit nature of the expressions, all objects mentioned so far depend regularly on $N$ as long as $N$ varies over a single stratum $\Quot^R(M;\tau)$.

Recall also the moduli space
\begin{equation}
    B_R(M;\ut M):=\set{L\subeq_R M: \tl RL\supeq \ut M}.
\end{equation}
Consider the index set
\begin{equation}
    Z_n:=\set*{\rho=(\boldsymbol{\lambda},\boldsymbol{\mu},n):\mat{\boldsymbol{\lambda},\boldsymbol{\mu},\boldsymbol{\nu}\in \N^{d\times b(R)}\\ \boldsymbol{\lambda}_{\geq}=\boldsymbol{\mu}_{\geq}}},
\end{equation}
and $Z=\bigsqcup_{n\in \N} Z_n$. Given $\rho\in Z$, define the stratum
\begin{equation}
    B_R(M;\ut M;\rho):=\set*{
        L\subeq_R M: \mat{
            \tl L\in X_{-\boldsymbol{\lambda}}^\circ \\ 
            \ut M' \in X_{\boldsymbol{\mu}}\\ [\tl L:L]=n}
        },
\end{equation} 
where $\tl L:=\tl R L$ and $\ut M':=D_{-\boldsymbol{\lambda}}^{-1}\iota_{\tl L}^{-1} \tl R^d$. Again, we have $B_R(M;\ut M)=\bigsqcup_{\rho\in Z} B_R(M;\ut M;\rho)$. 

\begin{lemma}\label{lem:key-decomposition}
    For any $\tau=(\boldsymbol{\lambda},\boldsymbol{\mu},\boldsymbol{\nu},n) \in Y$, let $\rho = (\boldsymbol{\lambda},\boldsymbol{\mu},n)$, then we have an isomorphism of varieties
    \begin{equation}
        \Quot^R(M;\tau)\simeq B_R(M;\ut M;\rho) \times X_{\boldsymbol{\nu}}^\circ.
    \end{equation}
\end{lemma}
\begin{proof}
    The key is that \Cref{lem:padding_fiber_schubert}(b) provides the necessary ingredient to be substituted into the $h$ in \Cref{lem:padding}(b). Starting with it and carefully chasing the diagram \eqref{diag:iso}, we get the required isomorphism:
    \begin{equation}
        N \mapsto (\iota_{\tl L} D_{-\boldsymbol{\lambda}} \iota_{\ut M'} \sigma_{\boldsymbol{\mu}}^{-1} x_{\tl N''}^{-1} N'',\tl N''),
    \end{equation}
    where $\tl L, \ut M', \tl N'', N''$ depend on $N$ according to rules defined before. The inverse map is
    \begin{equation}
        \iota_{\tl L} D_{-\boldsymbol{\lambda}} \iota_{\ut M'} \sigma_{\boldsymbol{\mu}}^{-1} x_{\tl N''} L'' \mapsfrom (L,\tl N''),
    \end{equation}
    where $L'':=\sigma_{\boldsymbol{\mu}} \iota_{\ut M'}^{-1} D_{-\boldsymbol{\lambda}}^{-1} \iota_{\tl L}^{-1} L$.
\end{proof}

We are now ready to prove \Cref{thm:rationality-motivic}. Below, we present a refined version of \Cref{thm:rationality-motivic} that captures the full strength of our approach. For a nonnegative signature $\boldsymbol{\lambda}\in \N^{d\times b(R)}$, let $r_i(\boldsymbol{\lambda})$ denote the number of nonzero entries in $\lambda^{(i)}$.
\begin{theorem}\label{thm:rationality-motivic-fine}
    We have the following identity in 
    \begin{equation}
    \KVar{k}[[\boldsymbol{u},\boldsymbol{v},t]]=\KVar{k}[[u^{(i)}_j, v_i,t:1\leq i\leq b(R),1\leq j\leq d]]:
    \end{equation}
    \begin{equation}\label{eq:rationality-refined}
        \begin{multlined}
            \prod_{i=1}^{b(R)} (v_i;\L)_d \times \sum_{\boldsymbol{m},\boldsymbol{n},n} \bracks*{\set*{N\subeq_R M: \mat{\boldsymbol{type}(\tl L/\ut M)=\boldsymbol{m} \\
        [[\tl L:\tl N]]=\boldsymbol{n}\\
        [\tl N:N]=n\\
        }}}\, \boldsymbol{u}^{\boldsymbol{m}} \boldsymbol{v}^{\boldsymbol{n}} t^n \\
        = \sum_{\rho=(\boldsymbol{\lambda},\boldsymbol{\mu},n)\in Z} [B_R(M;\ut M;\rho)] \boldsymbol{u}^{\boldsymbol{\mu}_\geq} t^n \prod_{i=1}^{b(R)}(v_i;\L)_{r_i(\boldsymbol{\mu}_\geq)},
        \end{multlined}
    \end{equation}
    where $\tl N:=\tl RN$ and $\tl L:=\tl N+\ut M$. Moreover, the right-hand side is a polynomial.
\end{theorem}
\begin{proof}
    For $\boldsymbol{\nu}=(\nu^{(i)})_{i=1}^{b(R)}\in (\N^d)^{b(R)}$, denote $\boldsymbol{n}(\boldsymbol{\nu})=(\abs{\nu^{(i)}})_{i=1}^{b(R)}$. Then the sum in the left-hand side of \eqref{eq:rationality-refined} equals
    \begin{equation}
        \sum_{\tau=(\boldsymbol{\lambda},\boldsymbol{\mu},\boldsymbol{\nu},n)\in Y} [\Quot^R(M;\tau)]\, \boldsymbol{u}^{\boldsymbol{\mu}_\geq} \boldsymbol{v}^{\boldsymbol{n}(\boldsymbol{\nu})} t^n.
    \end{equation}
    By \Cref{lem:key-decomposition} and \Cref{cor:padding-fiber-gen}, the left-sum sum of \eqref{eq:rationality-refined} further becomes
    \begin{equation}
        \sum_{\rho=(\boldsymbol{\lambda},\boldsymbol{\mu},n)\in Z} [B_R(M;\ut M;\rho)] \frac{1}{\prod_{i=1}^{b(R)}(\L^{r_i(\boldsymbol{\mu}_\geq)}v_i;\L)_{d-r_i(\boldsymbol{\mu}_\geq)}} \boldsymbol{u}^{\boldsymbol{\mu}_\geq} t^n.
    \end{equation}
    Simplifying yields the required identity. Finally, the right-hand side of \eqref{eq:rationality-refined} is a polynomial because among the summands in which $B_R(M;\ut M;\rho)$ is nonempty, $n$ is bounded by \Cref{lem:bounded} and $\abs{\boldsymbol{\lambda}}=\abs{\boldsymbol{\mu}}$ are bounded by the fact that $\ut M\subeq \tl L \subeq \tl RM$.
\end{proof}

\begin{proof}
    [Proof of \Cref{thm:rationality-motivic}]
    Note that $[M:N]=[M:\ut M]-[\tl L:\ut M]+[\tl L:\tl N]+[\tl N:N]$. As a result, taking $u^{(i)}_j=t^{-1}$ and $v_i=t$ in \Cref{eq:rationality-refined} and multiply both sides by $t^{[M:\ut M]}$, we get
    \begin{equation}
        \prod_{i=1}^{b(R)} (t;\L)_d \times Z_M(t) = \sum_{\rho=(\boldsymbol{\lambda},\boldsymbol{\mu},n)\in Y} [B_R(M;\ut M;\rho)] t^{n-\abs{\boldsymbol{\mu}_\geq}+[M:\ut M]}\prod_{i=1}^{b(R)}(t;\L)_{r_i(\boldsymbol{\mu}_\geq)}.
    \end{equation}
    By organizing the summands on the right-hand side according to $\boldsymbol{r}(\boldsymbol{\mu}_\geq)$, and noting that $n-\abs{\boldsymbol{\mu}}+[M:\ut M]=[M:L]$ for $L\in B_R(M;\ut M;\rho)$, the right-hand side equals the right-hand side of \eqref{eq:rational-formula-motivic}.
\end{proof}

\section{Duality and reflection principle}\label{sec:func-eq}

In this section, we prove that certain lattice zeta functions satisfy a functional equation, generalizing a result of Yun \cite{yun2013orbital} from rank $1$ to arbitrary rank $d$. Our approach is inspired by the work of Bushnell and Reiner \cite{bushnellreiner1980zeta}, and we pay special attention to the subtleties arising from matrix manipulations and transposition in higher rank. 

Let $R$ be an arithmetic local order, and let $\Omega_R$ denote its dualizing module, which is uniquely defined up to isomorphism \cite[Corollary 21.14]{eisenbudcommutative}. For any $R$-lattice $M$ of rank $d$, we define the \textbf{normalized lattice zeta function} by
\begin{equation}\label{eq:def-normalized-lattice}
    \nu_M(s)=\nu_M^R(s) := \zeta_M^R(s) / \zeta_{\tl R^d}^{\tl R}(s) = \zeta_M^R(s) \prod_{i=1}^{b(R)} (q_i^{-s};q_i)_d,
\end{equation}
which is a Dirichlet polynomial by \Cref{thm:rationality-arithmetic}.

\begin{theorem}[Reflection principle]\label{thm:reflection-local}
    Let $R$ be an arithmetic local order, and let $E=\Omega_R^d$ be the direct sum of $d$ copies of its dualizing module. Let $\Delta = \abs{\tl R/R}$ be the Serre invariant. Then as Dirichlet polynomials, we have
    \begin{equation}
        \nu_E^R(s) = \Delta^{d^2-2ds} \nu_E^R(d-s).
    \end{equation}
\end{theorem}

\subsection{Trace pairing and dual lattice}
As the first step of the proof, we note that any arithmetic local order fits into the setting of \cite[\S 2]{yun2013orbital} in the following sense.
\begin{lemma}\label{lem:z}
    Given an arithmetic local order $R$, there exists a subring $Z\subeq R$ such that
    \begin{enumerate}
        \item $Z$ is a DVR;
        \item $R$ is a module over $Z$ of finite rank;
        \item $K=\Frac(R)$ is a separable algebra over $Q:=\Frac(Z)$.
    \end{enumerate}
\end{lemma}
\begin{proof}
    If $R$ is of characteristic zero, then we simply take $Z=\Zp$. If $R$ is of characteristic $p>0$, then for $1\leq i\leq b(R)$, by Cohen structure theorem, $\tl R_i\simeq \F_{q_i}[[T_i]]$, where $q_i$ are powers of $p$. By \Cref{lem:conductor-finite}, for $c\gg 0$, $T^c:=\sum_{i=1}^{b(R)} T_i^c$ lies in $R$; in particular, we may choose such $c$ that is not divisible by $p$. Take $Z=\Fp[[T^c]]$, then $Z\subeq R$. The algebra $K/Q$ is a direct product of field extensions $Q\incl K_i$, which can be identified with $\Fp((T^c))\incl \F_{q_i}((T))$. This is a finite separable field extension when $p\nmid c$, finishing the proof.
\end{proof}

Throughout the rest of this section, we will fix $Z$ (and thus $Q$) as above, even though the statement of \Cref{thm:reflection-local} does not depend on the choice of $Z$. Let us recall the notion of trace pairing on $K$ following \cite[\S 2]{yun2013orbital}, which will depend on $Q$. Consider the trace map $\tr_{K/Q}:K\to Q$ that defines $\tr_{K/Q}(\alpha)$ to be the trace of the multiplication-by-$\alpha$ map on $K$, viewed as a $Q$-linear endomorphism. Let $D$ be a generator for the different ideal of $K/Q$ as an ideal of $\tl R$. Define the (modified) trace pairing $K\times K\to Q$ by $(x,y)\mapsto \tr_{K/Q}(D^{-1}xy)$; it is nondegenerate because $K$ is a separable algebra over $Q$. For any \textbf{fractional ideal} $I$ of $R$ (namely, a rank-one $R$-lattice in $K$), define $I^{\vee}=\Hom_Z(I,Z)$ and view it as a fractional ideal of $R$ by 
\begin{equation}
    I^{\vee}:=\set{x\in K: \tr_{K/Q}(D^{-1}xI)\subeq Z}.
\end{equation}
The choice of $D$ ensures that $(\tl R)^\vee=\tl R$ as fractional ideals. 

For any $R$-lattice $L$ in a finite-rank free $K$-module $V$, consider the $L^{\vee}:=\Hom_Z(L,Z)$ with an $R$-module structure induced from the one on $L$. By the general tensor-hom duality, we have a canonical isomorphism of $R$-modules $L^{\vee}\simeq \Hom_R(L,R^\vee)$. We view $L^\vee$ as an $R$-lattice in $\Hom_K(V,K)$ using the natural inclusion $\Hom_R(L,R^\vee)\to \Hom_K(V,K)$. As a subset of $\Hom_K(V,K)$, we have $L^\vee=\set{\theta\in \Hom_K(V,K): \theta(L)\subeq R^\vee}$. We have the double dual property $(L^\vee)^\vee=L$ as lattices in $K^d$ because any $R$-lattice is free of finite rank over $Z$, and $\Hom_Z(\cdot,Z)$ clearly has this property on free $Z$-modules of finite rank. For this reason, $R^\vee$ is a dualizing module of $R$, and thus $\Omega_R\simeq R^\vee$. From now on, we will identify $\Omega_R$ with $R^\vee$ as a fractional ideal.

For two $R$-lattices $L_1\supeq L_2$ of $V$, we define $(L_1:L_2)=\abs{L_1/L_2}$. This notion is extended uniquely to any pair of $R$-lattices in $V$ by the property $(L_1:L_2)(L_2:L_3)(L_3:L_1)=1$, cf.~\Cref{lem:colength}. We call $(L_1:L_2)$ the \textbf{index} of $L_2$ in $L_1$, even when $L_1$ does not contain $L_2$. 

For two $R$-lattices $L, M$ in free $K$-modules $V,W$, respectively, we view $\Hom(L,M)$ as a subset of $\Hom_K(V,W)$ by
\begin{equation}
    \Hom(L,M)=\set{x\in \Hom_K(V,W):xL\subeq M}.
\end{equation}
It will be important to keep track of different copies of free $K$-modules. To this end, we use the standard notation $K^{d\times d}=\Mat_d(K)$, $K^{d\times 1}$ for the space of column vectors, $K^{1\times d}$ for the space of row vectors. We canonically have $\Hom_K(K^{d\times 1},K)=K^{1\times d}$, so if $L\subeq K^{d\times 1}$ is an $R$-lattice, then canonically $L^\vee \subeq K^{1\times d}$. Let $(\cdot)^T$ denote the transpose operator.

\subsection{The proof via harmonic analysis}
The proof of \Cref{thm:reflection-local} requires refining the zeta function to sum over individual isomorphism classes, and then re-interpreting this sum as a zeta integral.

\begin{definition}
    Let $L,M$ be two torsion-free $R$-modules of rank $d$. The \textbf{restricted lattice zeta function} is
    \begin{equation}
        \zeta_M(s;L)=\zeta_M^R(s;L):=\sum_{\substack{N\subeq_R M\\ N\simeq_R L}} (M:N)^{-s}.
    \end{equation}
    The lattice zeta function is thus the sum over all isomorphism classes: $\zeta_M(s) = \sum_{[L]} \zeta_M(s;L)$.
\end{definition}

To express this as an integral, we introduce the standard notions from harmonic analysis. Let $A:=\Mat_d(K)$, so $A^\times = \GL_d(K)$. Let $\mu=\mathrm{d}x$ be the additive Haar measure on $K$ normalized so that $\mu(\tl R)=1$. The measure extends to $K^{d\times 1}$, $K^{1\times d}$, and $A=K^{d\times d}$ as product measures, so in particular $\mu(\Mat_d(\tl R))=1$. Let $d^\times x$ be the multiplicative Haar measure on $A^\times$ normalized so that $\mu^\times(\GL_d(\tl R)) = \prod_{i=1}^{b(R)} (q_i^{-1};q_i^{-1})_d^{-1}$. For $x\in A^\times$, the norm is defined by $\norm{x}_A := (\mathcal{N}:x\mathcal{N})^{-1}$ for any $R$-lattice $\mathcal{N}$ in $A$, or equivalently $\norm{x}_A=(N:xN)^{-d}$ for any $R$-lattice $N$ in $K^d$. These are related by $d^\times x = \norm{x}_A^{-1} dx$.

The restricted zeta function can now be written as a zeta integral on $A$. A \textbf{Schwartz--Bruhat} (SB) function on $A$ is a locally constant and compactly supported function from $A$ to $\C$. The \textbf{zeta integral} of an SB function on $A$ is defined as
\begin{equation}
    Z(\Phi;s) := \int_{A^\times} \Phi(x) \norm{x}_A^s d^\times x.
\end{equation}

\begin{lemma}[{\cite[Eq. (11)]{bushnellreiner1980zeta}}]
    For lattices $L,M$ in $V=K^{d\times 1}$,
    \begin{equation}\label{eq:restricted-zeta-integral}
        \zeta_{M}(s;L)=\mu^{\times}(\Aut L)^{-1}(M:L)^{-s} Z(\one_{\Hom(L,M)};s/d).
        %\mu^{\times}(\Aut L)^{-1}(M:L)^{-s}\int_{x\in A^{\times}} \one_{\Inj(L,M)}(x)\norm{x}_A^{s/d}d^{\times}x,
    \end{equation}
    where $\one_{\Hom(L,M)}$ is the indicator function, and $\Aut L:=\set{x\in A:xL=L}$.
\end{lemma}

The functional equation arises from Tate's thesis for zeta integrals of Schwartz--Bruhat functions on $A$, viewed as a semisimple algebra over $Q$.  Fix an additive character $\chi:Q\to \C^\times$ that is trivial on $Z$ but nontrivial on any fractional ideal $I\nsubseteq Z$. For a trace function $\tr:A\to Q$ such that the pairing $A\times A\to Q, (x,y)\mapsto \tr(xy)$ is non-degenerate, the \textbf{Fourier transform} of an SB function $\Phi$ on $A$ (with respect to $\chi, \tr$, and $\mathrm{d}x$) is defined as
\begin{equation}
    \hhat{\Phi}(y)=\int_A \Phi(x)\chi(\tr(xy)) \, \mathrm{d}x,
\end{equation}
which is again an SB function on $A$. 

\begin{theorem}[{Local function equation for semisimple algebras \cite[Proposition 2]{bushnellreiner1980zeta}}] \label{thm:tate-func-eq}
    Let $\Phi$ and $\Psi$ be any two Schwartz--Bruhat functions on $A$. Their zeta integrals admit meromorphic continuation and satisfy the functional equation
    \begin{equation}\label{eq:tate-func-eq}
        \frac{Z(\Phi;s)}{Z(\Psi;s)}=\frac{Z(\hhat\Phi;1-s)}{Z(\hhat\Psi;1-s)}.
    \end{equation}
\end{theorem}

To proceed with the proof, we will take $\tr$ to be the modified trace $\tr:A\to Q$ defined by
\begin{equation}
    \tr(x):=\tr_{K/Q}(D^{-1}\tr_{A/K}(x)),
\end{equation}
where $\tr_{A/K}=\tr_{\Mat_d(K)/K}$ is the usual trace of matrices. It is non-degenerate because the trace pairing on $K/Q$ is non-degenerate by the separability of $K/Q$, and the usual trace pairing on the matrix algebra $A/K$ is non-degenerate.

The proof of \Cref{thm:reflection-local} proceeds by applying \Cref{thm:tate-func-eq} with specific choices for $\Phi$ and $\Psi$. A key calculation shows that the Fourier transform of the indicator function $\one_{\Hom(L,E)}$ (where $E=\Omega_R^d$) is proportional to the indicator function of the dual lattice, $(L^\vee)^T$. 

\begin{lemma}\label{lem:fourier-transform-indicator}
    Let $L$ be an $R$-lattice in $V=K^{d\times 1}$, and $E=\Omega_R^{d\times 1} \subeq V$. The Fourier transform of the indicator function of $\Hom(L,E)$ is given by
    \begin{equation}
        \hhat{\one}_{\Hom(L,E)} = \mu(L)^{-d} \one_{L^{1\times d}}=\mu(L)^{-d} \one_{\Hom((L^\vee)^T,E)^T},
    \end{equation}
    where $L^{1\times d}$ is the set of matrices in $K^{d\times d}$ whose columns all lie in $L$.
\end{lemma}
\begin{proof}
    First, note that $\Hom(L,E)=\Hom(L,\Omega_R)^{d\times 1}=(L^\vee)^{d\times 1}$, the set of matrices $x\in A$ whose rows $x^i$ belong to $L^\vee$.
    
    The Fourier transform of its indicator function $\one_{\Hom(L,E)}$ at a point $y \in A$ (with $i$-th column denoted by $y_i$) is then
    \begin{align}
        \hhat{\one}_{\Hom(L,E)}(y) &= \int_{x^1,\dots,x^d \in L^\vee} \chi(\tr_{K/Q}(D^{-1}\tr_{A/K}(xy))) \, \mathrm{d}x^1\cdots \mathrm{d}x^d \\
        &=\int_{x^1,\dots,x^d \in L^\vee} \chi(\tr_{K/Q}(D^{-1}\sum_{i=1}^d x^i y_i)) \, \mathrm{d}x^1\cdots \mathrm{d}x^d\\
        &= \prod_{i=1}^d \int_{x^i \in L^\vee} \chi(\tr_{K/Q}(D^{-1}x^i y_i)) \, \mathrm{d}x^i.\label{eq:fourier-transform-indicator}
    \end{align}
    The inner integral is $\hhat{\one}_{L^\vee}(y_i)$. We evaluate it by cases.

    \emph{Case 1: $y_i \in L$.} By the double dual property, $L = (L^\vee)^\vee=\set{v \in V : u v \in R^\vee \text{ for all } u \in L^\vee}$. Therefore, for any $x^i \in L^\vee$, the product $x^i y_i$ is in $R^\vee$. By the definition of $R^\vee$, this means $\tr_{K/Q}(D^{-1}x^i y_i) \in Z$. Since the character $\chi$ is trivial on $Z$, the integrand is always 1. The integral becomes $\int_{L^\vee} 1 \, \mathrm{d}x^i = \mu(L^\vee)=\mu(L)^{-1}$.

    \emph{Case 2: $y_i \notin L$.} Since $y_i \notin (L^\vee)^\vee$, there must exist some $x^i_0 \in L^\vee$ such that $x^i_0 y_i \notin R^\vee$, so $\tr_{K/Q}(D^{-1}x^i_0 y_i)\notin Z$. By our assumption of $\chi$, this implies that the function $x^i\mapsto \chi(\tr_{K/Q}(D^{-1}x^i y_i))$ defines a non-trivial character on the additive group $L^\vee$. The integral of a non-trivial character over the group is zero.

    Combining these cases, the $i$-th inner integral of \eqref{eq:fourier-transform-indicator} evaluates to $\mu(L)^{-1} \one_L(y_i)$, so the product is $\mu(L)^{-d}\one_{L^{1\times d}}(y)$, proving the first claimed equality.

    Finally, we verify $L^{1\times d}=\Hom((L^\vee)^T,E)^T$. This amounts to showing that a matrix $x\in A$ has all columns $x_i$ in $L$ if and only if $x_i^T\in \Hom((L^\vee)^T,R^\vee)$. The latter condition is equivalent to $x_i^T(L^\vee)^T\subeq R^\vee$, which is the same as $L^\vee x_i \subeq R^\vee$. By the double dual property, this is equivalent to $x_i\in L$. Thus, we have $L^{1\times d}=\Hom((L^\vee)^T,E)^T$, and the second claimed equality follows.
\end{proof}

A second necessary ingredient relates the automorphism groups of a lattice and its dual.

\begin{lemma}\label{lem:aut-of-dual}
    For any $R$-lattice $L$ in $K^{d\times 1}$, we have $\mu^\times(\Aut L)=\mu^\times(\Aut ((L^\vee)^T))$.
\end{lemma}
\begin{proof}
    The transpose map $x \mapsto x^T$ on $A=\Mat_d(K)$ is a measure-preserving automorphism. Thus, it suffices to show that $\Aut((L^\vee)^T) = (\Aut L)^T$. Since $\Aut L$ is the group of units of the subalgebra $\End L:=\Hom(L,L)\subeq A$, it suffices to show $\End((L^\vee)^T)=(\End L)^T$. This is equivalent to showing that for $x \in A$, we have $xL\subeq L$ if and only if $x^T(L^\vee)^T\subeq (L^\vee)^T$. The latter condition is equivalent to $L^\vee x \subeq L^\vee$. By definition, $L^\vee x \subeq L^\vee$ if and only if $(L^\vee x) L \subeq R^\vee$. By similar reason and the double dual property, $xL\subeq L$ if and only if $L^\vee(xL) \subeq R^\vee$. Since the action of matrices is associative, the two conditions are equivalent.
\end{proof}

We now state and prove the reflection principle for the restricted zeta function. As usual, define the \textbf{normalized restricted zeta function} by
\begin{equation}
    \nu_M(s;L)=\nu_M^R(s;L):=\frac{\zeta_M^R(s;L)}{\zeta_{\tl R^d}^{\tl R}(s)}.
\end{equation}

\begin{theorem}[Restricted reflection principle]\label{thm:reflection-local-restricted}
    Let $R$ be an arithmetic local order, $E=\Omega_R^d$, and $\Delta = \abs{\tl R/R}$ be the Serre invariant. For any torsion-free $R$-module $L$ of rank $d$, we have
    \begin{equation}
        \nu_E^R(s;L) = \Delta^{d^2-2ds} \nu_E^R(d-s;(L^\vee)^T).
    \end{equation}
\end{theorem}
\begin{proof}
    Since the result depends only on the isomorphism class of $L$, we may assume $L$ is a fixed lattice in $V=K^{d\times 1}$. We apply \Cref{thm:tate-func-eq} with the choices $\Phi = \one_{\Hom(L,E)}$ and $\Psi = \one_{\Mat_d(\tl R)} = \one_{\Hom(\tl R^d,\tl R^d)}$, and substitute $s \mapsto s/d$. By \eqref{eq:restricted-zeta-integral}, the left-hand side of \eqref{eq:tate-func-eq} becomes
    \begin{equation}
        \frac{Z(\Phi;s/d)}{Z(\Psi;s/d)} = \frac{\mu^\times(\Aut L) (E:L)^s \zeta_E^R(s;L)}{\mu^\times(\GL_d(\tl R)) \zeta_{\tl R^d}^{\tl R}(s)}.
    \end{equation}

    For the right-hand side of \eqref{eq:tate-func-eq}, we use \Cref{lem:fourier-transform-indicator}. The numerator is
    \begin{equation}
        Z(\hhat\Phi;1-s/d) = \mu(L)^{-d} Z(\one_{\Hom((L^\vee)^T,E)^T}; 1-s/d) = \mu(L)^{-d} Z(\one_{\Hom((L^\vee)^T,E)}; 1-s/d)
    \end{equation}
    since transposition preserves both measure and norm. By \eqref{eq:restricted-zeta-integral} and \Cref{lem:aut-of-dual}, we have
    \begin{equation}
        Z(\hhat\Phi;1-s/d) = \mu(L)^{-d} \mu^\times(\Aut L) (E:(L^\vee)^T)^{d-s} \zeta_E^R(d-s; (L^\vee)^T).
    \end{equation}

    For the denominator, note that $\hhat{\one}_{\Mat_d(\tl R)} = \one_{\Mat_d(\tl R)}$ (by direct computation or by applying \Cref{lem:fourier-transform-indicator} to $R = \tl R$, $L = E = \tl R^d$). Thus,
    \begin{equation}
        Z(\hhat\Psi;1-s/d) = \mu^\times(\GL_d(\tl R)) \zeta_{\tl R^d}^{\tl R}(d-s).
    \end{equation}

    Equating both sides of \eqref{eq:tate-func-eq} and simplifying using $(L:M) = \mu(L)\mu(M)^{-1}$, $\mu(L)\mu(L^\vee) = 1$, and $\mu(E) = \Delta^d$ (the last two follow from $(\tl R)^\vee = \tl R$ and $\mu(\tl R) = 1$), we obtain the desired result.
\end{proof}

\begin{proof}[Proof of \Cref{thm:reflection-local}]
    The map $[L] \mapsto [(L^\vee)^T]$ is an involution on the finite set of isomorphism classes of $R$-lattices of rank $d$. The total normalized zeta function is the sum of the restricted ones over all isomorphism classes. Summing the identity from \Cref{thm:reflection-local-restricted} over all $[L]$ gives:
    \begin{equation}
        \nu_E^R(s) = \sum_{[L]} \nu_E^R(s;L) = \sum_{[L]} \Delta^{d^2-2ds} \nu_E^R(d-s;(L^\vee)^T) = \Delta^{d^2-2ds} \nu_E^R(d-s). \qedhere
    \end{equation}
\end{proof}

\begin{remark}\label{rmk:motivic-fourier}
    The motivic analogue of \Cref{thm:reflection-local} in rank $d$ would assert that if $R/k$ is a geometric local order, then the \emph{polynomial} (in light of \Cref{thm:rationality-motivic})
    \begin{equation}
        N^R_E(t) := \frac{Z^R_E(t)}{Z^{\tl R}_{\tl R^d}(t)} \in \KVar{k}[t],
    \end{equation}
    where $E = \Omega_R^d$, satisfies the functional equation
    \begin{equation}\label{eq:motivic-reflection}
        N^R_E(t) = \L^{d^2 \delta} t^{2d\delta} N^R_E(\L^{-d} t^{-1}),
    \end{equation}
    where $\delta = [\tl R:R]$ is the Serre invariant. The rank 1 case is known \cite[Proposition 15]{goettscheshende2014refined}, at least when $R$ is Gorenstein.

    However, motivicizing our proof in rank $d>1$ is subtle. Recall that our proof requires summing \Cref{thm:reflection-local-restricted} over isomorphism classes of rank $d$ torsion-free $R$-modules. The moduli space of rank $d$ torsion-free $R$-modules is a stack, not a scheme, if $d>1$. Even if there are motivic Fourier transform techniques (such as those in \cite{cll2016motivic}), it is unclear if they are directly applicable to reach a motivic version of \Cref{thm:reflection-local-restricted} \emph{relative} to such a stack. In rank $1$, the corresponding moduli space, the compactified Jacobian, is a scheme, which plays a crucial role in the argument in \cite{goettscheshende2014refined}. Therefore, while the statement is natural, the motivic analogue of \Cref{thm:reflection-local} in higher rank remains an open problem.
\end{remark}

\section{Preliminaries on explicit computations}\label{sec:prelim}
In this section, we collect the necessary combinatorial tools and general properties of the zeta functions to prepare for the explicit computations in \Cref{sec:torus}.
\subsection{Basic hypergeometric series}\label{subsec:hypergeom}
We begin by defining the standard notation used in the theory of basic hypergeometric series. The $q$-Pochhammer symbol is defined for $n \in \Z_{\geq 0}\cup \set{\infty}$ by
\begin{equation}
    (a;q)_n := \prod_{k=0}^{n-1} (1-aq^k).
\end{equation}
The $q$-binomial coefficient is defined for integers $n, k$ as
\begin{equation}
    \qbinom{n}{k}_q := \frac{(q;q)_n}{(q;q)_k (q;q)_{n-k}} \in \N[q].
\end{equation}
Throughout our computations, we will use the standard convention that $1/(q;q)_n = 0$ if $n$ is a negative integer. This convention naturally truncates summation ranges, and also implies that $\qbinom{n}{k}_q=0$ if $k<0$ or $k>n$.

\subsection{Combinatorics of DVR modules and Hall polynomials}\label{subsec:dvr}
The key reason why we can derive explicit formulas in \Cref{sec:torus} is that we are able to reduce our module counting problems over orders to combinatorics of modules over certain DVR quotients. Let $(V,\pi,\Fq)$ be a DVR with uniformizer $\pi$ and residue field $\Fq$.
\begin{definition}
    Any finite(-cardinality) $V$-module $M$ is isomorphic to a direct sum $M \simeq \bigoplus_{i=1}^l V/\pi^{\lambda_i} V$ for a unique partition $\lambda=(\lambda_1 \geq \lambda_2 \geq \dots \geq \lambda_l > 0)$. This partition $\lambda$ is called the \textbf{type} of $M$. For a submodule $N \subeq M$, the \textbf{cotype} of $N$ in $M$ is the type of the quotient module $M/N$.
\end{definition}

For a partition $\lambda$, we denote its size by $\abs{\lambda} := \sum \lambda_i$ and its conjugate partition by $\lambda'$. As a result, the number of its parts is $\lambda'_1$, and it is also the rank as well as the minimal number of generators of a module of type $\lambda$. Recall the basic fact about counting surjections:
\begin{lemma}\label{lem:sur_probability}
    For $d,r\geq 0$, the number of surjective linear maps from $\Fq^d$ to $\Fq^r$ is given by 
    \begin{equation}
        q^{dr} \frac{(q^{-1};q^{-1})_d}{(q^{-1};q^{-1})_{d-r}}.
    \end{equation}
    Note that this expression is zero if $d<r$. Consequently, by Nakayama's lemma, for a $V$-module $M$ of type $\mu$,
    \begin{equation}
        \abs{\Surj_V(V^d,M)}=q^{d\abs{\mu}} \frac{(q^{-1};q^{-1})_d}{(q^{-1};q^{-1})_{d-\mu'_1}}.
    \end{equation}
\end{lemma}

The number of submodules of a given type and cotype is governed by a universal polynomial in the size of the residue field.

\begin{definition-theorem}[{\cite[Chapter II]{macdonaldsymmetric}}]\label{def:hall-poly}
    Given partitions $\lambda, \mu, \nu$, there exists a universal polynomial $g^\lambda_{\mu\nu}(q) \in \N[q]$, called the \textbf{Hall polynomial}, such that for any DVR $V$ with residue field $\Fq$, the value $g^\lambda_{\mu\nu}(q)$ is the number of submodules of type $\mu$ and cotype $\nu$ in a fixed $V$-module of type $\lambda$. Moreover, $g^\lambda_{\mu\nu}=g^\lambda_{\nu\mu}$.
\end{definition-theorem}

In our computations, we will only need the sum of these polynomials over all possible cotypes.

\begin{theorem}[{\cite{warnaar2013remarks}}]\label{thm:hall_skew}
    For partitions $\mu, \lambda$ and a DVR $(V,\pi,\Fq)$, the number of submodules of type $\mu$ in a fixed module of type $\lambda$ is given by the \textbf{skew Hall polynomial}
    \begin{equation}
        g^\lambda_\mu(q) := \sum_\nu g^\lambda_{\mu\nu}(q) = q^{\sum_{i\ge 1}\mu_{i}'(\lambda_{i}'-\mu_{i}')}\prod_{i\ge 1}\qbinom{\lambda_{i}'-\mu_{i+1}'}{\lambda_{i}'-\mu_{i}'}_{q^{-1}}.
    \end{equation}
    Note that $g^\lambda_\mu(q)\neq 0$ if and only if $\mu\subeq \lambda$.
\end{theorem}
By the symmetry $g^\lambda_{\mu\nu}=g^\lambda_{\nu\mu}$, the skew Hall polynomial $g^\lambda_\mu(q)$ also equals to the number of submodules of \emph{cotype} $\mu$ in a fixed module of type $\lambda$.

\begin{remark}\label{rmk:hall_in_box}
    An important special case is when the ambient module has type $\lambda=(m^d):=(\underbrace{m,\dots,m}_d)$, corresponding to a $d \times m$ rectangular Young diagram. For any sub-partition $\mu \subeq (m^d)$, the Hall polynomial $g^{(m^d)}_{\mu\nu}(q)$ is nonzero for precisely one partition $\nu$, which corresponds to the complement of $\mu$ in the box. We denote this by $\nu=(m^d)-\mu$. In this case, $g^{(m^d)}_\mu(q) = g^{(m^d)}_{\mu, (m^d)-\mu}(q)$.
\end{remark}

A \textbf{DVR quotient} is a quotient of the form $A=V/\pi^n V$ for some $n\geq 0$. A finite module over $A$ is simply a finite $V$-module whose type $\lambda$ satisfies $\lambda_1\leq n$. By the type or cotype of a module over $A$, we mean the type or cotype of the corresponding module over $V$. If $A$ is a product of DVR quotients $A=\prod_{i=1}^b A_i$, then the \textbf{type vector} of a finite module over $A$ is the $b$-tuple of partitions whose $i$-th partition is the type of $M\otimes_A A_i$ over $A_i$. For $b$-tuples of partitions $\boldsymbol{\lambda}=(\lambda^{(i)})_{i=1}^b,\boldsymbol{\mu}=(\mu^{(i)})_{i=1}^b$, define $g^{\boldsymbol{\lambda}}_{\boldsymbol{\mu}}=\prod_{i=1}^b g^{\lambda^{(i)}}_{\mu^{(i)}}$.

We conclude this subsection with some motivic lemmas.
\begin{lemma}\label{lem:hall-in-box-motivic}
    Let $k$ be an algebraically closed field, $\tl R=\prod_{i=1}^{b}k[[T_i]]$, $\boldsymbol{\lambda}=(\lambda^{(i)})_{i=1}^{b(R)}$ be a $b$-tuple of partitions, each with at most $d$-parts, and $\tl M$ is an $\tl R$-lattice of rank $d$. Then
    \begin{equation}\label{eq:hall-in-box-motivic}
        \set{\tl L\subeq \tl M:\boldsymbol{type}(\tl M/\tl L)=\boldsymbol{\lambda}}
    \end{equation}
    is a disjoint union of affine cells. In particular, its motive is $g^{((m^d),\dots,(m^d))}_{\boldsymbol{\lambda}}(\L)$ for any $m\geq \max_i \lambda^{(i)}_1$.
\end{lemma}
\begin{proof}
    Assume $\tl M=\tl R^d$, then by \Cref{subsec:schubert}, the moduli space \eqref{eq:hall-in-box-motivic} is a disjoint union of Schubert cells. The assertion about its motive then follows from the finite-field counterpart.
\end{proof}

\begin{lemma}\label{lem:skew-hall-type-motivic}
    Let $k$ be an algebraically closed field, $\tl R=\prod_{i=1}^{b}k[[T_i]]$, $\boldsymbol{\lambda},\boldsymbol{\mu}$ be $b(R)$-tuples of partitions, each with at most $d$-parts, and $\ut M,\tl L$ are two fixed $\tl R$-lattices of rank $d$ such that $\boldsymbol{type}(\tl L/\ut M)=\boldsymbol{\lambda}$. Then
    \begin{equation}\label{eq:skew-hall-type-motivic}
        \set{\tl N: \ut M\subeq \tl N\subeq \tl L, \boldsymbol{type}(\tl N/\ut M)=\boldsymbol{\mu}}
    \end{equation}
    is a disjoint union of affine cells. In particular, its motive is $g^{\boldsymbol{\lambda}}_{\boldsymbol{\mu}}(\L)$.
\end{lemma}
\begin{proof}
    It suffices to prove the case $b=1$. By passing to isomorphic copies, we may assume $\ut M=\tl R^d$ and $\tl L=D_{-\lambda} \tl R^d$. Then using \Cref{subsec:schubert}, let $\nu$ be an arbitrary signature and $\tl N\in X_{-\nu}^\circ$, then $\tl N\supeq \ut M$ if and only if $\nu\geq 0$. Since $\tl N=x_{\tl N}\tl R^d$ is the column span of the normal form of $\tl N$, the condition $\tl N\subeq \tl L$ is equivalent to that the $i$-th row of $x_{\tl N}$ is divisible by $T^{-\lambda_i}$ for all $i$. This condition cuts out an affine subspace of $X_{-\nu}^\circ$. Finally, the condition $\mathrm{type}(\tl N/\ut M)=\mu$ is equivalent to $\nu_\geq =\mu$. Putting these together, the moduli space \eqref{eq:skew-hall-type-motivic} is a disjoint union of affine subspaces of $X_{-\nu}^\circ$, where $\nu$ ranges over all permutations of $\mu$, so we are done.
\end{proof}

\subsection{Framework for explicit computations}
Given a local order $R$, which lattice zeta functions are the most promising sources of new $q$-series? We propose two natural candidates: those associated with the lattices $R^d$ and $\tl R^d$. This line of inquiry, however, depends on a positive answer to a fundamental question:

\begin{question}
If $R$ is a geometric local order and $E=R^d$ or $\tl R^d$, is the motive $[\Quot_n(E)]$ a polynomial in $\L$?
\end{question}

While answering this in full generality is out of reach, we believe that this polynomiality holds in any case where explicit computations are feasible. We therefore introduce the following notion to define the scope of our study:

\begin{definition}
A geometric local order $R$ is \textbf{good} if $[\Quot_n(R^d)]$ and $[\Quot_n(\tl R^d)]$ are polynomials in $\L$ for all $d, n \ge 0$. For a good order, we view the motivic zeta function $Z_E(t)$ as a series $Z_E(t;\L)$ in two variables, where $E=R^d$ or $\tl R^d$.
\end{definition}
Essentially, a local order is ``good'' if the associated counting problems over $\Fq$ yield answers that are polynomial in $q$.

For good orders, the reflection principle holds motivically. By combining the point-count result of \Cref{thm:reflection-local} with a standard spreadout argument (in the sense of \cite[Appendix]{hausel2008mixed}), we obtain:
\begin{proposition}\label{thm:reflection-motivic-good}
    Let $R/k$ be a good geometric local order that is Gorenstein. Then the motivic reflection principle \eqref{eq:motivic-reflection} holds for $E=R^d$. \hfill \qedsymbol
\end{proposition}

Our goal is to understand the nature of these two-variable series arising from good local orders. To that end, we now collect some general properties that not only aid our explicit computations, but also provide consistency checks for future study and serve as a rich source of nontrivial $q$-series identities. A key theme throughout is that $\tl R^d$ is a convenient object for direct computation, while $R^d$ is often more structurally significant, for instance due to the reflection principle.

To better compare our results with classical $q$-series, we adopt a particular rescaling of $Z_{R^d}(t)$. Recall the motivic Coh zeta function for a geometric local order $R/k$,
\begin{equation}\label{eq:motivic-coh}
    \Zhat_R(t) := \sum_{n \ge 0} [\Coh_n(R)]\, t^n \in \KStck{k}[[t]],
\end{equation}
and its arithmetic counterpart for an arithmetic local order $R$,
\begin{equation}\label{eq:arithmetic-coh}
    \zetahat_R(s) := \sum_{Q\in \textbf{FinMod}_R/{\sim}} \frac{1}{\abs{\Aut_R(Q)}}\abs{Q}^{-s}.
\end{equation}
In light of \Cref{thm:A}, we may view these as limits of their ``finitized'' counterparts, defined below.

\begin{definition}\label{def:finitized}
The \textbf{finitized motivic Coh zeta function} is
\begin{equation}
    \Zhat_{R,d}(t) := Z^R_{R^d}(\L^{-d}t),
\end{equation}
and its arithmetic counterpart is
\begin{equation}
    \zetahat_{R,d}(s) := \zeta^R_{R^d}(s+d).
\end{equation}
The \textbf{normalized} versions are $\Nhat_{R,d}(t) := \Zhat_{R,d}(t) / \Zhat_{\tl R,d}(t)$ and $\nuhat_{R,d}(s) := \zetahat_{R,d}(s) / \zetahat_{\tl R,d}(s)$.
\end{definition}

Though the finitized Coh zeta function is just a rescaling of the lattice zeta function for $R^d$, we expect it to be more combinatorially tame, given that its $d\to\infty$ limit is guaranteed to exist. As an initial observation, we can analyze its special values.

\begin{proposition}\label{prop:s-zero-specialization}
    Let $R$ be a Gorenstein arithmetic local order and let $\Delta = \abs{\tl R/R}$. The special value of $\nuhat_{R,d}(s)$ at $s=0$ can be computed from the data of $\tl R^d$ alone:
    \begin{equation}
        \nuhat_{R,d}(0) = \Delta^{-d^2} \nu_{\tl R^d}^R(0).
    \end{equation}
    Similarly, if $R/k$ is a good Gorenstein geometric local order and $\delta = [\tl R:R]$, then
    \begin{equation}
        \Nhat_{R,d}(1) = \L^{-d^2 \delta} N_{\tl R^d}^R(1).
    \end{equation}
\end{proposition}
\begin{proof}
    The arithmetic case follows by combining \Cref{thm:reflection-local} and \Cref{cor:central-arithmetic}. The geometric case follows by combining \Cref{thm:reflection-motivic-good} and \Cref{cor:central-motivic}.
\end{proof}

Another general property arises from considering the Euler characteristic. For an algebraically closed field $k$, let $\chi:\KVar{k}\to \Z$ (or $\Q
_\ell$) be the Euler characteristic ring homomorphism arising from Betti cohomology or $\ell$-adic cohomology. Note that $\chi(\L)=1$.
\begin{proposition}\label{prop:chi-power}
    Let $k$ be an algebraically closed field, $R$ be a finitely generated $k$-algebra or a completion thereof, and $E$ be a finitely generated module over $R$. Then for any $d \ge 0$, we have
    \begin{equation}
        \chi(Z_{E^d}(t)) = \chi(Z_E(t))^d.
    \end{equation}
    In particular, if $R$ is a good geometric order, then for $E=R$ or $E=\tl R$, we have $Z_{E^d}(t;1) = Z_E(t;1)^d$.
\end{proposition}

Numerical data from the examples in \Cref{sec:torus} suggests a similar identity when specializing $\L$ to a root of unity. We conjecture that for $r \mid d$ and $\zeta_r$ a primitive $r$-th root of unity,
\begin{equation}
    Z_{E^d}(t;\zeta_r) \overset{?}{=} Z_E(t^r;1)^{d/r}.
\end{equation}
We expect this ``cyclic-sieving'' phenomenon might have a geometric interpretation using a twisted Euler characteristic.

\begin{proof}[Proof of \Cref{prop:chi-power}]
    We use the theorem of Bia\l ynicki-Birula \cite{bb1973fixed}, which states that for a variety with a torus action, its Euler characteristic equals that of its fixed-point locus.

    Consider the natural scaling action of the torus $T = (k^\times)^d$ on $E^{\oplus d}$. For each $n\geq 0$, this action induces a $T$-action on each Quot scheme $\Quot_{n}(E^d)$. A submodule $F \subeq E^d$ corresponds to a $T$-fixed point if and only if it is a direct sum of submodules in each component, i.e., $F = F_1 \oplus \dots \oplus F_d$ where each $F_i \subeq E$. This can be shown with a standard Vandermonde matrix argument for a generic one-parameter subgroup of $T$.

    This characterization implies that the fixed-point locus has a locally closed decomposition:
    \begin{equation*}
        (\Quot_n(E))^T = \bigsqcup_{n_1+\dots+n_d=n} \prod_{i=1}^d \Quot_{n_i}(E).
    \end{equation*}
    Applying the theorem of Bia\l ynicki-Birula, we get
    \begin{equation*}
        \chi(\Quot_{n}(E^d)) = \chi((\Quot_{n}(E^d))^T) = \sum_{n_1+\dots+n_d=n} \prod_{i=1}^d \chi(\Quot_{n_i}(E)),
    \end{equation*}
    so the $n$-th coefficients of the claimed equality match.
\end{proof}

\section{Explicit formulas for the \texorpdfstring{$y^2=x^n$}{y2=xn} singularity}\label{sec:torus}

In this section, we apply the framework from the previous section to compute the normalized lattice zeta functions for the local rings corresponding to the $(2,n)$ torus knot and link singularities. 

\subsection{Structure of the rings \texorpdfstring{$R_{2,n}$}{R2n}}
We consider two families of planar curve germs. For $m \ge 1$, we define the \textbf{cusp} singularity or the $(2,2m+1)$-\textbf{torus knot} singularity by
\begin{equation}
    R_{2,2m+1} := k[[X,Y]]/(Y^2-X^{2m+1}),
\end{equation}
and the \textbf{nodal} singularity or the $(2,2m)$-\textbf{torus link} singularity by
\begin{equation}
    R_{2,2m} := k[[X,Y]]/(Y(Y-X^m)).
\end{equation}
When the characteristic of $k$ is not $2$, $R_{2,2m}$ is isomorphic to the ring $k[[X,Y]]/(Y^2-X^{2m})$.

The key properties of these rings, including their normalizations and conductors, are summarized below.

\begin{itemize}
    \item The cusp ($R_{2,2m+1}$): This ring is an integral domain. Its normalization is the DVR $\tl R = k[[T]]$, and the inclusion can be realized by the parametrization $X=T^2, Y=T^{2m+1}$. The conductor is $\fc = (T^{2m})\tl R$, and the quotient by the conductor is a DVR quotient $R/\fc \simeq k[[X]]/(X^{m})$.

    \item The node ($R_{2,2m}$): This ring has two branches. Its normalization is $\tl R = k[[T_1]] \times k[[T_2]]$. The inclusion can be identified with the subring $k[[T_1^m, T_2^m, T_1+T_2]] \subeq \tl R$, with parametrization $X=T_1+T_2$ and $Y=T_1^m$. The conductor is $\fc=(T_1^m, T_2^m)\tl R$, and the quotient is a DVR quotient $R/\fc \simeq k[[X]]/(X^m)$.
\end{itemize}

A crucial feature for our computations is that in both cases, the quotient ring $R/\fc$ is a DVR quotient, allowing us to use the combinatorics of Hall polynomials. Throughout, denote $A=R/\fc=k[[X]]/(X^m)$ and $\tl A = \tl R/\fc$.

\subsection{Parametrizing extension fibers}
Recall that the extension fiber $E_R(\tl L):=\set{L\subseteq_{R}\tl L: \tl RL=\tl L}$ is a main player in the formulas in \Cref{sec:rationality} and \Cref{sec:motivic}. In this subsection, we show how to explicitly classify the lattices in $E_R(\tl L)$ by reducing the problem modulo the conductor $\fc$.

Let $\tl L$ be an $\tl R$-lattice. The classification of $L \in E_R(\tl L)$ depends on the structure of the $\tl A$-module $V:=\tl L/\fc \tl L$, where $A=R/\fc$ and $\tl A = \tl R/\fc$. The key observation is that $V$ decomposes into a direct sum of two free $A$-modules of rank $d$.

\begin{notation}
Let $V:=\tl L/\fc \tl L$. We define a decomposition $V=V_1 \oplus V_2$ as follows:
\begin{itemize}
    \item For the cusp $R_{2,2m+1}$: Let $F_{\tl L}$ be a free $R$-sublattice of $\tl L$ such that $\tl R F_{\tl L} = \tl L$. Define the $A$-submodules $V_2 := F_{\tl L}/\fc\tl L$ and $V_1 := T V_2$. Then $V_1$ and $V_2$ are free $A$-modules of rank $d$, and $V=V_1 \oplus V_2$.
    \item For the node $R_{2,2m}$: Define the $A$-submodules $V_i := e_i V$ for $i=1,2$, where we recall $e_i$ is the idempotent of the $i$-th branch of $\tl R$. Then $V_1$ and $V_2$ are free $A$-modules of rank $d$, and $V=V_1 \oplus V_2$.
\end{itemize}
Any $L \in E_R(\tl L)$ corresponds to an $A$-submodule $W_L := L/\fc \tl L \subeq V$, and we have $[\tl L:L]=[V:W_L]$.
\end{notation}

The crucial step is to translate the condition $\tl R L = \tl L$ into a simpler algebraic condition on $W_L$ inside $V$.

\begin{lemma}\label{lem:extension-condition}
An $R$-lattice $L$ with $\fc\tl L \subeq L \subeq \tl L$ is in the extension fiber $E_R(\tl L)$ if and only if its corresponding $A$-submodule $W_L$ satisfies:
\begin{enumerate}
    \item $W_L + V_1 = V$ if $R=R_{2,2m+1}$.
    \item $W_L + V_1 = V$ and $W_L + V_2 = V$ if $R=R_{2,2m}$.
\end{enumerate}
\end{lemma}
\begin{proof}~
\begin{enumerate}
    \item For the cusp case, denote $F=F_{\tl L}$ for brevity. From the ring structure, we have $\tl R=R+TR$, which implies the identities $\tl RL=L+TL$ and $\tl L=\tl RF=F+TF$. The condition $L \in E_R(\tl L)$ is therefore $\tl RL = \tl L$, or equivalently, $L+TL = \tl L$. We claim this is equivalent to the condition $L+TF=\tl L$. Our proof will exploit the topological nilpotence of $T$.

($\Leftarrow$) Assume $L+TF=\tl L$. To prove $L+TL=\tl L$, it suffices to show that $\tl L=L+TL+T^n F$ for all $n\geq 1$, since for $n\gg 0$, $T^n F \subeq L$. We proceed by induction on $n$. The base case $n=1$ follows from the assumption: $\tl L = L+TF \subeq L+T\tl L = L+T(L+TF) = L+TL+T^2F$. For the induction step, assume $\tl L=L+TL+T^n F$. Then we have
    \begin{equation}
        \tl L = L+TL+T^n F \subeq L+TL+T^n \tl L = L+TL+T^n(L+TF) = L+TL+T^{n+1}F,
    \end{equation}
where the last equality uses the fact that $T^n L\subeq L$ if $n$ is even and $T^n L\subeq TL$ if $n$ is odd, since $T^2\in R$.

($\Rightarrow$) Assume $L+TL=\tl L$. We have the inclusion
    \begin{equation}
        \tl L = L+TL \subeq L+T\tl L = L+T(F+TF) = L+TF+T^2 F.
    \end{equation}
To prove $\tl L=L+TF$, it suffices to show that $\tl L=L+TF+T^{2^n}F$ for all $n\geq 1$, again because $T^{2^n}F\subeq L$ for $n\gg 0$. We have just proved the base case $n=1$. For the inductive step, assume $\tl L=L+TF+T^{2^n}F$. Then
    \begin{align}
    \tl L &= L+TF+T^{2^n}F \subeq L+TF+T^{2^n}\tl L \\
    &= L+TF+T^{2^n}(L+TF+T^{2^n}F) \\
    &= (L+T^{2^n}L)+(TF+T^{2^n+1}F)+T^{2^{n+1}}F\\
    &=L+TF+T^{2^{n+1}}F. 
    \end{align}
    
Having established the equivalence between $L\in E_R(\tl L)$ and $L+TF=\tl L$, reducing the latter modulo $\fc\tl L$ gives the desired condition $W_L + V_1 = V$.

\item  For the nodal case, $\tl R L = e_1 L\oplus e_2 L$ and $\tl L=e_1\tl L\oplus e_2\tl L$, so $\tl RL=\tl L$ if and only if $e_i L = e_i \tl L$ for $i=1,2$. Since multiplication by $e_i$ is a projection, this is equivalent to $L + e_{3-i} \tl L = \tl L$. Reducing these two conditions modulo $\fc \tl L$ gives what we want.
\end{enumerate}
\end{proof}

With this equivalence, we can classify the submodules $W_L$ using the decomposition $V=V_1 \oplus V_2$. This leads to the following parametrizations for the extension fiber.

\begin{lemma}\label{lem:cusp_extension_fiber}
    Let $R=R_{2,2m+1}$. The map $L \mapsto (W_L \cap V_1, \varphi_L)$, where $\varphi_L$ is the composition $V_2 \to V \to V/W_L \simeq V_1/(W_L \cap V_1)$, defines a bijection
    \begin{equation}
        E_R(\tl L) \to \set{(W',\varphi): W'\subeq_A V_1, \varphi\in \Hom_A(V_2,V_1/W')}.
    \end{equation}
    Moreover, under this bijection, we have an isomorphism of $A$-modules $\tl L/L \simeq V_1/W'$.
\end{lemma}
\begin{proof}
    By \Cref{lem:extension-condition}, $L \in E_R(\tl L)$ if and only if its corresponding $A$-submodule $W_L$ satisfies $W_L + V_1 = V$. The result then follows directly from applying \Cref{lem:extension_problem}.
\end{proof}

\begin{lemma}\label{lem:node_extension_fiber}
    Let $R=R_{2,2m}$. The map $L \mapsto (W_L \cap V_1, \varphi_L)$ defines a bijection
    \begin{equation}
        E_R(\tl L) \to \set{(W',\varphi): W'\subeq_A V_1, \varphi\in \Surj_A(V_2,V_1/W')},
    \end{equation}
    where $\Surj_A$ denotes the set of surjective $A$-module homomorphisms. Moreover, we have $\tl L/L \simeq_A V_1/W'$.
\end{lemma}
\begin{proof}
    By \Cref{lem:extension-condition}, an $R$-lattice $L$ is in the extension fiber if and only if its corresponding $A$-submodule $W_L$ satisfies both $W_L+V_1=V$ and $W_L+V_2=V$. We first classify submodules satisfying the first condition using \Cref{lem:extension_problem}, which gives the pair $(W', \varphi_L)$. The second condition, $W_L+V_2=V$, is then equivalent to the map $\varphi_L: V_2 \to V_1/W'_L$ being surjective.
\end{proof}

\subsection{Zeta functions for $\tl R$-lattices}
We can now apply the parametrizations from the previous subsection to compute the zeta functions for $\tl R$-lattices, which are the ``easier" cases. The formulas are expressed as sums over partitions involving Hall polynomials. For clarity, we handle the arithmetic case first. Throughout this subsection, we let $k=\Fq$ and $t:=q^{-s}$.

\begin{proposition}\label{prop:cusp-tl}
    Let $R=R_{2,2m+1}$. Then
    \begin{equation}
        \nu_{\tl R^d}^R(s) = \sum_{\mu\subeq (m^d)} g^{(m^d)}_\mu(q)(q^d t)^{\abs{\mu}}.
    \end{equation}
\end{proposition}
\begin{proof}
    Write $\tl L=\tl R^d$ for brevity. By \Cref{thm:rationality-tl-arithmetic}, $\nu_{\tl L}^R(s) = \sum_{L \in E_R(\tl L)} (\tl L:L)^{-s}$. We use the classification from \Cref{lem:cusp_extension_fiber}, after fixing an arbitrary choice of $F_{\tl L}$. A pair $(W',\varphi)$ corresponds to a lattice $L$ with colength $[\tl L:L]=[V_1:W']$. If the cotype of the $A$-module $W'$ in $V_1$ is $\mu$, then $[V_1:W']=\abs{\mu}$.

    Recall $V_1\simeq A^d$ and $A$ is a length $m$ DVR quotient, so the type of $V_1$ is $(m^d)$. The number of choices for $W'$ with cotype $\mu$ is given by the skew Hall polynomial $g^{(m^d)}_\mu(q)$. For a fixed $W'$, the number of choices for $\varphi \in \Hom_A(V_2, V_1/W')$ is $\abs{V_1/W'}^d = q^{d\abs{\mu}}$. Summing over all possibilities gives
    \begin{equation}
        \zeta_{\tl R^d}^R(s) = \sum_{\mu\subeq (m^d)} g^{(m^d)}_\mu(q) q^{d\abs{\mu}} (q^{\abs{\mu}})^{-s} = \sum_{\mu\subeq (m^d)} g^{(m^d)}_\mu(q) q^{d\abs{\mu}} t^{\abs{\mu}},
    \end{equation}
    as desired.
\end{proof}

\begin{proposition}\label{prop:node-tl}
    Let $R=R_{2,2m}$. Then
    \begin{equation}
        \nu_{\tl R^d}^R(s) = \sum_{\mu\subeq (m^d)} g^{(m^d)}_\mu(q) \frac{(q^{-1};q^{-1})_d}{(q^{-1};q^{-1})_{d-\mu_1'}} (q^d t)^{\abs{\mu}}.
    \end{equation}
\end{proposition}
\begin{proof}
    The argument is analogous to the cusp case, but uses the classification from \Cref{lem:node_extension_fiber}. For a fixed submodule $W'$ with cotype $\mu$, the number of choices for $\varphi$ is $\abs{\Surj_A(V_2, V_1/W')}$. By \Cref{lem:sur_probability}, this number is $q^{d\abs{\mu}}\frac{(q^{-1};q^{-1})_d}{(q^{-1};q^{-1})_{d-\mu_1'}}$. Summing over all possible cotypes $\mu$ gives the desired result.
\end{proof}

\subsection{Zeta functions for $R^d$}\label{subsec:boundary}
We now tackle the more difficult computation for the lattice $R^d$. Throughout this subsection, let $M=R^d$ and choose $\ut M=\fc R^d$. The strategy is to re-stratify the boundary locus $B_R(M;\ut M)=\set{(\tl L,L):\tl L\supeq \ut M, L\in E_R(\tl L;M)}$ defined in \Cref{def:boundary-locus}, in terms of an auxiliary \textbf{boundary $R$-lattice} $L_b:=\tl L\cap M$. Note that $L_b$ recovers $\tl L$ by $\tl L=\tl R L_b$, since $\tl RL_b\supeq \tl RL=\tl L$. Not every $L_b$ can arise this way: it is necessary that $\tl RL_b\cap M=L_b$. Having chosen $L_b$, the condition $L\in E_R(\tl L; M)$ is equivalent to $L\in E_R(\tl L;L_b)$. The set $B_R(M; \ut M)$ can thus be described as the disjoint union
\begin{equation}
    B_R(M;\ut M)=\bigsqcup_{L_b\in \partial_R(M;\ut M)} E_R(\tl RL_b;L_b),
\end{equation}
where
\begin{equation}
    \partial_R(M;\ut M)=\set{L_b:\ut M\subeq L_b\subeq_R M,\; \tl RL_b\cap M=L_b}.
\end{equation}
The following lemma describes the general structure of the inner set.

\begin{lemma}[Structure of the restricted extension fiber]\label{lem:restricted-extension-structure}
    Let $L_b \in E_R(\tl L)$ be classified by the datum $(W'_{L_b}, \varphi_{L_b})$ according to \Cref{lem:cusp_extension_fiber} or \Cref{lem:node_extension_fiber}. A lattice $L \in E_R(\tl L)$ is a sublattice of $L_b$ if and only if its classifying datum $(W'_L, \varphi_L)$ satisfies:
    \begin{enumerate}
        \item $W'_L$ is a submodule of $W'_{L_b}$.
        \item $\varphi_L$ is a lift of $\varphi_{L_b}$ (i.e., $\varphi_{L_b}$ is the composition $V_2 \xrightarrow{\varphi_L} V_1/W'_L \to V_1/W'_{L_b}$).
    \end{enumerate}
\end{lemma}
\begin{proof}
    A lattice $L \in E_R(\tl L)$ is a sublattice of $L_b$ if and only if its corresponding $A$-submodule $W_L$ is a submodule of $W_{L_b}$. The result follows from the definitions of the bijections in \Cref{lem:extension_problem}.
\end{proof}

This structural lemma allows us to compute the relevant generating function for $E_R(\tl RL_b;L_b)$, which we find depends only on the type of the $A$-module $\tl R L_b / L_b$. 

\begin{lemma}[Generating function for the restricted extension fiber]\label{lem:general-purpose-counting}
    Assume $k=\Fq$. Let $L_b$ be an $R$-lattice, and let $(m^d)-\lambda$ be the type of the $A$-module $\tl R L_b/L_b$. The generating function
    \begin{equation}
        G_{L_b}(t) := \sum_{L \in E_R(\tl R L_b; L_b)} t^{[L_b:L]}
    \end{equation}
    is given by:
    \begin{enumerate}
        \item If $R=R_{2,2m+1}$, then
        \begin{align}
            G_{L_b}(t)& = \sum_{\mu} g^{\lambda}_\mu(q) (q^{d}t)^{\abs{\lambda}-\abs{\mu}} \\&= \sum_{\mu} g^{\lambda}_\mu(q) (q^{d}t)^{\abs{\mu}}.\label{eq:cusp-restricted-formula}
        \end{align}
        \item If $R=R_{2,2m}$, then
        \begin{equation}\label{eq:node-restricted-formula}
            G_{L_b}(t) = \sum_{\mu} g^{\lambda}_\mu(q)  \frac{(q^{-1};q^{-1})_{\lambda_m'}}{(q^{-1};q^{-1})_{\mu_m'}} (q^{d}t)^{\abs{\lambda}-\abs{\mu}}.
        \end{equation}
    \end{enumerate}
\end{lemma}
\begin{proof}
    Let $\tl L=\tl RL_b$ and let $(W_{L_b}',\varphi_{L_b})$ be the classification datum for $L_b\in E_R(\tl L)$. From \Cref{lem:cusp_extension_fiber} and \Cref{lem:node_extension_fiber}, the cotype of $W_{L_b}'$ in $V_1$ is $(m^d)-\lambda$, so the type of $W_{L_b}'$ is $\lambda$. We now apply \Cref{lem:restricted-extension-structure} and sum over all data $(W', \varphi)$ classifying $L\in E_R(\tl L;L_b)$. We classify the submodules $W'$ by their type $\mu$ in $W'_{L_b}$; since the type of $W'_{L_b}$ is $\lambda$, there are $g^\lambda_\mu(q)$ choices for $W'$. 
    
    For each such $W'$, we count the number of valid lifts $\varphi$.
    For the cusp case, this is $\abs{\Hom_A(V_2, W'_{L_b}/W')} = \abs{W'_{L_b}/W'}^d = q^{d(\abs{\lambda}-\abs{\mu})}$.
    For the nodal case, however, we must count surjective lifts. By \Cref{lem:sur_homogeneity}, every surjection $\varphi_{L_b}$ has the same number of surjective lifts, so the number of surjective lifts for each given $\varphi_{L_b}$ is
    \begin{equation}
        \frac{\abs{\Surj_A(V_2,V_1/W')}}{\abs{\Surj_A(V_2,V_1/W'_{L_b})}} = \frac{q^{d(md-\abs{\mu})}(q^{-1};q^{-1})_d/(q^{-1};q^{-1})_{\mu_m'}}{q^{d(md-\abs{\lambda})}(q^{-1};q^{-1})_d/(q^{-1};q^{-1})_{\lambda_m'}} = q^{d(\abs{\lambda}-\abs{\mu})} \frac{(q^{-1};q^{-1})_{\lambda_m'}}{(q^{-1};q^{-1})_{\mu_m'}}
    \end{equation}
    by \Cref{lem:sur_probability}.

    Finally, since $[L_b:L]=[W'_{L_b}:W']=\abs{\lambda}-\abs{\mu}$, summing over all $\mu \subeq \lambda$ gives the formulas. To obtain the second formula in the cusp case, we repeat the same argument, but with $\mu$ being the \emph{cotype} of $W'$ in $W_{L_b}'$ instead.
\end{proof}

The following lemma provides the final ingredients for our main computation. It shows every lattice between $\ut M$ and $M$ is a boundary lattice, and how the combinatorial data of a boundary lattice $L_b$ (namely, its cotype $\lambda$ inside $M$) determines the structure of $\tl RL_b/\ut M$ and $\tl RL_b/L_b$, which are needed to compute the lattice zeta function for $M$.

\begin{lemma}\label{lem:boundary-lattice-invariants}
    Let $L_b$ be any $R$-lattice such that $\ut M \subeq L_b \subeq M$, and let $\lambda$ be the type of the $A$-module $M/L_b$. Then
    \begin{enumerate}
        \item We have $\tl RL_b\cap M=L_b$, so that $L_b\in \partial_R(M;\ut M)$.
        \item The type of the $A$-module $\tl R L_b / L_b$ is $(m^d)-\lambda$.
        \item The type vector of the $\tl A$-module $\tl R L_b / \ut M$ is given by
        \begin{itemize}
            \item a single partition $2((m^d)-\lambda)$, if $R=R_{2,2m+1}$, where for any partition $\mu=(\mu_1,\dots)$, $2\mu$ denotes the partition $(2\mu_1,\dots)$. 
            \item the pair of partitions $((m^d)-\lambda,(m^d)-\lambda)$, if $R=R_{2,2m}$.
        \end{itemize}
        In particular, the rank vector of $\tl R L_b/\ut M$ over $\tl A$ is $(d-\lambda_m')\one$, where $\one=1$ if $R=R_{2,2m+1}$ and $\one=(1,1)$ if $R=R_{2,2m}$. 
    \end{enumerate}
\end{lemma}
\begin{proof}
    Let $\tl L=\tl RL_b$. Consider the tower of lattices
    \[
    \begin{tikzcd}[row sep=tiny, column sep=large]
    & \tl R M \ar[dl, no head] \ar[dr, no head] & \\
    \tl L \ar[dr, no head] & & M \ar[dl, no head] \\
    & L_b \ar[d, no head] & \\
    & \ut M &
    \end{tikzcd}
    \]
    Define $W_b:=L_b/\ut M$. Modulo $\ut M$, the tower becomes
    \[
    \begin{tikzcd}[row sep=tiny, column sep=large]
        & \tl A^d \ar[dl, no head] \ar[dr, no head] & \\
        \tl A W_b \ar[dr, no head] & & A^d \ar[dl, no head] \\
        & W_b \ar[d, no head] & \\
        & 0 &
    \end{tikzcd}
    \]
    The $A$-cotype of $W_b$ in $A^d$ is $\lambda$, so the $A$-type of $W_b$ is $(m^d)-\lambda$.

    For the sake of uniformity, define an element $J\in \tl A$ by $J=T$ in the cusp case and $J=e_1$ in the node case. A crucial structural observation is that $\tl A=A\oplus JA$; moreover, the multiplication map $A\to \tl A$ by $J$ is injective. Since $W_b\subeq A^d$, this means $\tl AW_b=W_b\oplus JW_b$ (with $J$ acting injectively), so $\tl AW_b\cap A^d=W_b$ and $\tl AW_b/W_b\simeq JW_b\simeq W_b$ as $A$-modules. This proves parts (a)(b) the lemma.
    
    For (c), the assertion $\tl AW_b=W_b\oplus JW_b$ with $J$ acting injectively implies that $\tl AW_b\simeq_{\tl A} W_b\otimes_A {\tl A}$, namely, the structure of $\tl AW_b$ is determined intrinsically by the structure of $W_b$. Since for any $1\leq \ell\leq m$,
    \begin{equation}
        \frac{k[[X]]}{(X^\ell)} \otimes_A \tl A \simeq \begin{cases}
            \frac{k[[T]]}{(T^{2\ell})} & \text{if } R=R_{2,2m+1}, \\
            \frac{k[[T_1]]}{(T_1^\ell)} \times \frac{k[[T_2]]}{(T_2^\ell)} & \text{if } R=R_{2,2m},
        \end{cases}
    \end{equation}
    we conclude the final part of the lemma.
\end{proof}

We are now ready to compute the zeta function for $R^d$.

\begin{proposition}\label{prop:cusp-full}
    Let $k=\Fq$ and $R=R_{2,2m+1}$. Then
    \begin{equation}\label{eq:cusp-full-initial}
        \nu_{R^d}^R(s) = \sum_{\lambda,\mu} g^{(m^d)}_{\lambda}(q)\, g^{\lambda}_{\mu}(q)\,(t;q)_{d-\lambda'_m} t^{\abs{\lambda}}(q^{d} t)^{\abs{\mu}}.
    \end{equation}
\end{proposition}
\begin{proof}
    We apply \Cref{thm:rationality-arithmetic}, which can be restated as
    \begin{equation}
        \nu_{R^d}^R(s) = \sum_{L_b: \ut M\subeq L_b\subeq_R M} (t;q)_{\rk(\tl R L_b/\ut M)} \sum_{L \in E_R(\tl R L_b; L_b)} t^{[M:L]}=\sum_{L_b: \ut M\subeq L_b\subeq_R M} (t;q)_{\rk(\tl R L_b/\ut M)} t^{[M:L_b]} G_{L_b}(t).
    \end{equation}
    Let $\lambda$ be the cotype of $L_b$ in $M$. By \Cref{lem:boundary-lattice-invariants}, the rank of $\tl R L_b/\ut M$ is $d-\lambda'_m$, and the type of $\tl RL_b/L_b$ is $(m^d)-\lambda$, so $G_{L_b}(t)$ is directly given by \eqref{eq:cusp-restricted-formula} in \Cref{lem:general-purpose-counting}. Putting everything together yields the desired formula.
\end{proof}

\begin{theorem}\label{thm:node-full}
    Let $k=\Fq$ and $R=R_{2,2m}$. Then
    \begin{equation}\label{eq:node-full}
        \nu_{R^d}^R(s) = \sum_{\lambda,\mu} g^{(m^d)}_\lambda(q) \, g^{\lambda}_\mu(q)  \frac{(q^{-1};q^{-1})_{\lambda_m'}}{(q^{-1};q^{-1})_{\mu_m'}} (t;q)_{d-\lambda_m'}^2 t^{\abs{\lambda}} (q^{d}t)^{\abs{\lambda}-\abs{\mu}}.
    \end{equation}
\end{theorem}
\begin{proof}
    The argument is analogous to the cusp case. \Cref{thm:rationality-arithmetic} is now restated as
    \begin{equation}
        \nu_{R^d}^R(s) = \sum_{L_b: \ut M\subeq L_b\subeq_R M} (t;q)_{\rk_1(\tl R L_b/\ut M)} (t;q)_{\rk_2(\tl R L_b/\ut M)}  t^{[M:L_b]} G_{L_b}(t).
    \end{equation}
    Let $\lambda$ be the cotype of $L_b$ in $M$. By \Cref{lem:boundary-lattice-invariants}, the rank vector of $\tl R L_b/\ut M$ is $(d-\lambda'_m)\one$, and the type of $\tl RL_b/L_b$ is $(m^d)-\lambda$, so $G_{L_b}(t)$ is directly given by \eqref{eq:node-restricted-formula} in \Cref{lem:general-purpose-counting}. Putting everything together yields the desired formula.
\end{proof}

\begin{remark}\label{rmk:motivic-torus}
    If $k$ is an algebraically closed field, an approach similar to that in \Cref{sec:motivic} should establish that $[\Quot_n(E)]$ is a polynomial in $\L$ for $E = \tl R^d$ or $R^d$. Moreover, when $R = R_{2,2m+1}$, the scheme $\Quot_n(E)$ admits an affine paving. This would allow all our counting formulas to upgrade naturally to motivic formulas. While we do not provide a full proof here, we highlight the key observations for the more subtle case $E = R^d$:
    \begin{itemize}
        \item All submodule selection problems are covered by \Cref{lem:hall-in-box-motivic} and \Cref{lem:skew-hall-type-motivic}: 
        \begin{enumerate}
            \item Choosing $L_b$ amounts to selecting a submodule $L_b/\ut M$ of prescribed cotype in the module $A^d$ of rectangular type; \Cref{lem:hall-in-box-motivic} shows that the space of such $L_b$ is a union of affine spaces.
            \item The proof of \Cref{lem:general-purpose-counting} involves choosing a submodule $W'$ of prescribed type in $W_{L_b}'$ (not cotype), and \Cref{lem:skew-hall-type-motivic} provides the required affine paving for the space of $W'$.
        \end{enumerate}
        \item The Iwahori elements yield explicit isomorphisms between a ``moving'' $\tl R$-lattice (such as $\tl RL_b$) and a fixed $\tl R$-lattice (such as $\tl R^d$), which are useful for constructing Zariski-local or global trivializations of certain fiber bundles. This is important: often, we have a surjection where the motives of the base and of each fiber are known, but to conclude that the motive of the total space is the product of the motives of the fiber and of the base, such trivializations are needed.
        \item This explicit understanding applies to \Cref{lem:key-decomposition}, which is well-suited for establishing affine pavings.
        \item In the nodal case, the existence of an affine paving remains open, since part of the parametrization involves the space of surjections, which is only quasi-affine.
    \end{itemize}
\end{remark}

\section{Combinatorial identities}\label{sec:combo}
In this section, we prove two $q$-series identities that give simplification or special value for $\nu^R_{R^d}(s)$, with $R=R_{2,2m+1}$ or $R=R_{2,2m}$. 

\subsection{Simplification in the cusp case}
Let $k=\Fq$ and $R=R_{2,2m+1}$. The goal of this subsection is to prove a surprising simplification of $\nu^R_{R^d}(s)$:
\begin{theorem}\label{thm:cusp_simplification}
    Let $k=\Fq$ and $R=R_{2,2m+1}$. Then
    \begin{equation}
        \nu^R_{R^d}(s) = \sum_{\mu\subeq (m^d)} g^{(m^d)}_\mu(q)(q^d t^2)^{\abs{\mu}}.
    \end{equation}
\end{theorem}
In particular, $\nu^R_{R^d}(s)=\nu^R_{\tl R^d}(2s)=\nu^R_{\tl R^d}(s)|_{t\mapsto t^2}$. This coincidence remains a total mystery. 

To prove this theorem, we need an identity (\Cref{lem:cusp_squaring}) that can be viewed as a bounded analogue of the Cauchy identity for the skew Hall--Littlewood polynomials, in the sense of \Cref{rmk:skew-cauchy}. We thank S.~Ole Warnaar for sketching its proof and pointing out this connection. The proof uses a summation identity due to Cauchy:
\begin{lemma}[{Cauchy, \cite[(II.5)]{gasperrahman}}]
    As formal series in $a,z,q$, we have
    \begin{equation}
        \sum_{k} (-1)^k q^{\binom{k}{2}}\frac{(a;q)_k}{(q;q)_k (az;q)_k} z^k = \frac{(z;q)_\infty}{(az;q)_\infty}.
    \end{equation}
\label{lem:cauchy}
\end{lemma}
Recall that all sums are naturally truncated by the rule $1/(q;q)_k=0$ for $k\in \Z_{<0}$. 

\begin{lemma}\label{lem:cusp_squaring}
    For $d,m\geq 0$ and any partition $\mu$, we have the following identity in $\Z[t,q]$:
    \begin{equation}\label{eq:cusp_squaring}
        \sum_{\lambda} g^{(m^d)}_\lambda(q)\, g^\lambda_\mu(q)\, t^{\abs{\lambda}} (t;q)_{d-\lambda_m'} = g^{(m^d)}_\mu(q) \,t^{\abs{\mu}}.
    \end{equation}
\end{lemma}

\begin{remark}\label{rmk:skew-cauchy}
    The $m\to \infty$ limit of \eqref{eq:cusp_squaring} reduces to the following identity
    \begin{equation}
        (t;z)_d \sum_\lambda P_\lambda(1,z,\dots,z^{d-1};z)\,Q'_{\lambda/\nu}(t;z)=P_\mu(1,z,\dots,z^{d-1};z)
    \end{equation}
    where $z=q^{-1}$ and $P_\lambda, Q'_{\lambda/\mu}$ are the modified Hall--Littlewood polynomial and skew Hall--Littlewood polynomial, respectively; see \cite{warnaar2013remarks}. This follows from the skew Cauchy identity
    \begin{equation}
        \sum_{\lambda} P_\lambda(\underline{x};z) \, Q'_{\lambda/\mu}(\underline{y};z)=P_\mu(\underline{x};z) \prod_{i,j} \frac{1}{1-x_iy_j}
    \end{equation}
    by setting $\underline{x}=(1,z,\dots,z^{d-1})$ and $\underline{y}=t$. 
\end{remark}

\begin{proof}[Proof of \Cref{lem:cusp_squaring}]
    For brevity, denote $(q^{-1})_n:=(q^{-1};q^{-1})_n$. Applying \Cref{thm:hall_skew} and dividing both sides by $q^{-\sum_{i\geq 1} \mu_i'^2} (q^{-1})_d$, the original identity is equivalent to
    \begin{equation}\label{eq:identity_proof_1_1}
        \sum_{\lambda} q^{-\sum_{i\geq 1}\lambda_i'(\lambda_i'-\mu_i')}\frac{(t;q)_{d-\lambda_m'}(q^d t)^{\abs{\lambda}}}{\prod_{i\geq 0} (q^{-1})_{\lambda_i'-\lambda_{i+1}'}}\prod_{i\geq 0} {\lambda_i'-\mu_{i+1}'\brack \lambda_i'-\mu_i'}_{q^{-1}} =\frac{(q^d t)^{\abs{\mu}}}{\prod_{i\geq 0}(q^{-1})_{\mu_i'-\mu_{i+1}'}},
    \end{equation}
    where $\lambda_0'=\mu_0':=d$.

    The sketch of proof of \eqref{eq:identity_proof_1_1} is as follows. Denote the left-hand side of \eqref{eq:identity_proof_1_1} by $f_{\mu,m,d}(t,q)$. We first fix $\lambda_1',\dots,\lambda_{m-1}'$ and sum over all $\lambda_m'$. It turns out that the inner sum can be evaluated by \Cref{lem:cauchy}. With this simplification, the outer sum over all $\lambda_1',\dots,\lambda_{m-1}'$ turns out that this sum is a multiple of $f_{\bbar \mu,m-1,d}(t,q)$, where $\bbar\mu$ is the partition with $\bbar \mu'_i=\mu'_i$ for $1\leq i\leq m-1$ and $\bbar \mu'_m=0$. The identity \eqref{eq:identity_proof_1_1} then follows from induction on $m$. 

    Here are details. If $m=0$, the identity reduces to $1=1$. Now assume $m\geq 1$. By extracting the factors that do not depend on $\lambda_m'$, and substituting $k=\lambda_m'-\mu_m'$, the left-hand side of \eqref{eq:identity_proof_1_1} equals    
    \begin{equation}\label{eq:identity_proof_1_2}
    \begin{multlined}
        \frac{(q^d t)^{\mu'_m}}{(q^{-1})_{\mu_m'}}\sum_{\lambda_1', \dots\lambda_{m-1}'} q^{-\sum_{i=1}^{m-1}\lambda_i'(\lambda_i'-\mu_i')} \frac{(q^d t)^{\sum_{i=1}^{m-1} \lambda_i'}}{\prod_{i=0}^{m-2}(q^{-1})_{\lambda_i'-\lambda_{i+1}'}} \prod_{i=1}^{m-1} {\lambda_i'-\mu_{i+1}'\brack \lambda_i'-\mu_i'}_{q^{-1}}
        \\
        \sum_{k} q^{-k(k+\mu_m')} \frac{(t;q)_{d-\mu_m'-k}(q^d t)^k}{(q^{-1})_{\lambda_{m-1}'-\mu_m'-k}(q^{-1})_k}.
    \end{multlined}
    \end{equation}

    We reorganize the inner sum of \eqref{eq:identity_proof_1_2} into a $q^{-1}$-series as follows
    \begin{equation}\label{eq:identity_proof_1_3}
        \frac{(t;q)_{d-\mu_m'}}{(q^{-1})_{\lambda'_{m-1}-\mu_m'}}\sum_{k} (-1)^k q^{-\binom{k}{2}}\frac{(q^{\lambda_{m-1}'-\mu_m'};q^{-1})_k}{(q^{-1})_k (q^{d-1-\mu_m'}t;q^{-1})_k} (q^{d-1-\lambda_{m-1}'}t)^k.
    \end{equation}

    Applying Lemma \ref{lem:cauchy} with $a=q^{\lambda_{m-1}'-\mu_m'}$ and $z=q^{d-1-\mu_m'}t$, we get after simplication
    \begin{equation}
        f_{\mu,m,d}(t,q) = (q^d t)^{\mu_m'} q^{-\mu_m'^2} {\mu_{m-1}'\brack \mu_m'}_{q^{-1}} f_{\bbar\mu,m-1,d}(t,q).
    \end{equation}

    The desired identity then follows from induction on $m$.
\end{proof}

\begin{proof}[Proof of \Cref{thm:cusp_simplification}]
    The formula follows from evaluating the sum over $\lambda$ on the right-hand side of \eqref{eq:cusp-full-initial} using \Cref{lem:cusp_squaring}. 
\end{proof}

\subsection{Special value in the nodal case}
In this subsection, we establish the remarkably simple value of the normalized finitized Coh zeta function for the nodal singularity at $s=0$. Proving this result by directly evaluating the complex multi-sum for $\nu_{R^d}^R(s)$ would be equivalent to proving a new, difficult Rogers-Ramanujan type identity. Instead, we present a short proof that notably avoids any deep $q$-series transformations, relying instead on the general properties of our framework and a combinatorial counting argument.

\begin{theorem}\label{thm:node_special_value}
    Let $k=\Fq$ and $R=R_{2,2m}$. For any $d \ge 0$, we have
    \begin{equation}
        \nuhat_{R,d}(0) = 1.
    \end{equation}
\end{theorem}

Our proof bypasses the complex formula for $\nu_{R^d}^R(s)$ entirely. We instead apply \Cref{prop:s-zero-specialization}, which reduces the problem to computing $\nu^R_{\tl R^d}(0)$. The value of this simpler zeta function is given by the following combinatorial identity.

\begin{lemma}\label{lem:node-t=1}
    Let $R=R_{2,2m}$ and $k=\Fq$. Then
    \begin{equation}\label{eq:node-t=1}
        \nu^R_{\tl R^d}(0)=\sum_{\mu} g^{(m^d)}_\mu(q)\, q^{d\abs{\mu}} \frac{(q^{-1};q^{-1})_d}{(q^{-1};q^{-1})_{d-\mu_1'}}=q^{md^2}.
    \end{equation}
\end{lemma}
\begin{proof}
    Let $(V,\pi,\Fq)$ be a DVR and $B = V/\pi^m V$. The right-hand side, $q^{md^2}$, is the total number of $B$-linear homomorphisms $f:B^d\to B^d$. The left-hand side enumerates the same set of homomorphisms, but stratified by the type, $\mu$, of the image of $f$. The term $g^{(m^d)}_\mu(q)$ counts the number of possible images of type $\mu$, and the remaining factor counts the number of surjective maps from $B^d$ onto a fixed image of type $\mu$, by \Cref{lem:sur_probability}. Summing over all possible types $\mu$ recovers the total count.
\end{proof}

\begin{proof}[Proof of \Cref{thm:node_special_value}]
    Substituting the result from \Cref{lem:node-t=1} and the Serre invariant $\Delta=\abs{\tl R/R}=q^m$ into the formula from \Cref{prop:s-zero-specialization}, we immediately get
    \begin{equation}
        \nuhat_{R,d}(0) = \Delta^{-d^2} \nu_{\tl R^d}^R(0) = (q^m)^{-d^2} \cdot q^{md^2} = 1,
    \end{equation}
    as required.
\end{proof}

\appendix
\section{Affine Grassmanians for reduced curve germs}\label{sec:AGreduced}
Let $R$ be a reduced curve germ (see Definition~\ref{def:reducedcurvegerm}). The goal of this section is to set up a notion of ``$\GL_n$-affine Grassmannian for $R$'': they are ind-projective schemes that parametrize $R$-lattices (Proposition~\ref{prop:ind}), as well as studying three important ``constructible morphisms'' between them that encode certain lattice operations that arise from the six functor formalism (Theorem~\ref{thm:consmaps}). This generalize classical affine Grassmannians studied in  geometric representation theory (cf.~ \cite{Zhu_AffineGrass}), where $R=k[[T]]$. The language can provide a more systematic framework for the constructions in \Cref{sec:motivic}. In the earlier version of this paper \cite{huangjiang2023torsionfree}, these affine Grassmannians played a central role in establishing the motivic theorem~(\Cref{thm:rationality-motivic}) and its relative upgrade. The approach used there was considerably more involved than the explicit Schubert cell parametrization presented in the current version, and has since been superseded. Nevertheless, we record these constructions here for future reference. {We note that the framework here differs from the common approach relating moduli spaces of $R$-modules to affine Grassmannians via affine Springer fibers (cf.~\cite[\S 6.3]{gks2021link}), as we do not rely on the choice of a DVR $Z\subseteq R$ (as in \Cref{lem:z}) or require $R$ to be a planar singularity.} This appendix is not used elsewhere in the paper. 

%\RJ{Put assumption on $R$. For simplicity, should assume that $R$ is over $k$ and the residue is also $k$. We should assume that $R$ is topologically finitely generated over $k$.}
\subsection{Definition and basic properties} Let $k$ be an arbitrary field. 
\begin{definition}\label{def:reducedcurvegerm}
    A reduced curve germ over $k$ is a local order $(R,\mathfrak{m})$ as in Definition~\ref{def:local-order} that is a $k$-algebra, such that $R/\mathfrak{m}$ is a finite dimension $k$-vectors space. %\YHcom{I notice that in some theorem statements below, $\tl R$ is treated as a reduced curve germ. I suggest we clarify that we allow any finite product of reduced curve germs to be applied the same setting, and mention that in any proof one can WLOG work on one factor. } \RJ{This should be fine for multiple components.}
\end{definition}
This is more general than the arithmetic and geometric orders considered in Example~\ref{eg:local-orders}. Following \Cref{sec:lattice}, we denote by $\tl{R}$ the normalization of $R$, and by $K$ the total fraction field of $R$. 
\begin{example}
The ring $\mathbb{C}[[x]]$ is a reduced curve germ over $\mathbb{R}$. This may seem pathological at the first glance, but it naturally arises in practice, for example, as the normalization of the reduced curve germ $R=\mathbb{R}[[x,y]]/(x^2+y^2)$ over $\mathbb{R}$. 
Furthermore, a reduced curve germ may not be geometric reduced. For example, let $k=\mathbb{F}_p(t)$ and $k'=k(t^{\frac{1}{p}})$. Then $R=k[[x,y]]/(y^p-tx^p)$ is a reduced curve germ over $k$, with $\tl R = k'[[x]]$. However, $R\otimes_kk'$ is non-reduced. 
\end{example}

Let $S$ be a $k$-algebra. We write $S \cotimes R$ \resp $S\cotimes K$ as the completion of $S\otimes R$ \resp $S\otimes K$  with respect to the $1\otimes \mathfrak{m}$-adic topology. For example, $S\cotimes k[[T]]=S[[T]]$ and $S\cotimes k((T))=S((T))$.

\begin{definition}\label{def:affineGrass}
An \textbf{$S$-family of rank $d$ $R$-lattices} (or a \textbf{rank $d$ $R$-lattice over $S$}) is a finitely generated $S\cotimes R$-submodule $\mathcal{L}\subseteq S\cotimes K^d$, such that $\mathcal{L}\otimes (S\cotimes K)=S\cotimes K^d$ and $S\cotimes K^d/\mathcal{L}$ is flat over $S$. Let $\Gr_{R}(K^d)$ be the functor over $\textbf{Alg}_k$ sending $S$ to $S$-families of rank $d$ $R$-lattices. When $d$ is clear from the context, we often abbreviate $\Gr_{R}(K^d)$ as $\Gr_R$.
\end{definition}

\begin{remark}\label{rmk:sanwich} Let $R$ be a reduced curve germ over $k$ with $b$ branches. By Cohen's structure theorem, the normalization $\ttilde{R}$ is a reduced curve germ over $k$ that must be of the form 
    \begin{equation*}
        \ttilde{R}=k_1[[T_1]]\times \dots \times k_b[[T_b]], \text{ where } [k_i:k]<\infty.
    \end{equation*}
Write $\underline{T}= T_1+\dots+T_b$. Let $\mathcal{L}$ be an $R$-lattices over a $k$-algebra $S$. Since $\mathcal{L}$ is finitely generated, using the conductor as in Remark~\ref{rmk:lattice-bounded}, it is not hard to show that $$\underline{T}^{N}S\cotimes \ttilde{R}\subeq \mathcal{L}\subeq \underline{T}^{-N} S\cotimes \ttilde{R}$$ for $N\gg 0$. In fact, being finitely generated implies that $\mathcal{L}\subeq \underline{T}^{-N} S\cotimes \ttilde{R}$, and $\mathcal{L}\otimes (S\cotimes K)=S\cotimes K^d$ implies that $\underline{T}^{N}S\cotimes \ttilde{R}\subeq \mathcal{L}$.
\end{remark}
\begin{remark}[Base changes]\label{rmk:basechange}
    Let $\mathcal{L}$ be an $S$-family of rank $d$ $R$-lattices, and let $S\rightarrow S'$ be a morphism. The base change $\mathcal{L}_{S'}$ is the $S'\cotimes R$-module $S'\cotimes_S \mathcal{L}$. We claim that the natural map $\mathcal{L}_{S'}\rightarrow S'\cotimes K^d$ is an injection that makes  $\mathcal{L}_{S'}$ an $S'$-family of rank $d$ $R$-lattices. Following Remark~\ref{rmk:sanwich}, we find $N\gg 0$ so that $\underline{T}^{N}S\cotimes \ttilde{R}\subeq \mathcal{L}\subeq \underline{T}^{-N} S\cotimes \ttilde{R}$. For each $r\in \mathbb{N}^+$, we have the following  exact sequence \begin{equation}\label{eq:secquencwdf}
        0\rightarrow \frac{\mathcal{L}}{\underline{T}^{rN}S\otimes \ttilde{R}}\rightarrow \frac{\underline{T}^{-N}S\otimes \ttilde{R}}{\underline{T}^{rN}S\otimes \ttilde{R}}  \rightarrow  \frac{\underline{T}^{-N} S\otimes \ttilde{R}}{\mathcal{L}}\rightarrow 0
    \end{equation}
The flatness of $S\cotimes K^d/\mathcal{L}$ implies the flatness of $(\underline{T}^{-N} S\otimes \ttilde{R})/\mathcal{L}$, so (\ref{eq:secquencwdf}) remains exact when we tensor it to $S'$. Take inverse limit over $r\rightarrow \infty$ of (\ref{eq:secquencwdf})$\otimes_S S'$ and note that inverse limit is left exact, we find that the natural map $\mathcal{L}_{S'} \rightarrow \underline{T}^{N}S'\cotimes \ttilde{R}$ is an injection. So $\mathcal{L}_{S'}\rightarrow S'\cotimes K^d$ is an injection. It is now easy to check that $\mathcal{L}_{S'}$ is an $S'$-family of $R$-lattices. 
  %  It is useful to note that $\mathcal{L}_{S'}= (\mathcal{L}\otimes_{S\cotimes R} (S\cotimes R))\cotimes_SS' =\mathcal{L}\otimes_{S\cotimes R} (S'\cotimes R)$. Since $S\cotimes K^d/\mathcal{L}$ is flat over $S$, and is exact, the natural map $\mathcal{L}_{S'}\rightarrow S'\cotimes K^d$ is injective. 
 \end{remark}

\begin{remark}\label{rmk:comparetozhu}
    For $R=k[[T]]$ as a reduced curve germ over $k$, our definition of $S$-families of rank $d$ $R$-lattices differs from \cite[Definition~1.1.1]{Zhu_AffineGrass}, which requires projectivity of $\mathcal{L}$ over $S[[T]]$ instead of the flatness of $S((T))^d/\mathcal{L}$ over $S$. However, with some effort, one can show that they are equivalent: ``projectivity implies flatness'' follows from an argument similar to the first half of the proof of \cite[Lemma 1.1.5]{Zhu_AffineGrass}, while the converse follows from the second half.
\end{remark}
%\begin{remark}\label{rmk:freel}Given a finitely generated locally free $S\cotimes R$-module $\mathcal{L}$, there exists a Zariski cover $S'\rightarrow S$ such that $\mathcal{L}_{S'}$ is free over $S'\cotimes R$. To see this, one can assume that $S$ is a local ring. Then $S\cotimes R$ is also a local ring. Therefore $\mathcal{L}_{S'}$ is free. \end{remark}
Recall the definition of ind-scheme from \S\ref{sub:indscheme}. 
\begin{proposition}\label{prop:ind}
$\Gr_R$ is represented by an ind-projective ind-scheme over $k$. 
\end{proposition}
\proof 
The proof is routine. Let $\mathbb{G}r_{\ttilde R, i}$ be the scheme $\bigsqcup_{r\geq 0}\Gr(r, \underline{T}^{-i} \ttilde R/\underline{T}^{i}\ttilde R)$. Let $\Gr_{R,i}:\textbf{Alg}_k\rightarrow \mathbf{Set}$ be the subfunctor of $\mathbb{G}r_{\ttilde R,i}$, defined as \begin{align}\Gr_{R,i}:S\rightarrow \{S\cotimes R\text{-module quotients } \underline{T}^{-i} S\cotimes \ttilde R/\underline{T}^{i}S\cotimes \ttilde{R}\rightarrow Q \text{ that are locally free over $S$}\}.
\end{align}
By Remark~\ref{rmk:sanwich}, an $S$-families of rank $d$ $R$-lattices is sandwiched between $\underline{T}^{i}S\cotimes \ttilde{R}$ and $\underline{T}^{-i}S\cotimes\ttilde{R}$ for some $i\gg 0$. And $S$-families of rank $d$ $R$-lattices
$\mathcal{L}$ such $\underline{T}^{i}S\cotimes \ttilde{R}\subeq \mathcal{L}\subeq \underline{T}^{-i}S\cotimes\ttilde{R}$  correspond bijectively to $S$-points of $\Gr_{R,i}$. Therefore $$\Gr_R=\bigcup_{i} \Gr_{R,i}.$$ 
Let $(\underline{T}^{2i})\cap R$ be the ideal contraction of $(\underline{T}^{2i})$ to $R$. Then $R/((\underline{T}^{2i})\cap R)$ is a $k$-algebra of finite length over $k$. As a result, we can pick a finite set of generators $\{g_1,g_2,...,g_n\}$ of $R/((\underline{T}^{2i})\cap R)$ over $k$. This helps us to re-interpret $\Gr_{R,i}$ as 
    \begin{equation}\label{eq:finitetypepiece}
    \Gr_{R,i}:S\rightarrow \left\{\begin{aligned}
      &S\text{-submodules } M\subeq 
    \underline{T}^{-i} S\cotimes \ttilde R/\underline{T}^{i}S\cotimes \ttilde{R} \text{ with quotients locally free over }S, \\  &\text{and } g_jM\subeq M \text{ for } j=1,2,...,n.
    \end{aligned}
    \right\}.
\end{equation}
Since each $g_j$ imposes a closed condition, we find that $\Gr_{R,i}$ is a closed subscheme of $\mathbb{G}r_{\ttilde R, i}$. In particular, $\Gr_{R,i}$ is projective. This implies that $\Gr_{R,i}\hookrightarrow \Gr_{R,i+1}$ is a closed embedding. Therefore $\Gr_R$ is ind-projective. $\hfill\square$
\iffalse
\begin{corollary}
    \label{lm:defineover}
Let $S$ be a $k$-algebra. An $S$-families of $R$-lattices is defined over a finite type subalgebra.
\end{corollary}
\proof An $S$-families of $R$-lattice is an $S$-point of $\Gr_R$, which factors through some $\Gr_{R,i}$ for $i\gg 0$.  Since $\Gr_{R,i}$ is finite type, it commutes with direct limit. Writing $S$ as a direct union of finite type subalgebras, we see that an $S$-point is factors though some finite type subalgebra of $S$. $\hfill\square$
\fi

\begin{remark}\label{rmk:finitetype}
Let $S$ be a $k$-algebra, and let $\mathcal{L}$ be an $S$-family of rank $d$ $R$-lattices. Since each $\Gr_{R,i}$ is of finite type, we see that $\mathcal{L}$ is already defined over a finite type subalgebra $S'\subseteq S$. 
\end{remark}\begin{remark}
    We will equip $\Gr_R$ with a ``universal lattice'' $\mathcal{U}_R$, which can be thought of as a compatible system of lattices over each $\Gr_{R,i}$.
\end{remark}

\begin{remark}  We state some observations without proof. Let $k=\overline{k}$ (albeit some claims may hold over more general fields).
  \begin{enumerate}
      \item  There is an open ind-subscheme $\Gr_R^\circ\subseteq \Gr_R$ whose $S$-points parameterize families of $R$-lattices that are locally free over $S\cotimes R$. Then $\Gr_R^\circ$ is  naturally compactified by $\Gr_R$. When $\tl R= k[[T_1]]\times ...\times k[[T_b]]$, we have $\Gr_{\tl R}=\Gr_{\tl R}^\circ$.%\YHcom{This works for any branching number, right? Then we can remove the assumption and say we always have this.}
      \item  It is not clear if $\Gr_R$ admits a loop space interpretation. However, $\Gr_R^\circ$ admits a loop space interpretation (at least on the level of $k$-points): $\Gr_R^\circ= \GL_d(K)/\GL_d(R)$.
      \item Similar to the classical case, we have the so called Beauville--Laszlo uniformization. Let $X/k$ be a reduced curve with a singularity $x\in X(k)$, and let $R$ be the germ of $X$ at $x$. Then $k$-points on $\Gr_R(K^d)$ parametrize rank $d$ torsion free sheaves on $X$ trivialized away from $x$.
  \end{enumerate}
    
\end{remark}
\subsubsection{Relation to Quot schemes} Let $S$ be a $k$-algebra. Let $\mathcal{M}\subeq S\cotimes K^d$ be an $S$-family of $R$-lattices. Write $\Gr_{R,S}$ for the base change of $\Gr_R$ to $\Spec S$. Define a subfunctor $\Gr_R(\mathcal{M})\subseteq\Gr_{R,S}: \mathbf{Alg}_S\rightarrow \textbf{Set}$
sending an $S'\in \mathbf{Alg}_{S}$ to $S'$-families of lattices  $\mathcal{L}\subeq\mathcal{M}_{S'}$. Then $\Gr_R(\mathcal{M})$ is represented by an ind-closed subscheme of $\Gr_{R,S}$: in fact,  $\Gr_R(\mathcal{M})=\lim_{i\rightarrow \infty} \Gr_{R,i}(\mathcal{M})$, where each $\Gr_{R,i}(\mathcal{M})$ is a closed subscheme of $\Gr_{R,i}\times_k S$ (cf. \ref{eq:finitetypepiece}) cut out by the condition ``contained in $\mathcal{M}$''. Write $$\Gr_{R}\xleftarrow{p_1} \Gr_{R,S}\xrightarrow{ p_2} S.$$
Then there is a natural grading \begin{equation}\label{eq:grading}
\Gr_{R,i}(\mathcal{M}) =\bigsqcup_{n\geq 0} \Gr_{R,i}^n(\mathcal{M}),\text{ and (by passing to limit)  }\Gr_R(\mathcal{M})=\bigsqcup_{n\geq 0} \Gr_R^n(\mathcal{M}),
\end{equation}
where over each scheme $\Gr_{R,i}^n(\mathcal{M})$ the sheaf $p_2^*\mathcal{M}/p_1^*\mathcal{U}_R$ is flat of rank $n$. 
The ind-scheme $\Gr_R^n(\mathcal{M})$ is essentially a ``Quot scheme''. The following is just a relative version of Definition~\ref{def:indindind}:
\begin{definition}[Relative Quot scheme]
    Let $\Quot_{n}^{R/\mathfrak{m}^r}(\mathcal{M}/\mathfrak{m}^r\mathcal{M}): \mathbf{Alg}_S\rightarrow \textbf{Set}$ be the functor that sends an $S'$ to $S'\otimes_S (S\cotimes R/\mathfrak{m}^r)$-quotients of $S'\otimes_S (\mathcal{M}/\mathfrak{m}^r\mathcal{M})$ which are locally free ${S'}$-modules of rank $n$. It is classically known that $\Quot^R_{n}(\mathcal{M})$ is represented by a projective $S$-scheme, and each map 
$\Quot_{n}^{R/\mathfrak{m}^r}(\mathcal{M}/\mathfrak{m}^r\mathcal{M})\hookrightarrow \Quot_{n}^{R/\mathfrak{m}^{r+1}}(\mathcal{M}/\mathfrak{m}^{r+1}\mathcal{M})$ is a closed immersion. We then define $\Quot_{n}^{R}(\mathcal{M})$ as the ind-scheme $\bigcup_r\Quot_{n}^{R/\mathfrak{m}^r}(\mathcal{M}/\mathfrak{m}^r\mathcal{M})$. 
\end{definition}

\begin{proposition}\label{prop:BGtoQuot}Notation as above. \begin{enumerate}
    \item There is a natural isomorphism $\Quot_{n}^{R}(\mathcal{M})\xrightarrow{\sim} \Gr_{R}^n(\mathcal{M})$. 
\item If $r\geq n$, then $\Gr_{R,r}^n(\mathcal{M})_{\mathrm{red}}=\Gr_{R}^n(\mathcal{M})_{\mathrm{red}}$.
\end{enumerate}
\end{proposition}
\proof Fix $N\gg 0$ so that $\underline{T}^{N}\subeq \m$. Note that $\{\underline{T}^{Nr}\}$ and $\{\m^{r}\}$ form two cofinial systems of ideals in $R$. For $r\geq 1$, define $\Quot_{n}^{R/\underline{T}^{Nr}}(\mathcal{M}/\underline{T}^{Nr}\mathcal{M}): \mathbf{Alg}_S\rightarrow \textbf{Set}$ be the functor that sends an $S'$ to $S'\otimes_S (S\cotimes R/\underline{T}^{Nr})$-quotients of $S'\otimes_S (\mathcal{M}/\underline{T}^{Nr}\mathcal{M})$ which are locally free ${S'}$-modules of rank $n$. Let $S'$ be an $S$-algebra. An $S'$-point of $\Quot_{n}^{R/\underline{T}^{Nr}}(\mathcal{M}/\underline{T}^{Nr}\mathcal{M})$ corresponds to a map $S'\otimes_S (\mathcal{M}/\underline{T}^{Nr}\mathcal{M}) \rightarrow \mathcal{Q}$. Composing this map with the natural quotient map $\mathcal{M}_{S'}\rightarrow S'\otimes_S (\mathcal{M}/\underline{T}^{Nr}\mathcal{M})=\mathcal{M}_{S'}/\underline{T}^{Nr}\mathcal{M}_{S'}$, and taking the kernel, we get an element in $\Gr_{R,Nr}^n(\mathcal{M})$. This defines a natural map $\nu_r:\Quot_{n}^{R/\underline{T}^{Nr}}(\mathcal{M}/\underline{T}^{Nr}\mathcal{M})\rightarrow \Gr_{R,Nr}^n(\mathcal{M})$. We left the readers to check that it is an isomorphism and is compatible with embeddings $r\rightarrow r+1$. Taking limit to get the natural isomorphism as claimed (note that this does not depend on the choice made).

For the second assertion, it suffices to check that for any $S$ algebra  $F$ which is a field, we have $\Gr_{R,n}^n(\mathcal{M})(F)=\Gr_{R}^n(\mathcal{M})(F)$. The proof is routine; cf. \S\ref{sub:indscheme}.
 $\hfill\square$
\subsection{Constructible morphisms between affine Grassmanians}\label{sub:eandi}A key feature of our theory of affine Grassmanians is its flexibility of moving between different base rings. Let $R\subseteq R'\subseteq\tl R$ be reduced curve germs over $k$, then we have a natural finite map $\pi: \Spec R'\rightarrow \Spec R$. Let $S$ be a $k$-algebra, we write $R'_S= S\cotimes R'$ and $R_S= S\cotimes R$. Then the induced map $ \Spec R'_S\rightarrow \Spec R_S$ is again finite. By abuse of notation, we again call it $\pi$. Following the six functor formalism, we have three meaningful operators
\begin{align*}
    \pi_* = \pi_{!}&: \mathrm{Coh}(\Spec R'_S)\rightarrow \mathrm{Coh}(\Spec R_S),\\
    \pi^*,\pi^{!}&: \mathrm{Coh}(\Spec R_S)\rightarrow \mathrm{Coh}(\Spec R'_S).
\end{align*}
We summarize their effects on lattices. Fix an $S$-family of $R$-lattices $\mathcal{L}\subeq V=S\cotimes K^d$:
\begin{enumerate}
    \item There is a natural map \begin{equation}\label{eq:ext}
 \begin{aligned}
        \pi^*\mathcal{L}=R'_S\otimes_{R_S} \mathcal{L}&\rightarrow V,\\
        r'\otimes l& \mapsto r'l,
    \end{aligned}       
    \end{equation} whose image is denoted by $\vec{\pi}^*\mathcal{L}$ (the symbol ``\,$\vec{\cdot}\,\,$'' means ``going to the image'' in our context). This is not a lattice in general. But when $S=k'$ is a field extension of $k$, $\vec{\pi}^*\mathcal{L}$ is a lattice over $k'$, and is in fact the smallest  $R'$-lattice over $k'$ in $V$ containing $\mathcal{L}$. When $S=k'$, we will call $\vec{\pi}^*\mathcal{L}$ the \textbf{lattice extension} of $\mathcal{L}$ to $R'$. %\footnote{Note that in general $\pi^*L\rightarrow V$ is not an embedding, and $\pi^*L$ can be very pathological. For example, one can take $R=k[[T^2,T^3]]$, $R'=k[[T]]$, and $L=k[[T]]$. Then $R'\otimes_{R} L$ is not isomorphic to its image in $k((T))$, which is $k[[T]]$.}. 
    \item There is a natural map 
    \begin{equation}\label{eq:int}
    \begin{aligned}
       \pi^!\mathcal{L}=\Hom_{R_S}(R'_S,\mathcal{L})&\rightarrow V,\\
        f& \mapsto f(1),
    \end{aligned} 
    \end{equation}
whose image is denoted by $\vec{\pi}^!\mathcal{L}$. This is not a lattice in general. But when $S=k'$ is a field extension of $k$, $\vec{\pi}^! \mathcal{L}$ is a lattice over $k'$, and is in fact the largest $R'$-lattice over $k'$ in $V$ contained in $\mathcal{L}$. When $S=k'$, we will call $\vec{\pi}^!\mathcal{L}$ the \textbf{lattice interior} of $\mathcal{L}$ in $R'$. %   is an embedding, due to the fact that $L$ is torsion free over $R$. The image of this embedding is the largest $R'$-lattice of $V$ contained in $L$, and is called the \textbf{lattice interior} of $L$ over $R'$. 
\end{enumerate}
For the opposite direction, let $\mathcal{L}'\subeq V$ be an $S$-family of $R'$-lattices, we have: 
\begin{enumerate}\setcounter{enumi}{2}
\item the pushforward $\pi_*\mathcal{L}'$ is $\mathcal{L}'$ but viewed as an $S$-family of $R$-lattices. 
\end{enumerate}
Let $k'$ be a field extension of $k$, and let $L'$ be an $R'$-lattice over $k'$. The \textbf{extension fiber} of $L'$ is the subset 
$$E_R(L'):=\set{L\in \Gr_{R}(k'):L\subeq_R  L'\text{ and }\vec{\pi}^* L=L'}.
$$ 
Similarly, the \textbf{interior fiber} of $L'$ is $$I_R(L'):=\set{L\in \Gr_{R}(k'): L\supeq_R  L'\text{ and }\vec{\pi}^!L=L'}.$$ 
Though defined set-theoretically,  extension and interior fibers have algebraic structure. We will see in Theorem~\ref{thm:consmaps} that this follow formally from the ``constructibility'' of the maps $$\underline{\vec{\pi}^*},\underline{\vec{\pi}^!}: \Gr_{R}\rightarrow \Gr_{R'}$$ 
that pointwise send an $R$-lattice $L$ over $k'$ to $R'$-lattices $\vec{\pi}^* L$ and $\vec{\pi}^! L$, respectively.

\begin{remark}
 When $R$ is Gorenstein and $L'$ is an $\tl R$-lattice over $k$, $E_R(L')$ and $I_R(L')$ encodes the same amount of information. This is a consequence of the Verdier duality. %This is part of the reason why we only studied the extension fiber in this paper. \YHcom{I don't know what the last sentence is trying to say: I don't know how to repeat the same recipe with interior even for Gorenstein $R$. The last theorem in the appendix shows that they are kind of the same without Gorenstein assumption; but I think the tricky part is that interior fiber can solve $\set{L:L\supeq_R \tl R^d}$ but not $\Quot^R(\tl R^d)$. Perhaps it is true that using the Verdier duality plus the general recipe using extension, we can solve the original Quot problem using interior, and this recipe will be even more nontrivial because duality is involved (like in Alexei's paper with Eric Carlsson}.
\end{remark} 
\subsubsection{Constructible morphisms}
\begin{notation}~\begin{enumerate}
    \item Let $X$ be a $k$-scheme. Recall that a stratification of $X$ is a morphism \begin{equation}\label{eq:strata}\Breve{X}=\bigsqcup_{\alpha\in\mathbf{A}}X_\alpha\rightarrow X
    \end{equation} that is bijective on the underlying sets, such that $|\mathbf{A}|<\infty$, and each $X_\alpha\rightarrow X$ is a locally closed subscheme. A set-theoretical map $f:X\rightarrow Y$ is called a \textbf{constructible morphism}, if there is a stratification (\ref{eq:strata}), such that each $f|_{X_\alpha}:X_\alpha\rightarrow Y$ upgrades to a morphism of schemes. We will denote it by $f:X\xrightarrow{c} Y$ to distinguish it from a general map. We say that $f$ is a \textbf{constructible isomorphism}, if it admits a constructible morphism as an inverse.
    
    In a slightly fancier language, one can localize the category of $k$-schemes by inverting stratifications (\cite[04VC]{stacks-project}), and call the localization category as the category of constructible $k$-schemes. Then a set theoretical map $f$ is a constructible morphism precisely when it underlies a morphism in the localization category.
    \item We can generalize this to the setting of ind-schemes. Let $X$ be an ind-scheme over $k$. Let $\mathbf{A}$ be a countable set. We say that 
$\Breve{X}=\bigsqcup_{\alpha\in \mathbf{A}}X_\alpha\rightarrow X$ is a stratification, if there exist $k$-schemes $\{X_i\}_{i\geq 0}$ and closed immersions $X_i\hookrightarrow X_{i+1}$ such that $X=\varinjlim_{i}X_i$, with the properties that (1) each $X_\alpha$ is a locally closed subscheme of some $X_i$ for $i\gg 0$, (2) for each $i$, $\Breve{X}_i=\bigsqcup_{\alpha\in \mathbf{A}}X_{\alpha}\cap X_i\rightarrow X_i$ is a stratification of the scheme $X_i$ (in particular, there are only finitely many $\alpha\in \mathbf{A}$ such that $|X_\alpha\cap X_i|\neq \varnothing$). A set-theoretical map $f:X\rightarrow Y$ between ind-schemes is called a constructible morphism, denoted by $f:X\xrightarrow{c} Y$, if $X$ admits a stratification $\bigsqcup_{\alpha\in \mathbf{A}}X_{\alpha}\rightarrow X$, such that $f|_{X_\alpha}:X_\alpha\rightarrow Y$ upgrades to a morphism from a scheme into an ind-scheme. The notion of constructible isomorphism can be defined in a similar manner. 
\end{enumerate}
\end{notation}
\begin{lemma}\label{lm:conlattice}
Let $X$ be a finite type $k$-scheme. Let $\mathcal{F}\subseteq \mathcal{O}_X\cotimes K^d$ be a coherent $\mathcal{O}_X\cotimes R$-submodule such that $\mathcal{F}\otimes (\mathcal{O}_X\cotimes K)=\mathcal{O}_X\cotimes K^d$. Then $\mathcal{F}$ is a rank $d$ \textbf{constructible $R$-lattice over $X$}, in the sense that there exists a stratification $\Breve{X}\rightarrow X$ as in (\ref{eq:strata}), such that for each $\alpha\in \mathbf{A}$ the pullback $\mathcal{F}|_{X_\alpha}$ is an $R$-lattice over $X_\alpha$. In particular, it induces a constructible morphism $f:X\xrightarrow{c} \Gr_R$ that comes from a morphism $\Breve{f}: \Breve{X}\rightarrow \Gr_R$, such that $\Breve{f}^*\mathcal{U}_R=\mathcal{F}|_{\Breve{X}}$, where $\mathcal{U}_R$ is the universal lattice over $\Gr_R$.
\end{lemma}
\proof Since $X$ is quasi-compact, we can reduce to the case where $X=\Spec     S$ is affine. Then $\mathcal{F}$ is finitely generated over $S$, so there is an $N\gg 0$ such that $\underline{T}^{N}\mathcal{O}_X\cotimes \ttilde{R}\subeq \mathcal{F}\subeq \underline{T}^{-N}\mathcal{O}_X\cotimes\ttilde{R}$. The quotient $(\underline{T}^{-N}\mathcal{O}_X\cotimes\ttilde{R})/\mathcal{F}$ is a coherent $\mathcal{O}_X$-module. Using flattening ideals (cf. \cite[tag 05P9]{stacks-project}), there is a morphism $$\pi:\Breve{X}=\bigsqcup_{\alpha\in \mathbb{N}} X_\alpha\rightarrow X$$ inducing bijection on the underlying sets, such that each $X_\alpha\rightarrow X$ is a locally closed subscheme, and the pullback of $(\underline{T}^{-N}\mathcal{O}_X\cotimes\ttilde{R})/\mathcal{F}$ to each $X_\alpha$ is flat of rank $\alpha$. Note that $X_\alpha=\varnothing$ when $\alpha\gg 0$, since the rank of $(\underline{T}^{-N}\mathcal{O}_X\cotimes\ttilde{R})/\mathcal{F}$ over closed points are bounded above by the rank of the free $\mathcal{O}_X$-module $(\underline{T}^{-N}\mathcal{O}_X\cotimes\ttilde{R})/(\underline{T}^{N}\mathcal{O}_X\cotimes\ttilde{R})$. Therefore $\pi$ is a stratification, and each $\mathcal{F}|_{X_\alpha}$ is a rank $d$ lattice. So we get a morphism $f_\alpha:X_\alpha\rightarrow \Gr_R$ such that the pullback of $f^*_\alpha\mathcal{U}$ is $\mathcal{F}|_{X_\alpha}$. $\hfill\square$
\subsubsection{The main structural result}
\begin{theorem}\label{thm:consmaps} Let the set up be as in the beginning of \S\ref{sub:eandi}. Consider the following three (set-theoretical) maps: \begin{align*}
    \underline{\pi_*}:   \Gr_{R'} &\rightarrow \Gr_{R},\\
    \underline{\vec{\pi}^*}:  \Gr_{R}&\rightarrow \Gr_{R'},\\
    \underline{\vec{\pi}^!}: \Gr_{R}&\rightarrow \Gr_{R'},
\end{align*}
such that for any field valued points $L\in \abs{\Gr_R}$ and $L'\in \abs{\Gr_{R'}}$, we have $\underline{\vec{\pi}^*}(L)= \vec{\pi}^* L$, $\underline{\vec{\pi}^!}(L)= \vec{\pi}^! L$, and $\underline{\pi_*}(L')=\pi_* L'$. Then $\underline{\pi_*}$ is a genuine morphism, while $\underline{\vec{\pi}^*}$ and $\underline{\vec{\pi}^!}$
are constructible morphisms. In particular, if $L'$ corresponds to a closed $k'$-point of $\Gr_{R'}$, then $E_R(L')$ and $I_R(L')$ are $k'$-points of the constructible subsets $\underline{\vec{\pi}^*}^{-1}(L')$ and $\underline{\vec{\pi}^!}^{-1}(L')$.   
\end{theorem}
\begin{proof}
 Let $\mathcal{U}_R$ (resp. $\mathcal{U}_{R'}$) be the universal lattice over $\Gr_R$ (resp. $\Gr_{R'}$).
The fact that $\underline{\pi_*}$ is a morphism is relatively easy: $\pi_*\mathcal{U}_{R'}$ is an $R$-lattice over $\Gr_R$ (understood as a compatible system of $R$-lattices over $\Gr_{R,i}$ in (\ref{eq:finitetypepiece})), hence induces the desired morphism $\underline{\pi_*}:\Gr_{R'} \rightarrow \Gr_{R}$. 

In the following we study the constructibility of the other two maps.  Write $\Gr_R=\varinjlim X_i$, where $X_i=\Gr_{R,i}$ as in (\ref{eq:finitetypepiece}). Let $X=\Spec S$ be a locally closed affine subscheme of some $X_i, i\gg 0$. It suffices to show that $\underline{\vec{\pi}^*}$ and $\underline{\vec{\pi}^!}$ are constructible when restricted to $X$. Let $\mathcal{U}_{R,S}\subeq S\cotimes K^d$ be its restriction to $X$. Since $\vec{\pi}^*\mathcal{U}_{R,S}$ and $\vec{\pi}^!\mathcal{U}_{R,S}$ are coherent $S\cotimes R'$-submodules of $S\cotimes K^d$ that are sandwiched between $\underline{T}^{N}S\cotimes\ttilde{R}$ and $\underline{T}^{-N}S\cotimes\ttilde{R}$ for some $N\gg 0$, they 
satisfy the condition of Lemma~\ref{lm:conlattice}. We then apply Lemma~\ref{lm:conlattice} to show that they are constructible $R'$-lattices over $X$, and obtain two constructible morphisms $c_1,c_2: X\xrightarrow{c}\Gr_{R'}$.

We then check that the underlying set theoretic maps of the two constructible morphisms $c_1,c_2$ coincide with $\underline{\vec{\pi}^*}$ and $\underline{\vec{\pi}^!}$. To do this, we can replace $X$ by a stratum so to assume that $ \vec{\pi}^* \mathcal{U}_{R,S}$ and $ \vec{\pi}^! \mathcal{U}_{R,S}$ are $S$-families of $R'$-lattices. It suffices to check that if $S\rightarrow F$ is a morphism in to a field $F$, then 
\begin{align}
    \label{eq:const1}\vec{\pi}^* \mathcal{U}_{R,F}=F \cotimes_S( \vec{\pi}^* \mathcal{U}_{R,S}), \\
    \label{eq:const2} \vec{\pi}^! \mathcal{U}_{R,F}=F \cotimes_S( \vec{\pi}^! \mathcal{U}_{R,S})
\end{align}
as submodules of $F\cotimes K^d$ (see \ref{rmk:basechange} for why RHS is contained in $F\cotimes K^d$). Since $\coker (R'_S\otimes_{R_S} \mathcal{U}_{R,S} \rightarrow S\cotimes K^d)= S\cotimes K^d/(\vec{\pi}^* \mathcal{U}_{R,S})$ is flat over $S$, we see that \footnote{In the following computation, since $R_S,R_F, R'_S, R'_F$ are all Noetherian, and since $\mathcal{U}_{R,S}$ and other modules are finitely generated over one of these rings (except the ones involving $K^d$, which can be handled via truncating to $T^{-N}\tl R$ for $N\gg 0$), taking $1\otimes \mathfrak{m}$-adic completion is exact in our context. So we can treat $\cotimes$ in the same way as treating $\otimes$.}\begin{equation}\label{A7}
    \begin{aligned}
   F \cotimes_S( \vec{\pi}^* \mathcal{U}_{R,S})&= F\cotimes_S \im (R'_S\otimes_{R_S} \mathcal{U}_{R,S} \rightarrow S\cotimes K^d)\\&= F\cotimes_S \ker(S\cotimes K^d\rightarrow \coker (R'_S\otimes_{R_S} \mathcal{U}_{R,S} \rightarrow S\cotimes K^d))\\
    &=\ker(F\cotimes K^d\rightarrow F\cotimes_S \coker (R'_S\otimes_{R_S} \mathcal{U}_{R,S} \rightarrow S\cotimes K^d))\\
    &=\ker(F\cotimes K^d\rightarrow  \coker (R'_F\otimes_{R_F} \mathcal{U}_{R,F} \rightarrow F\cotimes K^d))\\
    &=\im (R'_F\otimes_{R_F} \mathcal{U}_{R,F} \rightarrow F\cotimes K^d)\\
    &=\vec{\pi}^* \mathcal{U}_{R,F}. 
\end{aligned} 
\end{equation}
This proves (\ref{eq:const1}). To show (\ref{eq:const2}), first note that we have a commutative square \begin{center}
\begin{tikzcd}
{F\cotimes_S\Hom_{R_S}(R_S',\mathcal{U}_{R,S})} \arrow[d, "\alpha"'] \arrow[r] & F\cotimes_S(S\cotimes K^d) \arrow[d, "\simeq"] \\
{\Hom_{R_F}(R_F',\mathcal{U}_{R,F})} \arrow[r]                                 & F\cotimes_SK^d                          
\end{tikzcd}
\end{center}
The image of the bottom row is $\vec{\pi}^! \mathcal{U}_{R,F}$. Using the fact that $\coker (\Hom_{R_S}(R_S',\mathcal{U}_{R,S}) \rightarrow S\cotimes K^d)= S\cotimes K^d/(\vec{\pi}^! \mathcal{U}_{R,S})$ is flat over $S$, we can argue 
as in (\ref{A7}) that the image of the top row is $ F \cotimes_S( \vec{\pi}^! \mathcal{U}_{R,S})$. So to  
prove (\ref{eq:const2}), it suffices to show that $\alpha$ is surjective. Since $R_S$ is Noetherian,  $R'_S$ is finitely presented. Therefore we have a right exact sequence \begin{equation}\label{eq:resolution1}
    R_S^{\oplus m}\rightarrow R_S^{\oplus n} \rightarrow R'_S\rightarrow 0.
\end{equation}
This yields a commutative diagram 
\begin{center}
\begin{tikzcd}
            & {F\cotimes_S\Hom_{R_S}(R_S',\mathcal{U}_{R,S})} \arrow[r] \arrow[d, "\alpha"'] & {F\cotimes_S\mathcal{U}_{R,S}^{\oplus n}} \arrow[r, "\beta_1"] \arrow[d, "\simeq"] & {F\cotimes_S\mathcal{U}_{R,S}^{\oplus m}} \arrow[d, "\simeq"] \\
0 \arrow[r] & {\Hom_{R_F}(R_F',\mathcal{U}_{R,F})} \arrow[r]                                 & {\mathcal{U}_{R,F}^{\oplus n}} \arrow[r, "\beta_2"]                                & {\mathcal{U}_{R,F}^{\oplus m}}                               
\end{tikzcd}
\end{center}
where the first row follows by applying $F\cotimes_S\Hom_{R_S}(-,\mathcal{U}_{R,S})$ to (\ref{eq:resolution1}), and the second row by applying $\Hom(F\cotimes_S -,\mathcal{U}_{R,F})$ to (\ref{eq:resolution1}). Now we can natually identify $\beta_1$ with $\beta_2$. In particular, $\im \beta_1$ is identified with $\im \beta_2$. By the snake lemma, we see that $\coker \alpha \simeq \ker (\im \beta_1\rightarrow \im \beta_2)= 0$. This shows $\alpha$ is surjective, and (\ref{eq:const2}) follows. 
\end{proof}
Observe that $$\underline{\vec{\pi}^*}\circ \underline{\pi_*}=\underline{\vec{\pi}^!}\circ \underline{\pi_*}=\mathrm{id},$$ which can be easily checked over field valued points. It follows that $\underline{\pi_*}$ is injective, while $\underline{\vec{\pi}^*}$ and $\underline{\vec{\pi}^!}$ are surjective. This makes $\underline{\vec{\pi}^*}$ and $\underline{\vec{\pi}^!}$ resemble a certain ``fiberation'' that admits $\underline{\pi_*}$ as a section. Our last theorem is a result in this direction:

Let $k=\overline{k}$ and $R'=\tl{R}$, and let $\tl L$ be an $\tl R$-lattice over $k$. Identify $E_R(\tl L)$ and $I_R(\tl L)$ with their underlying constructible subsets.
\begin{theorem}
Notation as above. We have constructible isomorphisms \begin{align*}
        &\Gr_{R}\stackrel{c}{\simeq} \Gr_{\tl{R}}\times E_R(\tl L),\\
       & \Gr_{R}\stackrel{c}{\simeq} \Gr_{\tl{R}}\times I_R(\tl L).
    \end{align*}
\end{theorem}
\begin{proof}
Note that $E_R(\tl L)$ and $I_R(\tl L)$ do not depend on $\tl L$. Without loss of generality, we can take $\tl L = \tl R^{d}$. Stratify $\Gr_{\tl{R}}$ into Schubert cells as in \S\ref{subsec:schubert}. Then the universal lattice $\mathcal{U}_{\tl R}$ over each $X_\mu^{\circ}$ is \textit{constant}, i.e., $\mathcal{U}_{\tl R}$ is isomorphic to the pullback of $\tl{R}^{d}$ along the structure morphism $X_\mu^{\circ}\rightarrow \mathrm{Spec\,}k$, and this isomorphism is given tautologically and explicitly by the affine parameterization (\ref{eq:affineparameterization}). It follows immediately that the constructible morphisms $ \underline{\vec{\pi}^*}, \underline{\vec{\pi}^!}: \Gr_{R}\rightarrow \Gr_{\tl{R}}$ as per Theorem~\ref{thm:consmaps} induce isomorphisms $ \underline{\vec{\pi}^*}^{-1}(X_\mu^{\circ})\simeq X_\mu^{\circ}\times E_R(\tl L)$ and $ \underline{\vec{\pi}^!}^{-1}(X_\mu^{\circ})\simeq X_\mu^{\circ}\times I_R(\tl L)$. Putting the strata together, we win. 
\end{proof}

% \bibliographystyle{abbrv}
% \bibliography{bib}
\bibliography{submission.bbl}
   
\end{document}